\RequirePackage{fix-cm}

\documentclass[smallextended]{svjour3}  
\usepackage{amsfonts}
\usepackage{graphicx}
\usepackage{amssymb}
\usepackage{amsmath}
\usepackage{amscd}
\usepackage[all]{xy}
\usepackage{dsfont}
\usepackage{hyperref}

\newtheorem{thm}[subsubsection]{Theorem}
\newtheorem{cor}[subsubsection]{Corollary}
\newtheorem{lem}[subsubsection]{Lemma}

\newtheorem{prop}[subsubsection]{Proposition}
\newtheorem{conj}[subsection]{Conjecture}
\newtheorem{defn}[subsubsection]{Definition}

\numberwithin{equation}{subsection}

\newcommand{\set}[1]{\left\{#1\right\}}

\newcommand{\V}{\mathbb{V}_\lambda^\vee(\mathbb{C})}
\makeatletter

\newcommand{\Rmnum}[1]{\expandafter\@slowromancap\romannumeral #1@}
\makeatother

\begin{document}

%%%%% ------------- fill in your data below this line  -------------------
%%%%%    The following lines \Title ... \EndAddress must ALL be present
%%%%%    and in the given order.
\title{Twisted eigenvarieties and self-dual representations}
\subtitle{}

%\titlerunning{Short form of title}        % if too long for running head

\author{Zhengyu Xiang}

%\authorrunning{Short form of author list} % if too long for running head

\institute{Zhengyu Xiang \at 
              SCMS and Fudan University, \\
              Tel.: +86-21-55665634\\
              \email{fauthor@example.com}          }

\date{Received: date / Accepted: date}
% The correct dates will be entered by the editor

\maketitle

\begin{abstract}
For a reductive group $G$ and a finite order Cartan-type automorphism $\iota$ of $G$, we construct an eigenvariety parameterizing $\iota$-invariant cuspidal Hecke eigensystems of $G$. In particular, for $G=Gl_n$, we prove, any self-dual cuspidal Hecke eigensystem can be deformed in a p-adic family of self-dual cuspidal Hecke eigensystems containing a Zariski dense subset of classical points. 
% \PACS{PACS code1 \and PACS code2 \and more}
% \subclass{MSC code1 \and MSC code2 \and more}
\end{abstract}

\section{INTRODUCTION}

Consider $G$ a reductive group over $\mathbb{Q}$. Let $S_G(K_f)$ be the locally symmetric space associated to $G$ and a neat open subgroup $K_f$ of the finite adelic points of $G$. Let $T$ be a maximal torus of $G$ and $\lambda$ a regular dominant algebraic weight of $G$ with respect to $T$. Consider $\mathbb{V}_\lambda$, the finite dimensional irreducible algebraic representation of $G$ with highest weight $\lambda$, and its dual $\mathbb{V}_\lambda^\vee$. There is a standard action of the Hecke algebra $\mathcal{H}_G$ on the cohomology spaces $H^*(S_G(K_f), \mathbb{V}^\vee_\lambda(\mathbb{C}))$. Those automorphic representations which can be realized in $H^*(S_G(K_f),\mathbb{V}^\vee_\lambda(\mathbb{C}))$ are said of level $K_f$ and cohomological weight $\lambda$.

\vspace{1pc}
Let's fix a prime $p$ and an embedding $i_p: \overline{\mathbb{Q}}_p\hookrightarrow\mathbb{C}$. Ones are interested in the behavior of automorphic representations when their weights varying $p$-adically. This leads to the study of $p$-adic automorphic representations. For simplicity, assume $G$ splits over $\mathbb{Q}_p$. Let $B$ be a Borel subgroup of $G_{/\mathbb{Q}_p}$ containing $T$, consider the situation $K_f=K^pI_m$. Here $K^p\subset G(\mathbb{A}_f^p)$ and $I_m$ is an Iwahori subgroup of $G(\mathbb{Q}_p)$ in good position with respect to the pair $(B,T)$. Let $\mathcal{H}_p$ be the $p$-adic Hecke algebras of $G$ under this setting (see \S2.1). If $\pi$ is a finite slope automorphic representation of $G$ of algebraic cohomological weight $\lambda^{alg}$, its $p$-stabilizations are irreducible representations of $\mathcal{H}_p$ that can be realized in the cohomology space $H^*(S_G(K_f), \mathcal{V}^\vee_{\lambda}(\overline{\mathbb{Q}}_p))$ (refer \cite[\S4.1.9]{Urban}). Here $\lambda=\lambda^{alg}\epsilon$ is a $p$-adic arithmetic weight obtained by twisting $\lambda^{alg}$ with some finite order character $\epsilon$ of $T(\mathbb{Z}_p)$, and $\mathcal{V}_{\lambda}$ is the locally algebraic induced representation of a $p$-adic cell of $G(\mathbb{Q}_p)$ from $\lambda$ (see \S2.3). Those representations of $\mathcal{H}_p$ from $p$-stabilization are most important examples of $p$-adic automorphic representations, and are called classical. If we further remove the ``bad" places form $\mathcal{H}_p$, we obtain a commutative algebra $R_{\mathcal{S},p}$, which can be identified in the center of $\mathcal{H}_p$. Here $\mathcal{S}$ is the finite subset of ``bad" places defined in \S2.1. The central character of a classical $p$-adic automorphic representation defines a character of $R_{\mathcal{S},p}$ appearing in $H^*(S_G(K_f), \mathcal{V}^\vee_\lambda(\overline{\mathbb{Q}}_p))$ for some arithmetic $p$-adic weight $\lambda$. It is called a $p$-adic arithmetic Hecke eigensystem of weight $\lambda$. 

\vspace{1pc}
So ones are interested in interpolating the arithmetic Hecke eigensystems for weight $\lambda$ over the $p$-adic weight space $\mathfrak{X}$. To do this, Ash and Stevens developped the notion of ``overconvergent" cohomology, which played the role of ``overconvergent modular forms" in the classical theory of $p$-adic modular forms (\cite{Ash-Stevens}). Concretely speaking, for a $p$-adic weight $\lambda\in\mathfrak{X}(\overline{\mathbb{Q}}_p)$, one can construct a distribution space $\mathcal{D}_\lambda$, on which the $U_p$ operators acts as compact operaters (see \S2.3). This gives an action of $\mathcal{H}_p$ on the ``overconvergent " cohomology spaces $H^*(S_G(K_f), \mathcal{D}_\lambda)$. We call an irreducible representation of $\mathcal{H}_p$ (\emph{resp.} a character of $R_{\mathcal{S},p}$) appearing in $H^*(S_G(K_f), \mathcal{D}_\lambda(\overline{\mathbb{Q}}_p))$ a $p$-adic overconvergent automorphic representation (\emph{resp.} Hecke eigensystem). It is proved in \cite{Ash-Stevens}, \cite[Theorem 5.4.4]{Urban} and \cite[Corollary 8.6]{Xiang} that every finite slope arithmetic cuspidal Hecke eigensystem $\theta$ can be deformed into a $p$-adic analytic family of finite slope overconvergent cuspidal Hecke eigensystems. This result is a consequence of the existence of a geometric object named ``eigenvariety" after Coleman-Mazur's work of ``eigencurves". An eigenvariety for group $G$ is a rigid analytic space whose points parametrize overconvergent Hecke eigensystems. A large part of the works \cite{Ash-Stevens}, \cite{Urban} and \cite{Xiang} mentioned above are devoted to the construction of eigenvarieties for different groups.

\vspace{1pc}
There are two motivations for this paper. The first one is about the arithmeticity of a family of overconvergent Hecke eigensystems, that is, if a $p$-adic family obtained above contains enough arithmetic Hecke eigensystems. In the language of eigenvariety, one asks if $\theta$ is lying in an irreducible component of an eigenvareity containing a Zariski dense subset of arithmetic points (such a component is called arithmetic. If $\theta$ is \emph{not} in any arithmetic component, it is called arithmetically rigid). In \cite{APG}, Ash, Pollack and Stevens show that the answer is not always positive, in particular, for $G=Gl_3$, they make the next conjecture:

\begin{conj}[Ash-Pollack-Stevens]\label{APS}
Let $\theta$ be a finite slope cuspidal Hecke eigensystem of $Gl_3.$ If $\theta$ is not arithmetically rigid, then $\theta$ is essentially self-dual.
\end{conj}

In this paper, we obtain the inverse of its statement for $Gl_n$:

\begin{thm}
Every essentially self-dual finite slope cuspidal Hecke eigensystem of $Gl_n$ is not arithmetically rigid.
\end{thm}

We actually work on a more general situation. Let $\iota$ be Cartan-type automorphism of $G$ such that $\iota$ stabilizes $(B,T)$, and consider the $\iota$-invariant automorphic representations (\emph{resp.} overconvergent representations, Hecke eigensystems, etc.). Let $\mathfrak{X}^\iota$ be the subspace of $\mathfrak{X}$ consisting of $\iota$-invariant $p$-adic weights (see \S2.2), to study the family of $\iota$-invariant Hecke eigensystems with weights varying in $\mathfrak{X}^\iota$, we construct twisted eigenvarieties over $\mathfrak{X}^\iota$ parametrizing $\iota$-invariant finite slope overconvergent Hecke eigensystems (see \S6):

\begin{thm}[twisted eigenvarities]
There is an eigenvariety $\mathfrak{G}_{K^p}^\iota$ parameterizing $\iota$-invariant finite slope overconvergent Hecke eigensystems of $G$. Every point $y\in\mathfrak{G}_{K^p}^\iota(\overline{\mathbb{Q}}_p)$ can be viewed as a pair $(\lambda, \theta)$, where $\theta$ is a $\iota$-invariant finite slope overconvergent Hecke eigensystem of weight $\lambda\in\mathfrak{X}^\iota(\overline{\mathbb{Q}}_p)$. There is a subvariety $\mathfrak{E}_{K^p}^\iota$ of $\mathfrak{G}^\iota_{K^p}$, satisfying:
 \begin{itemize}
\item[$(a)$] For any arithmetic $(\lambda,\theta)\in\mathfrak{G}^\iota_{K^p}(\overline{\mathbb{Q}}_p)$, $(\lambda,\theta)$ is in $\mathfrak{E}^\iota_{K^p}(\overline{\mathbb{Q}}_p)$ if and only if $\theta$ is cuspidal and has a non-trivial $\iota$-twisted Euler-Poincare characteristic.
\item[$(b)$] Every irreducible component of $\mathfrak{E}^\iota_{K^p}$ is arithmetic, equipped with a projection onto a Zariski dense subset of $\mathfrak{X}^\iota$.
\item[$(c)$] $\mathfrak{E}_{K^p}^\iota$ is equidimensional with the same dimension to $\mathfrak{X}^\iota$.
\end{itemize}
\end{thm}

In particular, if $G=Gl_n$ and $\iota$ is a Cartan-type involution, the notion of $\iota$-invariant is same to self-dual. Then the twisted Euler-Poincare characteristic of an (essentially) self-dual Hecke eigensystem is always non-trivial. So $\mathfrak{E}_{K^p}^\iota$ parameterizes all self-dual finite slope cuspidal Hecke eigensystems (see \S7). In case that $n=3$, this also gives some hint for Ash-Pollack-Stevens' conjecture, as in Theorem 7.3.2 and Remark 7.3.3 below.

\vspace{1pc}
Our second motivation is to develop a twisted version of Urban's theory of finite slope character distribution \cite[\S4.5]{Urban}. A finite slope character distribution is a morphism $J: \mathcal{H}_p\rightarrow \overline{\mathbb{Q}}_p$ which is a linear combination of the traces of finite slope overconvergent representations. Urban proves that, there is an eigenvariety associated to every analytic family of effective finite slope character distributions, \cite[\S5]{Urban}. This eigenvariety parameterizes the finite slope overconvergent Hecke eigensystems appearing in the character distributions. However, Urban's theory excludes many interesting cases, like $Gl_n$ with $n>2$. The reason is, the coefficients of Urban's distributions are essentially given by the Euler-Poincare characteristics. So for a group $G$ such that $G(\mathbb{R})$ does not satisfy the Harish-Chandra condition, they are trivial. To avoids this issue, we construct the notion of ``twisted" finite slope character distributions (see \S5). Concretely, we construct a distribution which is a linear combination of the twisted traces of $\iota$-invariant finite slope overconvergent representations. We show this distribution has similar property as Urban's character distributions and it gives a construction of the twisted eigenvariety $\mathfrak{E}^\iota$ in the theorem above (see \S6).

\vspace{1pc}
In practice, there are two new difficulties occur since the involving of twisting action by $\iota$. The first one is the lack of a twisted version of Franke's trace formula as in \cite[Theorem 1.4.2]{Urban}, which plays an essential role to cut out the cuspidal representations from the whole cohomology. To cure this, we have to go through Franke's theory of Eisenstein spectral sequence (\cite{Franke}), and study how $\iota$ acting on each step of Franke's theory carefully. This is done in \S4 and we proves a twisted version of Franke's trace formula there (Theorem 4.3.10). The second difficulty appears during the construction of the twisted eigenvariety. Since we consider the twisted trace instead, locally our twisted distributions are no longer pseudo-representations as in \cite[\S5.3.1]{Urban}, so we do not have the ``second construction" as Urban did (\cite[\S5.3]{Urban}). We bypass this difficulty by borrowing the construction of the full eigenvariety in \cite{Xiang} to construct a ``bigger twisted eigenvariety" first and then working in this bigger space. This is done in \S6.3.

\vspace{1pc}
One can view Urban's finite slope character distribution as a $p$-adic analogue to Selberg's trace formula, then our theory gives an analogue to the twisted trace formula. In \cite[\S6]{Urban}, Urban gives a simplified geometric expansion for his distribution following the work of Franke \cite{Franke} and Arthur \cite{Arthur}. However, a complete expansion as \cite[(3)]{Arthur} can also be given. In a consequent paper \cite{Xiang2}, we will establish a corresponding geometric expansion for our twisted distributions as well. One can then expect a $p$-adic family version comparison between them as Arthur-Clozel's theory in \cite{AC}. This comparison will give a relation between eigenvarieties.

\subsection{Acknowledgment}
I'd like to thank Professor Eric Urban here, the base of this work on his paper \cite{Urban} is obvious. Without his help this paper will not exist.

\vspace{1pc}
\section{PRELIMINARY}
\subsection{notation}
Throughout this paper, we fix $p$ a rational prime number and an identification $\hat{\overline{Q}}_p\cong\mathbb{C}$. Let $\mathbb{A}=\mathbb{A}_{\mathbb{Q}}$ be the adelic ring of $\mathbb{Q}$, $\mathbb{A}_\infty$ and $\mathbb{A}_f$ its archimedean and finite part respectively. For any algebraic group $H$ over $\mathbb{Q}$, put $H_\infty=H(\mathbb{A}_\infty)$ and $H_f=H(\mathbb{A}_f)$. We also denote by $H(\mathbb{A})^1\subset H(\mathbb{A})$ the subgroup of all $h\in H(\mathbb{A})$ with $\prod_{v}|\xi(h)|_v=1$ for all characters of $H$ defined over $\mathbb{Q}$, where the product is running over all places of $\mathbb{Q}$.

\vspace{1pc} 
Let $G$ be a quasi-split\footnote{This assumption is not necessary but for the convenience of discussion only. Otherwise, one has to use the notation as in \cite[\S1.3.1]{Urban}.} reductive group over $\mathbb{Q}$, denote by $Z=Z_G$ its center. Let $K_{\infty}$ be a fixed maximal compact subgroup of $G_{\infty}$, and fix a good maximal compact subgroup $\mathbb{K}\subset G(\mathbb{A})$ whose archimedean component is $K_\infty$. For every prime number $l$, denote by $K_l$ an open compact subgroup of $G(\mathbb{Q}_l)$. Put $K_f=\prod_{l}K_l$  such that for almost all $l\neq p$, $K_l$ to be maximal. Denote by $K^p=K_f^p=\prod_{l\neq p}K_l$ and $K=K_{\infty}K_{f}$. Consider the locally symmetric space of $G$ associated to $K_f$:
\begin{equation}S_G(K_f):=G(\mathbb{Q})\backslash G(\mathbb{A})/KZ_{\infty}.\end{equation}
Properly choose a finite set of representatives $\{g_i\}_i$ in $G(\mathbb{A})$ such that
\begin{equation}G(\mathbb{A})=\bigsqcup_{i}G(\mathbb{Q})\times
G_{\infty}^{+}\times g_iK_f,\end{equation}
where $G_\infty^+$ is the identity component of $G_\infty$. We then have 
\begin{equation}S_G(K_f)\cong\bigsqcup_i\Gamma_i\setminus\mathcal {H}_G,\end{equation}
where $\Gamma_i = \Gamma(g_i, K)$ is the image of $g_iKg_i^{-1}\cap G(\mathbb{Q})^{+}$ in $G^{ad}(\mathbb{Q})$ and $\mathcal{H}_G =G^{+}_{\infty}/K_{\infty}Z_{\infty}$. We further assume $K$ is neat (that is, $\Gamma_i$ contains no element of finite order), then $S_G(K_f)$ is a smooth real analytic variety of a finite dimension, say, $d$. We also write
\begin{equation}
S_G:=G(\mathbb{Q})\backslash G(\mathbb{A})/K_\infty Z_{\infty}
\end{equation}

\vspace{1pc}
Let $T$ be a maximal torus of $G$ and $B$ a Borel subgroup of $G$ containing $T$. Let $N$ be the unipotent radical of $B$, and $N^-$ its opposite. At $p$, we fix a Iwahori subgroup $I$ of $G(\mathbb{Q}_p)$ with respect to $B$, this means that we have fixed compatible integral models $\mathcal{G, B, T, N, N^-}$ for $G, B, T, N, N^-$ over $\mathbb{Z}_p$ (according to a fixed chamber $C_I$ of the Bruhat-Tits building $\mathcal{BL}$ of $G_{\mathbb{Q}_p}$), such that $I=I_1$, where for any integer $m\geq1$, \begin{equation}I_m = \{g\in \mathcal{G}(\mathbb{Z}_p)|g\in
\mathcal B(\mathbb{Z}/p^m\mathbb{Z})\ mod\ p^m\}.\end{equation} 
By Iwahori decomposition,
\begin{equation}\label{Iwahori}I_m=(I_m\cap  N^-(\mathbb{Q}_p))\mathcal T(\mathbb{Z}_p)\mathcal N(\mathbb{Z}_p).\end{equation}
We normalize the Haar measure on $G(\mathbb{Q}_p)$ such that the measure of $I$ is $1$. Once fixing the Iwahori level at $p$, we write
\begin{equation}
\tilde{S}_{G,m}:=G(\mathbb{Q})\backslash G(\mathbb{A})/K_\infty I_m Z_{\infty}.
\end{equation}

Now put
\begin{equation}\label{T+}
T^+:=\{t\in T(\mathbb{Q}_p)|t\mathcal N(\mathbb{Z}_p)t^{-1}\subset \mathcal N(\mathbb{Z}_p)\}\end{equation}
\begin{equation}\label{T++}
T^{++}:=\{t\in T^+|\bigcap_{i\geq 1} t^i\mathcal N(\mathbb{Z}_p)t^{-i}=\{1\}\},\end{equation}
\begin{equation}\Delta^+_m := I_mT^+I_m,\  \Delta^{++}_m: =I_mT^{++}I_m,\end{equation} and consider the $p$-adic cells
\begin{equation}\Omega_m={I_m\cap N^-(\mathbb{Q}_p)}\backslash I_m \subseteq N^-(\mathbb{Q}_p)\backslash G(\mathbb{Q}_p).\end{equation}

For any $g\in \Delta^+_m$, write $g=n_g^-t_gn_g^+$ by Iwahori decomposition, then the $*$-action of $\Delta_m^+$ on $\Omega_m$ is defined as follow (see \cite[\S3.1.3]{Urban} and \cite[\S5.2]{Ash-Stevens}): Fixing a splitting $\xi$ of the exact sequence \begin{equation}1\rightarrow \mathcal T(\mathbb{Z}_p)\rightarrow T(\mathbb{Q}_p)\rightarrow T(\mathbb{Q}_p)/\mathcal T(\mathbb{Z}_p)\rightarrow 1,\end{equation}
for any $[x]\in \Omega$, define
\begin{equation}\label{diff}[x]*g= [\xi(t_g)^{-1}xg].\end{equation}
As in \cite[\S3.1.2 (11)]{Urban}, we choose $\xi$ so that for any algebraic character $\lambda^{alg}\in X^*(T)$ and $t\in T(\mathbb{Q}_p)$
\begin{equation}\label{split}
\lambda^{alg}(\xi(t))=|\lambda^{alg}(t)|_p^{-1}.
\end{equation}

\vspace{1pc}
The Atkin-Lehner algebra of $G$ at $p$ is defined by:
\begin{equation}\mathcal{U}_p=C_c^{\infty}(\Delta_m^+//I_m, \mathbb{Z}_p) \simeq \mathbb{Z}_p[T^+/\mathcal T(\mathbb{Z}_p)],\end{equation}
it does not depend on $m$. We then define the global $p$-adic Hecke algebras:
\begin{equation}\mathcal{H}_p:=\mathcal{H}_p(G)=C_c^{\infty}(G(\mathbb{A}_f^p))\hat{\otimes}\mathcal{U}_p,\end{equation}
and for any open compact subgroup $K^p$ of $G(\mathbb{A}_f^p)$, define its subalgebra of $K^p$-bi-invariant functions by:
\begin{equation}\mathcal{H}_p(K^p)=C_c^{\infty}(K^p\backslash G(\mathbb{A}_f^p)/K^p)\hat{\otimes}\mathcal{U}_p\end{equation}

Given $t\in T^+$, denote by $u_t$ the element in $\mathcal{U}_p$ whose image in $\mathbb{Z}_p[T^+/T(\mathbb{Z}_p)]$ is $t$. A Hecke operator $f$ is called admissible, if $f=f^p\otimes u_t$ and $t\in T^{++}$. We denote by $\mathcal{H}'_p$ the subalgebra of $\mathcal{H}_p$ generated by admissible operators. For fixed $K^p$, let $\mathcal {S}$ be the finite set of primes $l$ such that $K_l$ is not maximal, define
\begin{equation}R_{\mathcal{S},p}:=C_c^{\infty}(G(\mathbb{A}_f^{\mathcal{S}\cup\{p\}})//K^{\mathcal{S}\cup\{p\}})\otimes\mathcal{U}_p\end{equation}
$R_{\mathcal{S},p}$ is commutative and can be identified in the center of $\mathcal{H}_p(K^p)$. 

\vspace{1pc}
Throughout this paper, we assume that $G$ has a finite order automorphism $\iota$ of Cartan-type, that is, at $\infty$, $\iota$ is of the form $ad(g_\infty)\circ\theta$, for some $g_\infty\in G_\infty$ and the Cartan involution $\theta$ (with respect to $K_\infty$). It is innocuous to assume that the triples $(B,T,I_m)$ are stable under $\iota$. Indeed, let $(B,T,I_m)$ be such a triple and $\psi_0$ the based root datum associated to it, consider the splitting exact sequence \cite[2.14]{Springer}:
\begin{equation}
1\rightarrow Int(G)\rightarrow Aut(G)\xrightarrow{\beta} Aut(\psi_0)\rightarrow 1.
\end{equation}
If $(B,T,I_m)$ is not stable under $\iota$, we fix a splitting 
\begin{equation}
\gamma: Aut(\psi_0)\xrightarrow{\cong}Aut(G,B,T,\{u_\alpha\})\hookrightarrow Aut(G)
\end{equation}
and replace $\iota$ by its image $\iota'$ under $\gamma\beta$, then $\iota'$ fixes the pair $(B,T)$. Since $\iota'$ has the same image under $\beta$ as $\iota$, it differs $\iota$ by a conjugation. So $\iota'$ is also of Cartan-type. Since $Aut(\psi_0)$ is finite, that $\iota'$ is of finite order. Finally, noticing that $\{u_\alpha\}$ is the set of an arbtary choice of $1\neq u_\alpha\in U_\alpha$ in each unipotent root subgroup $U_\alpha$ associated to the basis $\{\alpha\}$ in $\psi_0$, we can properly choose $\{u_\alpha\}$ such that each $u_\alpha$ corresponds to a wall of one same chambre $\mathcal{C}$ in $\mathcal{BL}$. $\mathcal{C}$ gives an Iwahori subgrop which is stable under $\iota'$. 

\vspace{1pc}
Further assuming that $K^p_f$ is stable under $\iota$, we define $\iota$ acting on the Hecke algebra $\mathcal{H}_p(K^p)$ by sending $f$ to $f^\iota$, where $f^\iota(g):=f(g^{\iota^{-1}})$ for any $g\in G$. Noting that $T^+$ and $T^{++}$ are stable under $\iota$ by (\ref{T+}) and (\ref{T++}), This is well defined. Moreover, for $u_t\in\mathcal{U}_p$, 
\begin{equation}
u_t^\iota=u_{t^\iota}
\end{equation}

\subsection{weight spaces}
\subsubsection{classical weight and coweight}
Let $X^*(T)$ be the set of algebraic weights of $T$, and $X_*(T)$ the set of algebraic coweights. There is a canonical duality pairing 
\begin{equation}
(\ ,\ ): X^*(T)\times X_*(T)\rightarrow \mathbb{Z}
\end{equation}
such that for any $\lambda\in X^*(T)$, $\mu^\vee\in X_*(T)$ and $a\in\mathbb{G}_m$, 
\begin{equation}
\lambda\circ\mu^\vee(a)=a^{(\lambda ,\mu^\vee )}
\end{equation}
We define $\iota$ acting on $X^*(T)$ by sending $\lambda$ to $\lambda^\iota$ such that $\lambda^\iota(t)=\lambda(t^{\iota^{-1}})$ for any $t\in T$, and define $\iota$ acting on $X_*(T)$ by sending $\mu^\vee$ to $(\mu^\vee)^\iota$ such that $(\mu^\vee)^\iota(a)=(\mu^\vee(a))^{\iota^{-1}}$ for any $a\in\mathbb{G}_m$. It is easy to see that
\begin{equation}
(\lambda^\iota,(\mu^\vee)^{\iota})=(\lambda,\mu^\vee).
\end{equation}

\subsubsection{$p$-adic weight space}
There is a rigid space $\mathfrak{X}_{T}$ associated to $T$, such that for any field $L\subset\overline{\mathbb{Q}}_p$,
\begin{equation}\mathfrak{X}_{T}(L)=Hom_{cont}(\mathcal T(\mathbb{Z}_p),L^{\times}).\end{equation}
Since $\mathcal T(\mathbb{Z}_p)=\mathbb{Z}_p^r \times \Pi,$
with some finite group $\Pi$, that
\begin{equation}\mathfrak{X}_{T}(\overline{\mathbb{Q}}_p)=Hom_{qp}(\Pi, \overline{\mathbb{Q}}_p^{\times})\times
(B_{1,1}(\overline{\mathbb{Q}}_p)^{\circ})^r.\end{equation}
So the underlying space of $\mathfrak{X}_{T}$ is finite many copies of the $r$-tuple unit ball, its points are (continuous) $p$-adic weights. Put $Z_{K^p}=Z(\mathbb{Q})\bigcap K^pI$ and let $\mathfrak{X}:=\mathfrak{X}_{K^p}\subseteq
\mathfrak{X}_{T}$ be the Zariski closure of the subset of $p$-adic weights which are trivial on $Z_{K^p}$. The automorphism $\iota$ induces a operator on $\mathfrak{X}$ which sends $\lambda$ to $\lambda^{\iota}$, where $\lambda^\iota(t):=\lambda(t^{\iota^{-1}})$ for any $t\in \mathcal{T}(\mathbb{Z}_p)$. Denote by $\mathfrak{X}^\iota$ the subspace of $\mathfrak{X}$ consisting of $\iota$-invariant weights. 

\vspace{1pc} 
Recall, for any $n$, there is a rigid space $\mathfrak T_n$ such that for any field $L\subset\overline{\mathbb{Q}}_p$,
\begin{equation}\mathcal {O}(\mathfrak{T_n}_{/L})=\mathcal {A}_n(\mathcal T(\mathbb{Z}_p), L),\end{equation}
where $\mathcal {A}_n(\mathcal T(\mathbb{Z}_p), L)$ is the space of locally $n$-analytic $L$-valued functions on $\mathcal T(\mathbb{Z}_p)$. The natural pairing
\begin{equation}\label{above}\mathfrak{X}_{T}(L) \times \mathcal T(\mathbb{Z}_p)
\rightarrow L^\times, (\lambda, t)\mapsto \lambda(t)\end{equation} 
induces a continuous injective homomorphism $\mathcal T(\mathbb{Z}_p)\hookrightarrow
\mathcal {O}(\mathfrak{X}_{T})^{\times}$.

\begin{lem}\label{lemma2}
For any affinoid subdomain $\mathfrak{U} \subseteq
\mathfrak{X}$ or $\mathfrak{X}^\iota$, there exist a smallest integer
$n(\mathfrak{U})$, such that for any finite extension $L$ of
$\mathbb{Q}_p$, every element $\lambda \in \mathfrak{U}(L)$ is
$n(\mathfrak{U})$-locally analytic. Moreover, there is a
rigid analytic map $\mathfrak{U} \times
\mathfrak T_{n(\mathfrak{U})}\rightarrow B_{1,1}$, such that for any $L$, its realization at $L$-points is the pairing $(\ref{above})$.
\end{lem}
It follows immediately from \cite[lemma 3.4.6]{Urban}.

\subsection{analytic induced modules and distribution spaces}
\subsubsection{induced modules}
We recall some definitions from \cite[\S3.2]{Urban} first. For $\lambda^{alg} \in X^*(T)$, let $\mathbb{V}_{\lambda^{alg}}$ be the finite dimensional irreducible algebraic representation of $G$ with highest weight $\lambda^{alg}$. It can be viewed as the algebraic induced representation:
 \begin{equation}
 \mathbb{V}_{\lambda^{alg}}=ind_{B}^G(\lambda^{alg})^{alg}.
 \end{equation}
For field $L$ as last section, let $\epsilon: \mathcal T(\mathbb{Z}_p)\rightarrow L^\times$ be a finite character of order $m$. We identify $\lambda^{alg}$ with the $p$-adic weight obtained by the composition
\begin{equation}
\mathcal{T}(\mathbb{Z}_p)\hookrightarrow T(L)\xrightarrow{\lambda^{alg}}L^\times
\end{equation}
Now for the $p$-adic weight $\lambda= \lambda^{alg}\epsilon$, consider its $m$-locally analytic induction： \begin{equation}\mathcal{V}_\lambda=ind_{B}^G(\lambda)^{m-an}.\end{equation} 
There is a natural map \begin{equation}\mathbb{V}_{\lambda^{alg}}(L)(\epsilon)\hookrightarrow \mathcal{V}_\lambda (L).\end{equation}
The $*$-action described in \S2.1 induces an action of  $\Delta^+$  on $\mathcal{V}_{\lambda}(L)$, via the right $*$-translation.

\vspace{1pc}
For any $\lambda \in\mathfrak{X}(L)$, let $\mathcal{A}_{\lambda}(L)$ be the space of locally $L$-analytic functions $f$ on $I$ such that
\begin{equation}f(n^-tg)=\lambda(t)f(g),\end{equation}
where, as in (\ref{Iwahori}), $n^-\in I\cap N^-(\mathbb Q_p)$, $t\in \mathcal T(\mathbb{Z}_p)$ and $g\in I$. $\mathcal {A}_{\lambda}(L)$ can be viewed as a subspace of $\mathcal{A}(\Omega_1, L)$, the space of locally $L$-analytic functions on $\Omega_1$: let $\mathcal T(\mathbb{Z}_p)$ act on $\mathcal{A}(\Omega_1, L)$ by the natural translation, then
\begin{equation}\label{2.3.5}
\mathcal{A}_{\lambda}(L)=\mathcal{A}(\Omega_1, L)[\lambda]:=\{\phi\in\mathcal{A}(\Omega_1, L)|t\phi=\lambda(t)\phi\}
\end{equation}
The $*$-action $\Delta^+$ on $\mathcal{A}(\Omega_1, L)$ is naturally defined, it commutes with the translation action of $\mathcal T(\mathbb{Z}_p)$ above. So the $*$-action of $\Delta^+$ is well defined on
$\mathcal{A}_{\lambda}(L)$. For  $g \in \Delta^+$ and $\phi\in \mathcal{A}_\lambda(L)$, we have \begin{equation}g*\phi([x]): = \phi([x]*g)\end{equation}

\vspace{1pc}
Now define the distribution space
\begin{equation}\mathcal{D}_{\lambda}(L):=Hom_{cont}(\mathcal{A}_{\lambda}(L),L)\end{equation}
be the continuous dual of $\mathcal{A}_{\lambda}(L)$. The $*$-action of $\Delta^+$ on $\mathcal{D}_{\lambda}(L)$ is naturally defined. A deatiled study of $\mathcal{A}_{\lambda}(L)$ and $\mathcal{D}_{\lambda}(L)$ can be found in \cite[lemma 3.2.8]{Urban}, in particular, we have next proposition:

\begin{prop}\label{prop232}
$\mathcal{D}_{\lambda}(L)$ is a compact Frechet space over $L$. If
$\delta \in \Delta^{++}$, then the $*$-action of $\delta$ defines a
compact operator on $\mathcal{D}_{\lambda}(L)$.
\end{prop}

\begin{remark}
The theory of compact operators on orthonormalizable ($p$-adic) Banach spaces is originally due to Serre and generalized by Coleman \cite{Coleman}. The theory is generalized to compact Frechet spaces by Urban in \cite[section 2]{Urban}, where he shows that most results of compact operators on Banach spaces still hold for compact Frechet spaces.
\end{remark}

For $\lambda\in\mathfrak{X}^\iota$, $\iota$ acts on $\mathbb{V}_{\lambda}$, $\mathcal{V}_\lambda$ and $\mathcal{A}_\lambda$. Concretely, let $f$ be a function on $N^-(L)\backslash G(L)$, define $f^\iota(g)=f(g^{\iota^{-1}})$. If $f$ is in one of those induced modules, $f^\iota(bg)=f(b^{\iota^{-1}}g^{\iota^{-1}})=\lambda(t^{\iota^{-1}})f(g^{\iota^{-1}})=\lambda(t)f^{\iota}(g)$, for any $b=tn\in B$. So $\mathbb{V}_{\lambda}$, $\mathcal{V}_\lambda$ and $\mathcal{A}_\lambda$ are stable under $\iota$. We let $\iota$ act on $\mathcal{D}_\lambda$ via duality.

\subsubsection{analytic family of induced modules}
Let $\mathfrak{U}$ be an affinoid subdomain of $\mathfrak{X}$ or $\mathfrak{X}^\iota$. Fix $n\geq n(\mathfrak{U})$. There is a rigid space $(\Omega_m)_n^{rig}$ such that $\mathcal{O}({(\Omega_m)_n^{rig}}_{/L})=\mathcal{A}_n(\Omega_m, L)$ for any $L\subset\overline{\mathbb{Q}}_p$, where $\mathcal{A}_n(\Omega_m, L)$ is the space of locally $L$-analytic functions on $\Omega_m$ with local analytic radius $p^{-n}$. Keeping this identity, let $\mathcal{A}_{\mathfrak{U}, n}(L)$ be the ring of rigid analytic $L$-valued functions on $\mathfrak{U}\times{(\Omega_1)}_n^{rig}$ such that
\begin{equation}
f(\lambda, [tn])=\lambda(t)f(\lambda, [n])
\end{equation}
for any $\lambda\in\mathfrak{U}(L)$, $t\in
\mathcal T(\mathbb{Z}_p)$ and $n\in {\mathcal N(\mathbb{Z}_p)}$. Here we view $f(\lambda, -)$ as a function in $\mathcal{A}_n(\Omega_m, L)$. This implies that
\begin{equation}
\mathcal{A}_{\mathfrak{U},n}=\mathcal {O}(\mathfrak{U})\hat{\otimes}\mathcal{A}_n(\mathcal{N}(\mathbb{Z}_p)).
\end{equation}
In particular, $\mathcal{A}_{\mathfrak{U},n}$ is an
$\mathcal{O}(\mathfrak{U})$-orthonormalizable Banach space. Similar to (\ref{2.3.5}), since
\begin{equation}\mathcal{A}_{\mathfrak{U},n}=
\{f\in\mathcal{O}((\Omega_1)^{rig}_n)\otimes\mathcal{O}(\mathfrak{U})|t(f\otimes 1)=f\otimes t, t\in
T(\mathbb{Z}_p)^{rig}_n\},\end{equation}
that the $*$-action of $\Delta^+$ is well defined on $\mathcal{A}_{\mathfrak{U},n}$.

\vspace{1pc}
Now define
\begin{equation}\mathcal{A}_{\mathfrak{U}}:=\bigcup_{n\geq n(\mathfrak{U})}\mathcal{A}_{\mathfrak{U},n},\end{equation}
and let $\mathcal{D}'_{\mathfrak{U},n}:= Hom_{\mathcal{O}(\mathfrak{U})}(\mathcal{A}_{\mathfrak{U},n}, \mathcal{O}(\mathfrak{U}))$ be the continuous $\mathcal{O}(\mathfrak{U})$-dual of $\mathcal{A}_{\mathfrak{U},n}$. There is a canonical injective map
\begin{equation}
\mathcal{O}(\mathfrak{U})\hat{\otimes}_L\mathcal{D}_n(\mathcal N(\mathbb{Z}_p), L)\rightarrow \mathcal{D}_{\mathfrak{U},n}'.
\end{equation}
Let $\mathcal{D}_{\mathfrak{U},n}$ be the image of this map, define
\begin{equation}
\mathcal{D}_\mathfrak{U}:= \varprojlim \mathcal{D}_{\mathfrak{U},n}.
\end{equation}
$\mathcal{A}_{\mathfrak{U}}$ and $\mathcal{D}_\mathfrak{U}$ are $\Delta^+$-modules with the $*$-action. Since the inclusions
$\mathcal{A}_{\mathfrak{U},n}\subset\mathcal{A}_{\mathfrak{U},n+1}$
are completely continuous, $\mathcal{D}_\mathfrak{U}$ is a Frechet space over $\mathcal{O}(\mathfrak{U})$.

\vspace{1pc}
\begin{prop}\label{prop}
Notation as above, we have 
\begin{itemize}
\item[$(a)$] $\mathcal{A}_{\mathfrak{U}}\otimes_{\lambda}L\cong\mathcal{A}_{\lambda}(L)$ and $\mathcal{D}_{\mathfrak{U}}\otimes_{\lambda}L\cong\mathcal{D}_{\lambda}(L)$ via specialization. 
\item[$(b)$] If $\delta\in\Delta^{++}$, the $*$-action of $\delta$ gives a compact operator on the $\mathcal{O}(\mathfrak{U})$-projective compact Frechet space $\mathcal{D}_{\mathfrak{U}}$.
\end{itemize}
\end{prop}
All of these results can be found in \cite[section 3.4]{Urban}.

\begin{remark}
We make some remarks here:
\begin{itemize}
\item[$(a)$] The $*$-action of $\Delta^+$ on $\mathcal{D}$ is compatible with the
natural action of $I$ on it. 
\item[$(b)$] The $*$-actions of $\Delta^+$ on $\mathbb{V}^\vee_{\lambda^{alg}}(L)$, $\mathcal{V}^\vee_\lambda(L)$, $\mathcal{D}_\lambda(L)$ and $\mathcal{D}_{\mathfrak{U}}(L)$ are right action, we translate it into a left action by defining for every $\delta\in\Delta^+$
\begin{equation*}\delta*:=*\delta^{-1}\end{equation*}
\item[$(c)$] For $K_f=K^pI$, we view $\mathcal{D}$ as a $K_f$-module via the projection $K_f\rightarrow I$. 
\end{itemize}
\end{remark}

\vspace{1pc}
\section{TWISTED ACTIONS ON RESOLUTIONS AND COHOMOLOGY SPACES}

\subsection{cohomology spaces and resolutions} 
We recall first some standard results on the cohomology spaces on which we work later. Let $M$ be a $(G(\mathbb{Q}), K)$-module on which $Z_K$ acts trivially. $M$ defines a local system on $S_G(K_f)$, which is denoted by ${M}$ as well. One is interested in the cohomology space 
$H^*(S_G(K_f), M)$. In this paper, $M$ is one of $\mathbb{V}^\vee_{\lambda^{alg}}(L)$, $\mathcal{V}^\vee_\lambda(L)$, $\mathcal{D}_\lambda(L)$ and $\mathcal{D}_{\mathfrak{U}}(L)$, where the upper index $^\vee$ indicates the continuous dual space. 

\vspace{1pc}
There are two equivalent ways to define the cohomology. Let $\overline{S}_G(K_f)=\overline{S}_G/K_f$ be the Borel-Serre compactification of $S_G(K_f)$. Then $\overline{S}_G=G(\mathbb{Q})\backslash G(\mathbb{A}_f)\times\overline{\mathcal{H}}_G$ and  $\overline{\mathcal{H}}_G$ is a contractible real manifold with corners. There is a canonical projection:
\begin{equation}
\pi: \overline{S}_G\rightarrow\overline{S}_G(K_f),
\end{equation}
which extends the natural projection $\pi:S_G\rightarrow S_G(K_f)$.

\vspace{1pc}
Fix a finite triangulation of $\overline{S}_G(K_f)$ and pull it back to $\overline{S}_G$. Let $C_*(K_f)$ be the corresponding chain complex, that is, $C_q(K_f)$ is the free $\mathbb{Z}$-module over the set of $q$-dimensional simplexes of the pull-back triangulation. $C_*(K_f)$ admits a right $K_f$-action, and $C_q(K_f)$ is a free right $\mathbb{Z}[K_f]$-module of finite rank. We define
\begin{equation}
R\Gamma^{*}(K_f, M):=Hom_{K_f}(C_*(K_f), M),
\end{equation}
then $R\Gamma^{j}(K_f, M)$ is isomorphic to finitely many copies of $M$ and
\begin{equation}
h^j(R\Gamma^{*}(K_f,M))=H^j(S_G(K_f), M).
\end{equation}

Another way to define the cohomology is using the $M$-valued de Rham complex $\Omega^*(S_G(K_f), M)$. The natural duality between $\Omega^*(S_G(K_f))$ and $C_*(K_f)$ implies that the two definitions are coincident.

\begin{remark}
As summarized in \cite[\S1.2]{Urban}, for a $(G(\mathbb{Q}), K)$-module $M$, there are two equivalent ways to define the local system $M$ on $S_G(K_f)$, with respect to the $K_f$-module structure and to the $G(\mathbb{Q})$-module structure respectively.  So are the two definitions of cohomology space above.
\end{remark} 

\subsubsection{functoriality}
\noindent There is a functoriality for $R\Gamma^*(K_f, M)$. Let $\varphi: K_f'\rightarrow K_f$ be a group homomorphism and $\varphi^\#: M\rightarrow M'$ a homomorphism between a $K_f$-module $M$ and a $K'_f$-module $M'$, such that $\varphi^\#(\varphi(k')m)=k'\varphi^\#(m)$ for any $k'\in K_f'$ and $m\in M$. The pair $(\varphi, \varphi^\#)$ then induces a morphism $\varphi^*: R\Gamma^*(K_f, M)\rightarrow R\Gamma^*(K_f', M')$ up to homotopy, see \cite[\S4.2.5]{Urban}.

\subsubsection{Hecke operators on resolution and cohomology}
\noindent
Apply the functoriality, as in \cite[\S4.2]{Urban}, $f=f^p\otimes u_t\in \mathcal{H}_p(K^p)$ defines a morphism $R\Gamma(t):R\Gamma^*(K_f, M)\rightarrow R\Gamma^*(K_f, M)$ by the composition:
$$R\Gamma^*(K_f, M)\rightarrow R\Gamma^*(tK_ft^{-1}, M)\rightarrow R\Gamma^*(K_f\cap tK_ft^{-1}, M)\rightarrow R\Gamma^*(K_f, M)$$
where the first map is given by the pair $(ad(t^{-1}), m\mapsto t*m)$, the second is by the restriction map from $K_f\cap tK_ft^{-1}$ to $tK_ft^{-1}$, and the last one is given by the corestriction as writing 
\begin{equation}
K_f=\sqcup_{j}k_j(K_f\cap tK_ft^{-1}).
\end{equation}
It is easy to see that $R\Gamma(t_1)\circ R\Gamma(t_2)=R\Gamma(t_1t_2)$. This defines an action of $\mathcal{H}_p(K^p)$ on $R\Gamma^*(K_f, M)$ and therefore defines an action on the cohomology spaces $H^*(S_G(K_f), M)$. We denote this action by $*$ as well.

\vspace{1pc}
If $M=\mathcal{D}_\lambda(L)$ and $t \in T^{++}$, by the fact that $R\Gamma^q(K_f, M)$ is a finite copy of $M$, Proposition \ref{prop232} implies that $f$ is a compact operator on $R\Gamma^*(K^pI, \mathcal{D}_\lambda(L))$. If $\mathfrak{U}$ is an open affinoid of $\mathfrak{X}$ and $\lambda\in \mathfrak{U}$, by Proposition \ref{prop}, $R\Gamma^*(K^pI, \mathcal{D}_\lambda(L))$ can be obtained by the specialization of $R\Gamma^{*}(K^pI, \mathcal{D}_{\mathfrak{U}})$ at $\lambda$. Moreover, for affinoids $\mathfrak{U}'\subset\mathfrak{U}$, $R\Gamma^{*}(K^pI, \mathcal{D}_{\mathfrak{U}'})$ can be obtained via the natural restriction morphism $\mathcal{O}(\mathfrak{U})\rightarrow\mathcal{O}(\mathfrak{U}')$. The Hecke action is compatible with specialization and restriction.

\subsubsection{$\iota$ action on resolution and cohomology}
Assume $\lambda\in\mathfrak{X}^\iota$ and $\mathfrak{U}$ is an open affinoid of $\mathfrak{X}^\iota$. Let $M$ be one of $\mathbb{V}^\vee_{\lambda}(L)$, $\mathcal{V}^\vee_\lambda(L)$, $\mathcal{D}_\lambda(L)$ and $\mathcal{D}_{\mathfrak{U}}(L)$. Choose $K_f=K^pI\subset G(\mathbb{A}_f)$, such that $K^p$ is stable under $\iota$. Consider morphisms $\iota: K_f\rightarrow K_f$ and $\iota: M\rightarrow M$ defined as in \S2. 

\begin{lem}
Assume $M = \mathbb{V}^\vee_{\lambda}(L)$, $\mathcal{V}^\vee_\lambda(L)$, $\mathcal{D}_\lambda(L)$ or $\mathcal{D}_{\mathfrak{U}}(L)$. For $g\in I$, $x\in M$,
\begin{equation}g*x^\iota = (g^\iota*x)^\iota\end{equation}
Therefore, by the functoriality, $\iota$ defines an morphism on $R\Gamma^*(K_f, M)$ up to homotopy. In particular, $\iota$ acts on the cohomology $H^*(S_G(K_f), M)$.
\end{lem}
The lemma follows immediately from a computation by definition.

\subsubsection{Action of $^\iota\mathcal{H}_p(K^p)$}
For $\iota$-invariant $K^p$, define the $\iota$-twisted Hecke algebra:
\begin{equation}
^\iota \mathcal{H}_p(K^p): =\mathcal{H}_p(K^p) \rtimes\left<\iota\right>,
\end{equation} 
where $\left<\iota\right>$ is the finite group generated by $\iota$ and the semi-product $\rtimes$ is understood as a crossed product, since at every place the local Hecke algebra can be viewed as a group algebra of double cosets. We write $f\times\iota$ and $\iota\times f$ for the products of $f\in\mathcal{H}_p(K^p)$ and $\iota$. We similarly define $^\iota\mathcal{H}_p$, then $^\iota\mathcal{H}_p$ is the inductive limit of ${^\iota \mathcal{H}_p(K^p)}$.

\begin{lem}
There is an action of $^\iota\mathcal{H}_p(K^p)$ on $R\Gamma^*(K_f, M)$, extending the $*$-action of $\mathcal{H}_p(K^p)$ and $\iota$.
\end{lem}

\begin{proof}  
We have to check that the $*$-actions of $\mathcal{H}_p(K^p)$ and $\iota$ on $R\Gamma^*(K, M)$ are compatible in the sense that $\iota \times f = f^\iota \times \iota$. So we only have to check: 
\begin{equation}\iota\circ R\Gamma(t)\circ\iota^{-1}=R\Gamma(t^\iota),\end{equation}
which is again directly from the definition.
\end{proof}

\subsubsection{compare to the standard sheaf-theoretic action}
Assume $M=\mathbb{V}_\lambda^\vee(\mathbb{C})$, we compute the cohomology by de Rham complex:
\begin{equation}
H^q(S_G(K_f), M)=h^q(\Omega^*(S_G(K_f), M)),
\end{equation}
and we have the standard sheaf-theoretic definion of $\mathcal{H}_p$ action on it. By (\ref{diff}), we have for any $\phi\in\mathbb{V}_\lambda$ and $\delta\in\Delta^{+}$,
\begin{equation}\label{twistfactor}
\delta*\phi=\lambda(\xi(t_\delta))^{-1}(\delta\cdot \phi).
\end{equation}
This implies for any $f=f^p\otimes u_t\in\mathcal{H}_p$, that the $*$-action of $f$ on $H^q(S_G(K_p), M)$ is the standard action of $f$ twisting by $\lambda(\xi(t))$.

\vspace{1pc}
The $\iota$-action on $\Omega^*(S_G(K_f), M)=\Omega^*(S_G(K_f))\otimes M$ is define by $(\iota^{-1})^*\otimes\iota$, 
where $(\iota^{-1})^*$ means the pull-back on differential forms induced by the map 
\begin{equation}
\iota^{-1}: S_G(K_f)\leftarrow S_G(K_f)
\end{equation}
This $\iota$-action can be described explicitly as follow. Let $T(S_G)$ and $T(S_G(K_f))$ be the sheaves of left invarinat vecter fields on $S_G$ and $S_G(K_f)$ respectively, the projection $\pi$ induces an push-forward surjection:
\begin{equation}
\pi_*: T(S_G)\rightarrow T(S_G(K_f)).
\end{equation}
One views an $q$ differential form $\tau$ in $\Omega^*(S_G(K_f), M)$ as a map
\begin{equation}
\tau: \wedge^qT(S_G(K_f)\rightarrow \mathcal{O}(S_G(K_f))\otimes M
\end{equation}
then $\tau^\iota$ is defined as 
\begin{equation}
\tau^\iota(\pi_*\bar{v}_1\wedge\cdots\wedge\pi_*\bar{v}_q)([g]):=(\tau((\pi_*\iota^{-1}_*\bar{v}_1\wedge\cdots\wedge\pi_*\iota^{-1}_*\bar{v}_q))([g]^\iota))^{\iota}
\end{equation}
where $\bar{v}$ is a left invariant vector field on $S_G$, $g\in G(\mathbb{A})^1$ and $[g]$ indicates the class of $g$ in $S_G$ or $S_G(K_f)$. It is easy to check that $\iota$ is well defined on $H^*(S_G(K_f), M)$ under this definition. The duality between $\Omega^*(S_G(K_f))$ and $C_*(K_f)$ implies that this action coincides with the one defined by functoriality.

\subsection{twisted action on finite slope cohomology}
We need a lemma on the slope decomposition of compact projective Frechet spaces according to compact operators.
\begin{lem}\label{slopedecomp}
Let $A$ be a $\mathbb{Q}_p$-Banach algebra, $M$ a compact projective
Frechet $A$-module, and $f$ a compact $A$-linear operator of $M$.
Then the Fredholm determinant $R(f,X)$ of $f$ is entire over $A$. If $R(f,X)=Q(X)S(X)$ over $A$, such that $Q$ and $S$ are relatively prime and $Q$ is a Fredholm polynomial with invertible leading
coefficient, then there is a decomposition of $M$: 
\begin{equation}M=N_f(Q)\oplus F_f(Q)\end{equation}
into $f$-stable close submodules satisfing:
\begin{itemize}
\item[$(a)$] $Q^*(f)$ annihilates $N_f(Q)$ and is invertible on $F_f(Q)$;
\item[$(b)$] the projector on $N_f(Q)$ is given by $E_Q(f)$ with
$E_Q(X)\in XA\{\{X\}\}$ whose coefficients are polynomials in the coefficients of $Q$ and $S$.
\end{itemize}
Moreover, if $A$ is semi-simple, then $N_f(Q)$ is of finite rank,
and the characteristic polynomial of $f$ on $N_f(Q)$ is $Q$. In particular, for $h\in\mathbb{Q}_{\geq0}$, we may choose $Q(x)$ such that $N_f(Q)=M^{\leq h}$, the $\leq h$-slope decomposition of $M$.
\end{lem}

\begin{proof}
The lemma is known if $M$ is a projective Banach module by Serre \cite{Serre} and Coleman \cite{Coleman}. Now $M$ is a  projective compact Frechet space, there are projective $A$-Banach modules $M_n$ with compact operators $f_n$, such that \begin{equation}
M=\varprojlim M_n, \  f=\varprojlim f_n
\end{equation}
with $f_n=f|M_n$. Now $R(f,X)=det(1-Xf|M_n)$ for $n$ sufficiently large, so $R(f,X)=Q(X)S(X)$ gives the expected decomposition $M_n=N_{n,f}(Q)\oplus F_{n,f}(Q)$. Let $p_n$ be the projector of $M_n$ onto $N_{n,f}(Q)$, by \cite[Theorem 3.3]{Buzzard eigenvarieties}, there is a power series $\phi\in A[[T]]$ depending only on $Q$, such that $p_n=\phi(f_n)$. Moreover, $N_{n,f}(Q)$ and $F_{n,f}(Q)$ are given by the image and kernel of $p_n$ respectively. Denote by $u_{n+1, n}$ the transation map from $M_{n+1}$ to $M_n$. By definition, we have a commutative diagram: 
 \begin{equation}\begin{array}[c]{ccccc}
M&\rightarrow& M_{n+1}&\stackrel{u_{n+1,n}}{\rightarrow}& M_n\\
 \downarrow\scriptstyle{f}& &\downarrow\scriptstyle{f_{n+1}}&\swarrow\scriptstyle{} & \downarrow\scriptstyle{f_{n}}\\
M&\rightarrow& M_{n+1}&\stackrel{u_{n+1,n}}{\rightarrow}& M_n
 \end{array}\end{equation}
Taking projective limit, we have a projector $p=\varprojlim \phi(f_n)$ on $M$ and the decomposition $M=N_{f}(Q)\oplus F_{f}(Q)$. Indeed, by the definition of compact operator, $N_{n,f}(Q)$ are isomorphic for $n$ sufficiently large. So the last statement follows.
\end{proof}

\vspace{1pc}
Now consider $f=f^p\otimes u_t\in R_{\mathcal{S},p}$ admissible. Then $f$ defines a compact operator on the complex $R\Gamma(K_f, \mathcal{D}_\lambda(L))$. For given $h\in\mathbb{Q}_{\geq0}$, we define the $\leq h$-slope part of $H^*(S_G(K^pI), \mathcal{D}_\lambda(L))$ with respect to $f$ as in \cite[ \S2]{Urban}. Then the finite slope part of $H^*(S_G(K^pI), \mathcal{D}_\lambda(L))$ is defined by:
\begin{equation}
H^*_{fs}(S_G(K^pI), \mathcal{D}_\lambda(L)):=\varinjlim_{h} H^*(S_G(K^pI), \mathcal{D}_\lambda(L))^{\leq h}.
\end{equation} 
Since $R_{\mathcal{S},p}$ is in the center of $\mathcal{H}_p(K^p)$, $H^*_{fs}(S_G(K^pI), \mathcal{D}_\lambda(L))$ is independent of $f$, and endowed with the $*$-action of $\mathcal{H}_p(K^p)$. We also define
\begin{equation}
H^*_{fs}(\tilde S_G, \mathcal{D}_\lambda(L)):=\varinjlim_{K^p}H^*_{fs}(S_G(K^pI), \mathcal{D}_\lambda(L)).
\end{equation}

\begin{prop}\label{prop322}
Assume $\lambda\in\mathfrak{X}^\iota$. The $*$-action of $\mathcal{H}_p(K^p)$ on the finite slope cphomology  $H^*_{fs}(S_G(K^pI), \mathcal{D}_\lambda(L))$ extends to an action of $^\iota\mathcal{H}_p(K^p)$. Therefore the action of $\mathcal{H}_p$ on $H^*_{fs}(\tilde S_G, \mathcal{D}_\lambda(L))$ extends to an action of $^\iota\mathcal{H}_p$.
\end{prop}

\begin{proof}
We only have to prove the first statement. For this we need the next result from \cite[lemma 2.3.2]{Urban}:
\begin{lem}
Let $M$, $M'$ be two $L$-Banach $($or Frechet $)$ spaces, $u$ and $u'$ endomorphism of $M$ and $M'$, and $M=M^{\leq h}_u\oplus M_1$  and $M'={M'}^{\leq h}_{u'}\oplus M'_1$ their $\leq h$-slope decompositions respectively. Assume $f$ is a continuous linear map from $M$ to $M'$ such that $f\circ u = u'\circ f$, then $f$ respects the slope decompositions.
\end{lem}
Since $\iota\times f = f^\iota \times \iota$, by the lemma,
\begin{equation}\iota : H^q_{fs}(S_G(K^pI), \mathcal{D}_\lambda(L))^{\leq h}_{f^\iota}\rightarrow H^q_{fs}(S_G(K^pI), \mathcal{D}_\lambda(L))^{\leq h}_{f}\end{equation}
is well defined. Let $l$ be the order of $\iota$, define
\begin{equation}H^q_{fs}(S_G(K^pI), \mathcal{D}_\lambda(L))^{\leqslant h}_\iota:=\bigcap_{i=1}^l H^q_{fs}(S_G(K^pI), \mathcal{D}_\lambda(L))^{\leq h}_{f^{\iota^i}}.\end{equation}
Then
\begin{equation}\iota: H^q_{fs}(S_G(K^pI), \mathcal{D}_\lambda(L))^{\leqslant h}_\iota\rightarrow H^q_{fs}(S_G(K^pI), \mathcal{D}_\lambda(L))^{\leqslant h}_\iota.\end{equation}
Since the finite slope part is independent of $f$, we prove the proposition by taking the inductive limit.
\end{proof}

\subsection{$\iota$-invariant finite slope representations}
In this section, we introduce the $\iota$-invariant finite slope automorphic representations, which are the main objects we concern in this paper. We first recall a well-known result for admissible representations of a locally profinite group. Let $G$ be a locally profinite group, and $K$ an open compact subgroup of $G$. Write $\mathcal{H}(G)$ the Hecke algebra of compact supported smooth functions of $G$ and $\mathcal{H}(G,K)$ its subalgebra of $K$ bi-invariant functions. Then

\begin{prop}\label{Ktype}
The map $\pi\mapsto\pi^K$ gives a bijection between equivalence classes of irreducible smooth representations $(\pi, V)$ of $\mathcal{H}(G)$ such that $V^K\neq0$ and equivalence classes of irreducible $\mathcal{H}(G,K)$-representations.
\end{prop}

\subsubsection{Finite slope representations}
let $(\pi,V)$ be an irreducible representation of $\mathcal{H}_p$ defined over a $p$-adic field $L$. We say $\pi$ is admissible overconvergent of weight $\lambda\in\mathfrak{X}(L)$ if it is admissible and a subqoutient of $H^q(\tilde S_G, \mathcal{D}_\lambda(L))$. Since $\pi$ is admissible, that for any $K^p$, an element in $\mathcal{H}_p(K^p)$ acts on $V$ as an endomorphism of finite rank. By the fact that $\mathcal{H}_p=\varinjlim_{K^p}\mathcal{H}_p(K^p)$, there is a well-defined trace map:
\begin{equation}
J_{\pi}(f):=tr(\pi(f))
\end{equation}
for any $f\in\mathcal{H}_p$. We say $\pi$ is of level $K^p$ if $\pi^{K^p}$ as a representation of $\mathcal{H}_p(K^p)$ is not trivial. Let $\sigma$ be an irreducible representation of $\mathcal{H}_p(K^p)$. We say $\sigma$ is overconvergent of level $K^p$ and weight $\lambda$ if it is of the form $\pi^{K^p}$ for some admissible overconvergent $\pi$ with level $K^p$ and weight $\lambda$. Then $\sigma$ is finite dimensional and can be realized in the cohomology space $H^q(S_G(K^pI), \mathcal{D}_\lambda(L))$. For fixed $K^p$, the Hecke algebra $R_{\mathcal{S},p}$ is included in the center of $\mathcal{H}_p$. So the restriction of $\sigma$ to $R_{\mathcal{S},p}$ is a character, which is denoted by  $\theta_\sigma$. We call $\theta_\sigma$ an overconvergent Hecke eigensystem of level $K^p$ and weight $\lambda$. For such $\theta$, obviously $H^q(S_G(K^pI), \mathcal{D}_\lambda(L))[\theta]\neq0$. 

\vspace{1pc}
Let $\theta$ be a $\overline{\mathbb{Q}}_p$-valued character of $\mathcal{U}_p$. To recall the definition of the slope of $\theta$, we assume at the moment that $G$ is split at $p$ (refer \cite[\S4.1.2]{Urban} for general situation). If $\theta(u_t)=0$ for some $t\in T^+$, then we say that $\theta$ is of infinite slope. Otherwise, we say $\theta$ is of finite slope. It is easy to check that $\theta$ is of finite slope if and only if there is $t\in T^{++}$ such that $\theta(u_t)\neq0$. In this case, $\theta$ induces a homomorphism from $T(\mathbb{Q}_p)/T(\mathbb{Z}_p)$ to $\overline{\mathbb{Q}}_p^\times$. We then define the slope of $\theta$ to be the element $\mu_\theta\in X^*(T_{/\mathbb{Q}_p})$, such that for any $\mu^\vee\in X_*(T_{/\mathbb{Q}_p})^+$
\begin{equation}
(\mu_\theta,\mu^\vee)=v_p(\theta(u_{\mu^\vee(p)})).
\end{equation}
Now we define the slope of an overconvergent Hecke eigensystem $\theta$ as the slope of its restriction to $\mathcal{U}_p$. For a overconvergent representation $\pi$ or $\sigma$, define its slope as the slope of the overconvergent Hecke eigensystem associated to it. It is easy to see, $\pi$, $\sigma$ or $\theta$ is of finite slope if and only if it is realized in $H^q_{fs}(S_G, \mathcal{D}_\lambda(L))$. Moreover, for any $\mu\in X^*(T_{/\mathbb{Q}_p})$, we define
\begin{equation}
H^q(S_G(K_f), \mathcal{D}_\lambda(L))^{\leq\mu}=\bigoplus_{\theta\mid\mu_{\theta}\leq\mu}H^q(S_G(K_f), \mathcal{D}_\lambda(L))[\theta].
\end{equation}
It is easy to see that
\begin{equation}
\varinjlim_{\mu}H^q(S_G(K_f), \mathcal{D}_\lambda(L))^{\leq\mu}=H^q_{fs}(S_G(K_f), \mathcal{D}_\lambda(L)).
\end{equation}

\vspace{1pc}
Let $\theta$ be a $\overline{\mathbb{Q}}_p$-valued character of $\mathcal{U}_p$ and $\lambda$ an algebraic weight. We say that $\theta$ is non-critical with respect to $\lambda$, if its slope $\mu_\theta$ is. This means that for any $w\neq id$ in $\mathcal{W}_G$, the Weyl group of $G$, $\mu_\theta\notin w\cdot\lambda-\lambda+X^*(T)^+$.

\vspace{1pc}
\subsubsection{$\iota$-invariant finite slope representations}
Let $\rho$ be a finite slope overconvergent representation of $\mathcal{H}$, where $\mathcal{H}$ can be $\mathcal{H}_p$, $\mathcal{H}_p(K^p)$ or $R_{\mathcal{S},p}$. We denote by $\rho^\iota$ the $\iota$-twist of $\rho$, that is, the representation of $\mathcal{H}$ on $V_\rho$, which sends $f\in\mathcal{H}$ to $\rho^\iota(f):= \rho(f^\iota)$. We say $\rho$ is $\iota$-invariant if $\rho^\iota\cong\rho$. By \cite[Appendix]{BLS}, $\rho$ is $\iota$-invariant if and only if it can be extended to $^\iota\mathcal{H}$. Indeed, if $\rho^\iota\cong\rho$, then there is a linear operator $A:V_\rho\rightarrow V_\rho$ of order $l$ such that $A\circ\rho=\rho\circ A$. One can extend $\rho$ by setting $\iota^i$ acting on $V_\sigma$ via $A^i$. Generally, we use capital letters $\Pi$, $\Sigma$ and $\Theta$ for an representation of $^\iota\mathcal{H}_p$, $^\iota\mathcal{H}_p(K^p)$ and $^\iota R_{\mathcal{S},p}$ respectively. The next lemma summarizes results in \cite[Appendix]{BLS} :

\begin{lem}\label{LemmaBLS}
Notations as above 
\begin{itemize}
\item[$(a)$] Assume $\Sigma$ (resp. $\Pi$, $\Theta$) is irreducible, then its restriction to $\mathcal{H}_p(K^p)$ (resp. $\mathcal{H}_p$, $R_{\mathcal{S},p}$) is irreducible if and only if the trace of $\Sigma$ (resp. $\Pi$, $\Theta$) is not trivial on $\iota\times\mathcal{H}_p(K^p)$ (resp. $\iota\times\mathcal{H}_p$, $\iota\times R_{\mathcal{S},p}$).
\item[$(b)$] Assume $\sigma$ (resp. $\pi$, $\theta$ ) is irreducible $\iota$-invariant, then there are exactly $l$ extensions of $\sigma$ (resp. $\pi$, $\theta$) to $^\iota\mathcal{H}_p(k^p)$ (resp. $^\iota\mathcal{H}_p$, $^\iota R_{\mathcal{S},p}$), say $\tilde{\sigma}_1, \dots, \tilde{\sigma}_l$ (resp. $\tilde{\pi}_1, \dots, \tilde{\pi}_l$, $\tilde{\theta}_1, \dots, \tilde{\theta}_l$). Each two of them are differed by a character of order $l$ and are non-isomorphic.
\end{itemize}
\end{lem}

\begin{remark}
By Proposition \ref{Ktype}, if $\sigma=\pi^{K^P}$, we can assume that 
\begin{equation}\label{convention}
\tilde{\pi}_i^{K^p}=\tilde{\sigma}_i.
\end{equation}
Throughout this paper, we make this convention.
\end{remark}

\vspace{1pc}
Let $\tilde{\sigma}$ be a representation of $^\iota\mathcal{H}_p(K^p)$ which extends a $\iota$-invariant overconvergent representation $\sigma$ of $\mathcal{H}_p(K^p)$. We write
\begin{equation}
J_{\tilde{\sigma}}(f):=tr(\tilde\sigma(\iota\times f)).
\end{equation}
Last lemma tells that $J_{\tilde{\sigma}}$ is not trivial. It is also easy to see the Hecke eigensystem $\theta_\sigma$ satisfies
\begin{equation}
\theta(f)=\theta(f^\iota):=\theta^\iota(f)
\end{equation}
for any $f\in R_{\mathcal{S},p}$. We say such a Hecke eigensystem $\iota$-invariant. Let $\tilde{\theta}_{\tilde{\sigma}}$ be the restriction of $\tilde{\sigma}$ to $^\iota{R_{\mathcal{S},p}}$. Since $\iota$ is of finite order and $\tilde{\theta}_{\tilde{\sigma}}$ is finite dimensional, that $\tilde{\theta}_{\tilde{\sigma}}$ is diagonalizable under $\iota$. So $\tilde{\theta}_{\tilde{\sigma}}$ is a direct sum of one dimensional representations of $^\iota R_{\mathcal{S},p}$, which must be of the form $\tilde{({\theta}_\sigma)}_i$. So we have
\begin{equation}
\tilde{\theta}_{\tilde{\sigma}}=\bigoplus_{i=1}^l(\tilde{({\theta}_\sigma)}_i)^{m(\tilde{({\theta}_\sigma)}_i, \tilde{\theta}_{\tilde{\sigma}})}
\end{equation}
where, $m(\tilde{({\theta}_\sigma)}_i, \tilde{\theta}_{\tilde{\sigma}})$ is the multiplicity of $\tilde{({\theta}_\sigma)}_i$ in $\tilde{\theta}_{\tilde{\sigma}}$, 
\begin{equation}
\sum_{i}m(\tilde{({\theta}_\sigma)}_i, \tilde{\theta}_{\tilde{\sigma}})=\dim{\sigma}.
\end{equation}
If $f\in R_{\mathcal{S},p}$, then
\begin{equation}\label{ob1}
J_{\tilde{\sigma}}(f)=tr(\iota\mid\tilde{\sigma})\theta_\sigma(f)
\end{equation}
Since $\iota$ is of finite order, all its eigenvaules must be $p$-adic units. We remark here a simple but important observation that 
\begin{equation}\label{ob2}
v_p(tr(\iota\mid\tilde{\sigma}))\geq0.
\end{equation}

\vspace{1pc}
Let $\theta$ be a finite slope overconvergent Hecke eigensystem with slope $\mu_\theta$. It is easy to check that
\begin{equation}
\mu_{\theta^\iota}=\mu_\theta^{\iota}
\end{equation} 
So we conclude:
\begin{lem}
If a finite slope overconvergent representation (resp. Hecke eigensystem) is $\iota$-invariant, then its slope $\mu_\theta$ is $\iota$-invariant.
\end{lem}

\subsubsection{classicity}
Now we state a twisted analogue to \cite[proposition 4.3.10]{Urban}, which is the classicity theorem in the cohomological setting. For a dominant weight $\lambda \in X^*(T)$ and $t\in T^+$, define 
\begin{equation}
N(\lambda,t):=\inf_{w\neq id}|t^{w\cdot\lambda-\lambda}|_p
\end{equation} 
and 
\begin{equation}\label{Nlambda}
N^\iota(\lambda, t):=\inf_{i=1}^lN(\lambda,t^{\iota^i})
\end{equation}

\begin{prop}\label{classicity}
Let $\lambda=\lambda^{alg}\epsilon$ be an arithmetic weight of conductor $p^{n_\lambda}$, and $\mu$ a slope which is non-critical with respect to $\lambda^{alg}$. Then for any positive integer $m\geqslant n_\lambda$, we have the canonical isomorphism
\begin{equation}
H^*(S_G(K^pI),\mathcal{D}_\lambda(L))^{\leq\mu}\cong H^*(S_G(K^pI_m),\mathbb{V}^\vee_{\lambda^{alg}}(L,\epsilon))^{\leq\mu}
\end{equation}
Similarly, with respect to $f=f^p\otimes u_t$ with $t\in T^{++}$, for any $h\leqslant v_p(N^\iota(\lambda, t))$,
\begin{equation}
H^*(S_G(K^pI),\mathcal{D}_\lambda(L))_\iota^{\leqslant h}\cong H^*(S_G(K^pI_m),\mathbb{V}^\vee_{\lambda^{alg}}(L,\epsilon))_\iota^{\leqslant h}
\end{equation}
 Here $\mathbb{V}^\vee_{\lambda^{alg}}(L,\epsilon):=(\mathbb{V}_{\lambda^{alg}}(L)(\epsilon))^\vee$.
\end{prop}
The proof is same to \cite[proposition 4.3.10]{Urban}.

\subsubsection{spectral expansion of twisted overconvergent traces}
\begin{prop}\label{induction1}
Let $\lambda\in\mathfrak{X}^\iota(L)$. For any $\iota$-invariant finite slope overconvergent representation $\pi$ of $\mathcal{H}_p$, there are $l$ integers $\{m_i^q(\pi, \lambda)\}_{i=i}^{l}$, such that for all $f\in\mathcal{H}_p'$, 
\begin{equation}
tr(\iota\times f\mid H^q_{fs}(\tilde S_G, \mathcal{D}_\lambda(L))) = \sum_{\pi | \pi^\iota\cong\pi}\sum_{i=1}^{l} m^q_i(\pi, \lambda)J_{\tilde\pi_i}(f)
\end{equation}
\end{prop}
\begin{proof}
Since $f$ is admissible, the trace is convergent. Fix $t\in T^{++}$, for any $h\in\mathbb{Q}_{\geq 0}$, consider the $\iota$-stable $\leq h$-slope part $H^q(\tilde{S}_G, \mathcal{D}_\lambda(L))_\iota^{\leq h}$ of $H^q_{fs}(\tilde{S}_G, \mathcal{D}_\lambda(L))$ with respect to $u_t$. Here we define 
\begin{equation}
H^q(\tilde{S}_G, \mathcal{D}_\lambda(L))_\iota^{\leq h}:=\varinjlim_{K^p}H^q({S}_G(K^pI), \mathcal{D}_\lambda(L))_\iota^{\leq h}.
\end{equation}
It is equipped with an admissible $*$-action of $\mathcal{H}_p$ since $u_t$ is in the center of $\mathcal{H}_p$. As the proof of Proposition \ref{prop322}, this action extends to $^\iota\mathcal{H}_p$. 
Let $(\pi, V)$ be a finite slope overconvergent $\mathcal{H}_p$-submodule of $H^q(\tilde S_G, \mathcal{D}_\lambda(L))_\iota^{\leqslant h}$, and $\iota*(V)$ its image under the $*$-action of $\iota$. Consider the next three sets of finite slope overconvergent representations $A,B$ and $C$ whose elements are counted with multiplicity: $A=\{\pi|\iota*(V_\pi)=V_\pi\}$, $B=\{\pi|\pi^\iota\cong\pi, \iota*(V_\pi)\cap V_\pi=\emptyset\}$ and $C=\{\pi|\pi^\iota\ncong\pi\}$. 

If $\pi\in B$ or $C$, write 
\begin{equation}W_\pi=\bigoplus _{i=1}^l\iota^{i}*(V_\sigma)\end{equation}
as a $^\iota\mathcal{H}_p$-submodule of $H^q(\tilde{S}_G, \mathcal{D}_\lambda(L))_\iota^{\leq h}$. Then $\iota*$ permutes the components of $W_\pi$,  and $tr(\iota\times f | W_\sigma)$ is trivial. If $\pi\in A$ (this implies that $\sigma$ is $\iota$-invariant), then $V_\pi$ itself is an irreducible $^\iota\mathcal{H}_p$-submodule of $H^q(\tilde S_G, \mathcal{D}_\lambda(L))_\iota^{\leqslant h}$. In particular, $V_\pi=\tilde\pi_i$ for some $i$. Now for any $\iota$-invariant $\pi$, let $m^q_{i,h}(\pi, \lambda)$ be the mulitplicity of $\tilde{\pi}_i$ appearing in $A$. We have 
\begin{equation}tr(\iota\times f; H^q(\tilde S_G, \mathcal{D}_\lambda(L)))_\iota^{\leqslant h} = \sum_{\pi | \pi^\iota\cong\pi}\sum_{i=1}^{l} m^q_{i,h}(\pi, \lambda)J_{\tilde\pi_i}(f).\end{equation}
Let $h$ go to infinite, the multiplicity $m^q_{i,h}$ converges to some $m^q_i\in\mathbb{Z}$, and the proposition follows.
\end{proof}

\vspace{1pc}
\begin{remark}
If $G$ admits multiplicity one theorem, then $B=\emptyset$ in the proof above (otherwise, since $\sigma$ is self-dual, that $V_\sigma\cong V_{\sigma^\iota}\cong V^{*\iota}_\sigma$ would appear with multiplicitiy at least two).
\end{remark}

\vspace{3pc}
\section{TWISTED FRANKE'S TRACE FORMULA}
In this section, we prove a twisted version of Franke's trace formula \cite[\S7.7]{Franke}. In order to do so we review some results of Franke first.
 
\subsection{Notation}
Let $G$ be a quasi-split connected reductive group. We fix a minimal parabolic subgroup $P_0$ of $G$ containing $B$ and write its Langlands decomposition $P_0=M_0A_0N_0$. Generally, for a standard parabolic subgroup $P$ of $G$, we consider its Langlands decomposition $P=M_PA_PN_P$ such that $A_P\subset A_0$ and $M_0\subset M_P$, where $L_P=M_PA_P$ is the standard Levi subgroup of $P$ and $A_P$ is a maximal spit torus in the center of $L_P$. In particular, $A_0=T$. Write $\check{\mathfrak{a}}_P=X^*(P)\otimes\mathbb{R}$ and $\check{\mathfrak{a}}_P^G$ its subspace of elements whose restriction to $A_G$ are trivial. If $P=P_0$, we denote $\check{\mathfrak{a}}_{0}^G:=\check{\mathfrak{a}}_{P_0}^G$. We let $\check{\mathfrak{a}}_0^{G+}\subset\check{\mathfrak{a}}_0^G$ be the open positive Weyl chamber, and $^+\check{\mathfrak{a}}_0^G$ the open positive cone dual to it. 

\vspace{1pc}
Let $\mathfrak{g}$ be the Lie algebra of $G_\infty$, $\mathfrak{m}_G\subset\mathfrak{g}$ the intersection of kernals for all rational characters of $G$, and $Z(\mathfrak{m}_G)$ the center of the universal enveloping algebra of $\mathfrak{m}_G$. For any standard parabolic subgroup $P$, define the height function $H_P: P(\mathbb{A})\rightarrow\mathfrak{a}_{P}$ such that for any  $x\in P(\mathbb{A})$ and any character $\xi$ of $P$,
\begin{equation}
\prod_v|\xi(x)|_v=e^{\langle\xi,H_P(x) \rangle}.
\end{equation}
We also consider $H_P$ as a function on $G$ by the Iwasawa decomposition. Let $\mathbb{V}^\vee_\lambda$ be as in the last section. Assume $A_G$ acts on $\mathbb{V}^\vee_\lambda$ by a character $\xi_\lambda$, let $\mathcal{I}_\lambda\subset Z(\mathfrak{m}_G)$ be the annihilator of $\xi_\lambda$. For any $G(\mathbb{A}_f)$-module $M$ and $\mu\in\check{\mathfrak{a}}_G$, we denote by $M(\mu)$ the twist of $M$ in which the action of $g\in G(\mathbb{A}_f)$ on $M$ is multiplied by the factor $e^{\langle\mu,H_G(g) \rangle}$.

\vspace{1pc}
Let $R=R_G$ (\emph{resp.} $R^+=R^+_G$, $\Delta=\Delta_G$ ) be the set of roots (\emph{resp.} positive roots, simple positive roots ) with respect to the Borel pair $(B,T)$. If $L$ is a Levi subgroups of $G$, let $\rho_L$ be the half sum of all positive roots of $L$ with respect to $T$, in particular, write $\rho=\rho_G$. For a parabolic subgroup $P$, let $\rho_P\in\check{\mathfrak{a}}_P^G$ be the half sum of all the positive roots of $P$ in $N_P$. For parabolic subgroups $P\subset Q$, there is a natural projection $X^*(T)^+\otimes\mathbb{R}\rightarrow\check{\mathfrak{a}}_P^G\rightarrow\check{\mathfrak{a}}_Q^G$, under which the images of $\rho_G$ are $\rho_P$ and $\rho_Q$. Let $\mathcal{W}_G$ be the Weyl group of $G$ and fix $w_0^G\in\mathcal{W}_G$ a longest element.  We define
\begin{equation}
\mathcal{W}^L: = \{w\in\mathcal{W}_G\ |\ w^{-1}(\alpha)\in R^+,\ \forall \alpha \in R^+\cap R_L\}.
\end{equation}
As \cite[(7)]{Urban}, for any $w\in\mathcal{W}^L$,
\begin{equation}
w\cdot\lambda=w(\lambda+\rho)-\rho=w(\lambda+\rho_P)-\rho_P.
\end{equation} 

There is the Kostant decomposition:
\begin{equation}\label{Kostant}
H^q(\mathfrak{n},\mathbb{V}_\lambda^\vee(\mathbb{C}))=\bigoplus_{w\in\mathcal{W}^L\ l(w)=n-q}{\mathbb{V}^{L, \vee}_{w(\lambda+\rho_P)+\rho_P}}(\mathbb{C}),
\end{equation}
where $\mathfrak{n}$ is the Lie algebra of $N_\infty$, $n=dim (\mathfrak{n})$ and $\mathbb{V}^L_\mu$ is the finite dimensional irreducible algebraic representation of $L$ with highest weight $\mu$. Define
\begin{equation}
\mathcal{W}^L_{Eis}:= \{w\in\mathcal{W}^L\ |\ w^{-1}(\beta^\vee)> 0,\forall\beta\in R_P\}.
\end{equation}
If $\lambda$ is regular and $w\in\mathcal{W}^L_{Eis}$, then the Eisenstein series associated to a class in $H^*(S_L, {\mathbb{V}^{L,\vee}_{w\cdot\lambda}})$ defines an Eisenstein class in $H^*(S_G, \mathbb{V}_\lambda^\vee)$.

\subsection{The Eisenstein spectral sequence}
Let $V_G=C_{umg}^\infty(G(\mathbb{Q})A(\mathbb{R})^0\backslash G(\mathbb{A}))$ be the space of $C^\infty$-functions of uniform moderate growth, and $R_g$ the natural right translation action of $g\in G(\mathbb{A})$ on $V_G$.  Let $A_\lambda$ be the subspace of $V_G$ consisting of the functions which are annihilated by some power of $\mathcal{I}_\lambda$. One is interested in the cohomology group
\begin{equation}\label{original}
 H^*(S_G, \mathbb{V}_\lambda^\vee(\mathbb{C})).
\end{equation}
In \cite{Franke}, Franke proved that it can be computed by $(\mathfrak{m}_G,K_\infty)$ cohomology with coefficients in $A_\lambda$ as
\begin{equation}
\begin{aligned}
  & H^*(\mathfrak{m}_G,K_\infty; C^\infty(G(\mathbb{Q})A(\mathbb{R})^0\backslash G(\mathbb{A}))\otimes \mathbb{V}_\lambda^\vee(\mathbb{C}))(\xi_\lambda)\\
 &= H^*(\mathfrak{m}_G,K_\infty; A_\lambda\otimes \mathbb{V}_\lambda^\vee(\mathbb{C}))(\xi_\lambda).
 \end{aligned}
\end{equation}
Indeed, Franke proved that last cohomology space can be computed as the limit of a spectral sequence which is compiled via the Laurent coefficients of Eisenstein series. 

\vspace{1pc}
let $\mathcal{C}:=\mathcal{C}_G$ be the set of clases of associate parabolic $\mathbb{Q}$-subgroups of $G$. For $\set{P}\in\mathcal{C}_G$, let $V_G(\set{P})$ be the subspace of $V_G$ consists of functions which are negligible along all $Q\notin\set{P}$, that is, the space of functions $\phi\in V_G$, such that for every parabolic subgroup $Q\notin\set{P}$ and every $g\in A_Q(\mathbb{A})\mathbb{K}$, $R_g\phi_{N_Q}$ is orthogonal to the space of cuspidal functions on $L_Q(\mathbb{Q})\backslash L_Q(\mathbb{A})^1$, where $\phi_{N_Q}$ is the constant term of $\phi$ along $N_Q$, defined by
\begin{equation}
\phi_{N_Q}=\int_{N_Q(\mathbb{Q})\backslash N_Q(\mathbb{A})}\phi(ng)dn
\end{equation}
with respect to the normalized Haar measure:
\begin{equation}
\int_{N_Q(\mathbb{Q})\backslash N_Q(\mathbb{A})}dn=1.
\end{equation}
Set $A_{\lambda, \set{P}}:= A_\lambda \cap V_G(\set{P})$, then as $(\mathfrak{g}, K_\infty, G(\mathbb{A}_f))$-modules,
\begin{equation}\label{step1}
A_\lambda = \bigoplus_{\set{P}\in \mathcal{C}}A_{\lambda, \set{P}},
\end{equation}
and the cohomology $(\ref{original})$ equals
\begin{equation}
\bigoplus_{\set{P}}H^*(\mathfrak{m}_G,K_\infty; A_{\lambda,\set{P}}\otimes\mathbb{V}_\lambda^\vee(\mathbb{C}))(\xi_\lambda)
\end{equation}

\vspace{1pc}
For each $\set{P}$, there is a descending filtration (of finite length) $A_{\lambda, \set{P}}^{T,p}$ on the space $A_{\lambda, \set{P}}$. This filtration depends on a finite supported $\mathbb{Z}$-valued function $T$ which is defined on the closure $\overline{\check{\mathfrak{a}}_0^+}$, such that: 
\begin{itemize}
\item[$(T)$]  $T(\mu)<T(\nu)$ if $\mu\neq\nu$ and $\nu\in\mu-\overline{^+\check{\mathfrak{a}}_0}$.
\end{itemize}
The  successive quotients $A_{\lambda, \set{P}}^{T,p}/A_{\lambda, \set{P}}^{T,p+1}$ can be described in term of Eisenstein series. Following \cite[\S6]{Franke}, define $\mathcal{M}^{T,p}_{\lambda, \set{P}}$ to be the set of triples $t=(P,\Lambda,\chi):=(P_t,\Lambda_t,\chi_t)$ with the following properties:
\begin{itemize}
\item[$(a)$] $P\in\set{P}$ is a standard parabolic subgroup;
\item[$(b)$] $\Lambda:A_P(\mathbb{A})/A(\mathbb{R})^0A_P(\mathbb{Q})\rightarrow\mathbb{C}^\times$ is a continuous character. Let $\lambda_t\in(\check{\mathfrak{a}}_P^G)_\mathbb{C}$ be the differential of the archimedean component of $\Lambda$, we assume $Re(\lambda_t)\in\overline{\check{\mathfrak{a}}_P^{+}}$ and $T(Re(\lambda_t))=p$.
\item[$(c)$] $\chi: \mathfrak{Z}(\mathfrak{m}_G):\rightarrow\mathbb{C}^\times$ is a character. 
\item[$(d)$] $\lambda_t\in supp_{t}(\mathcal{I}_\lambda)$, i.e. for any $x\in\mathcal{I}_\lambda$, $\xi(x)(\lambda_t+\chi_t)=0$, where $\xi: \mathfrak{Z(m)}_G\rightarrow S(\mathfrak{t}\cap\mathfrak{m}_G)^{\mathcal{W}_G}$ is the Harish-Chandra isomorphism.
\end{itemize}
For $t,t'\in\mathcal{M}^{T,p}_{\lambda, \set{P}}$, define a morphism from $t$ to $t'$ to be an element of the Weyl set $\Omega(\mathfrak{a}_t,\mathfrak{a}_{t'})$ which maps $\Lambda_t$ to $\Lambda_{t'}$ and $\chi_t$ to $\chi_{t'}$. So $\mathcal{M}^{T,p}_{\lambda, \set{P}}$ is a groupoid. Let $\mathcal{C}^{T,p}_{\lambda, \set{P}}$ be a set of representatives for the isomorphism classes of objects of $\mathcal{M}^{T,p}_{\lambda, \set{P}}$.

\vspace{1pc}
For $t\in\mathcal{M}^{T,p}_{\lambda, \set{P}}$, define $V(t)$ to be the space of square integrable $\mathbb K\cap P(\mathbb{A})$-finite functions $f$ on $P(\mathbb{Q})A_P(\mathbb{R})^0N_P(\mathbb{A})\backslash P(\mathbb{A})$ with the following properties
\begin{itemize}
\item[$(a)$] For any parabolic subgroup $Q\subsetneq P$, $f_{N_P}$ is orthogonal to the space of cuspidal forms on $M_Q(\mathbb{Q})\backslash M_Q(\mathbb{A})$.
\item[$(b)$] $f(ag)=e^{-\langle\lambda_t,H_P(a)\rangle}\Lambda(a)f(g)$ for any $a\in A_P(\mathbb{A})$.
\item[$(c)$] $f$ is a $\chi$-eigenvector of $\mathfrak{Z}(\mathfrak{m}_G)$.
\end{itemize}
We let $W(t)=Ind_P^GV(t)$ be the space of $\mathbb K$-finite functions on the space $P(\mathbb{Q})A_P(\mathbb{R})^0N_P(\mathbb{A})\backslash G(\mathbb{A})$ such that for any $g\in G(\mathbb{A})$, the function $f(xg)$ of $x\in P(\mathbb{Q})A_P(\mathbb{R})^0N_P(\mathbb{A})\backslash P(\mathbb{A})$ belongs to $V(t)$. 

\vspace{1pc}
Let $S(t)$ be the symmetric algebra $S((\check{\mathfrak{a}}_P^G)_\mathbb{C})$, which is the space of polynomials on $({\mathfrak{a}}_P^G)_\mathbb{C}$, and also viewed as the algebra of finite sums of iterated derivatives at $\lambda_t$. $S(t)$ is equipped with the structure of $\mathfrak{a}_P$-module defined by the rule:
\begin{itemize}
\item[(A)] For $\xi\in\mathfrak{a}_P$, $\delta\in S(t)$ and any $\eta\in\mathfrak{a}_P^G$
\begin{equation}
(\xi\delta)(\eta):=e^{\langle\xi,\lambda_t+\rho_{P_t}\rangle}\delta(\eta+\xi).
\end{equation}
\end{itemize}
It is then extended to a $\mathfrak{p}$-module structure by letting $\mathfrak{m}$ and $\mathfrak{n}$ act trivially. $S(t)$ is also equipped with a $P(\mathbb{A}_f)$-module structure by the rule:
\begin{itemize}
\item[(B)] For any $x\in P(\mathbb{A}_f)$ and $\eta\in\mathfrak{a}_P^G$
 \begin{equation}
 (x\delta)(\eta)=e^{\langle\eta+\rho_{P_t},H_P(x)\rangle}\delta(\eta).
 \end{equation}
\end{itemize}
So we get a functor from the groupoid $\mathcal{M}^{T,p}_{\lambda, \set{P}}$ to the category of $(\mathfrak{g},K_\infty,G(\mathbb{A}_f))$-modules, it assigns to $t$ a module 
\begin{equation}
M(t):=W(t)\otimes S(t)=Ind_P^GV(t)\otimes S(t),
\end{equation} 

\vspace{1pc}
For $f\in W(t)$ and $\mu\in(\check{\mathfrak{a}}_P^G)_{\mathbb{C}}$, define the Eisenstein series
\begin{equation}\label{ES}
E(f,\mu):=\sum_{\gamma\in P(\mathbb{Q})\backslash G(\mathbb{Q})}e^{\langle\mu+\rho_P,H_P(\gamma g)\rangle}f(\gamma g).
\end{equation}
Moreover, for $f\otimes\delta\in W(t)\otimes S(t)$, let $\mathbf{MW}\delta E(f,\mu)\in V_G(\set{P})$ be the main value of the Laurent expansion of $\delta E(f,\mu)$ at $\lambda_\infty$ (refer \cite[\S6]{Franke}). For $\set{P}\in\mathcal{C}$, let $\mathcal{C}(\set{P})\subset\mathcal{C}$ be the subset defined by the property:
\begin{itemize}
\item[(P)] $\set{Q}\in\mathcal{C}(\set{P})$, if there is a parabolic $Q\in\set{Q}$ such that $Q$ contains some parabolic subgroup in $\set{P}$.
\end{itemize}
\cite[Theorem 14]{Franke} implies that the quotient $A_{\lambda, \set{P}}^{T,p}/A_{\lambda, \set{P}}^{T,p+1}$ is spanned by the main values $\mathbf{MW}\delta E(f,\mu)$ for all $f\otimes\delta$ in $M(t)$ when $t$ is running over all $\mathcal{M}^{T,p}_{\lambda, \set{Q}}$ and $\set{Q}$ in $\mathcal{C}(\set{P})$. 

\vspace{1pc}
We now state Franke's theorem (\cite[Theorem 19]{Franke}) which gives a spectral squence to compute the cohomology space (\ref{original}):

\begin{thm}[Franke's Eisenstein Spectral Sequence (ESS)]
Let $\lambda$ be a regular algebraic weight, then there is a spectral sequence:
\begin{equation*}
E_1^{p,q}=\bigoplus_{t\in\mathcal{C}^{T,p}_{\lambda, \set{P}}}H^{p+q}(\mathfrak{m}_G,K_\infty; M(t)\otimes\mathbb{V}_\lambda^\vee)\Rightarrow H^{p+q}(\mathfrak{m}_G,K_\infty; A_{\lambda,\set{P}}\otimes\mathbb{V}_\lambda^\vee).
\end{equation*}
Moreover, this spectral sequence degenerates.
\end{thm}

\vspace{1pc}
Now write $K^t_\infty:=K_\infty\cap P_t(\mathbb{R})$. We compute the summand of $E^{p,q}_1$:
\begin{equation}
\begin{aligned}
&\ \ \ \  H^{p+q}(\mathfrak{m}_G,K_\infty; M(t)\otimes\mathbb{V}_\lambda^\vee)\\
&=H^{p+q}(\mathfrak{m}_G,K_\infty; Ind_{P_t}^{G}V(t)\otimes S(t)\otimes\mathbb{V}_\lambda^\vee)\\
&=Ind_{P_t(\mathbb{A}_f)}^{G(\mathbb{A}_f)}H^{p+q}(\mathfrak{m}_G\cap\mathfrak{p}_t,K_\infty^t; V(t)\otimes S(t)\otimes\mathbb{V}_\lambda^\vee)\\
&=\bigoplus_{i+j=p+q}Ind_{P_t(\mathbb{A}_f)}^{G(\mathbb{A}_f)}H^i(\mathfrak{l}_t,K_\infty^t; H^j(\mathfrak{n}_t; \mathbb{V}^\vee_\lambda)\otimes V(t)\otimes S(t)).
\end{aligned}
\end{equation}
Apply the Kostant decomposition (\ref{Kostant}), we have 
\begin{equation}
\begin{aligned}
& H^i(\mathfrak{l}_t,K_\infty^t; H^j(\mathfrak{n}_t; \mathbb{V}^\vee_\lambda)\otimes V(t)\otimes S(t))=\\
&\bigoplus_{w\in\mathcal{W}^{L_t}\ l(w)=n_t-j}H^i(\mathfrak{l}_t,K_\infty^t;{\mathbb{V}^{L, \vee}_{w(\lambda+\rho_{P_t})+\rho_{P_t}}}\otimes V(t)\otimes S(t)).
\end{aligned}
\end{equation}
Using the notation of \cite{Franke}, for any $\Theta\in\check{\mathfrak{a}}_t$, let $\mathbb{C}_\Theta$ be the one dimensional vector space $\mathbb{C}$ on which $x\in\mathfrak{a}_t$ acts by muliplication of $e^{\langle  x,\Theta\rangle}$. Twisting $\mathbb{V}^{L, \vee}_{w(\lambda+\rho_{P_t})+\rho_{P_t}}\otimes V(t)$ by a proper $\mathbb{C}_\Theta$ to make it a trivial $\mathfrak{a}_t$-module, we apply K$\ddot{u}$nneth theorem for each summand in last equation with respect to $\mathfrak{l=m+a}$. Then a standard computation shows that
\begin{equation}
\begin{aligned}
& H^i(\mathfrak{l}_t,K_\infty^t; H^j(\mathfrak{n}_t; \mathbb{V}^\vee_\lambda(\mathbb{C}))\otimes V(t)\otimes S(t))=\\
& \bigoplus_{w\in\mathcal{W}^t_j }H^i(\mathfrak{m}_t,K^t_\infty;\mathbb{V}^{L_t,\vee}_{w(\lambda+\rho_{P_t})+\rho_{P_t}}(\mathbb{C})\otimes V(t))\otimes\mathbb{C}_{\rho_{R_t}+\lambda_t}
\end{aligned}
\end{equation}
where $\mathcal{W}^t_j$ is the subset of $w\in\mathcal{W}^L$, such that $l(w)=n-j$ and the natual projection of $w(\lambda+\rho_{P_t})$ to $\check{\mathfrak{a}}_{P_t}^G$ is $\lambda_t$. Combing all the results above, for $\lambda$ regular, the $E_1$-term, $E_1^{p,q}$, of (ESS) can be computed by
\begin{equation}
\bigoplus_{t\in\mathcal{C}_{\lambda,\set{P}}^{T,p}}\bigoplus_{i+j=p+q}\bigoplus_{w\in\mathcal{W}^t_j }Ind_{P_t(\mathbb{A}_f)}^{G(\mathbb{A}_f)}H^i(\mathfrak{m}_t,K^t_\infty;\mathbb{V}^{L_t,\vee}_{w(\lambda+\rho_{P_t})+\rho_{P_t}}(\mathbb{C})\otimes V(t))\otimes\mathbb{C}_{\rho_{R_t}+\lambda_t}
\end{equation}
Now Franke's theorem implies $H^r(S_G, \mathbb{V}_\lambda^\vee(\mathbb{C}))(\xi_\lambda^{-1})$ equals:
\begin{equation}\label{formula1}
\begin{aligned}
&\bigoplus_{\set{P}}\bigoplus_p\bigoplus_{t\in\mathcal{C}_{\lambda,\set{P}}^{T,p}}H^{r}(\mathfrak{m}_G,K_\infty; M(t)\otimes\V)=\\
&\bigoplus_{\set{P}}\bigoplus_p\bigoplus_{t\in\mathcal{C}_{\lambda,\set{P}}^{T,p}}\bigoplus_{w\in\mathcal{W}^t}Ind_{P_t(\mathbb{A}_f)}^{G(\mathbb{A}_f)}H^{r+l(w)-n_t}(\mathfrak{m}_t,K^t_\infty;\\
&\mathbb{V}^{L_t,\vee}_{w(\lambda+\rho_{P_t})+\rho_{P_t}}(\mathbb{C})\otimes V(t))\otimes\mathbb{C}_{\rho_{P_t}+\lambda_t}
\end{aligned}
\end{equation}
where $\mathcal{W}^t$ is the subset of $w\in\mathcal{W}^L$, such that  the natual projection of $w(\lambda+\rho_{P_t})$ to $\check{\mathfrak{a}}_{P_t}^G$ is $\lambda_t$.

\subsection{Twisted Franke's trace formula}
Let $\lambda\in X^*(T)$ be a regular dominant $\iota$-invariant weight, $f\in\mathcal{H}_p\subset C_c^\infty(G(\mathbb{A}_f))$ an admissible $p$-adic Hecke operator. In this section, we use (\ref{formula1}) to compute the alternating trace 
\begin{equation}\label{al}
tr(\iota\times f\mid H^*(S_G, \mathbb{V}_\lambda^\vee(\mathbb{C}))):=\sum_{r}(-1)^rtr(\iota\times f\mid H^r(S_G, \mathbb{V}_\lambda^\vee(\mathbb{C}))).
\end{equation}
The alternating trace $tr(f\mid H^*(S_G, \mathbb{V}_\lambda^\vee(\mathbb{C})))$ without twisting was computed by Franke and Urban in \cite[\S7.7]{Franke} and \cite[Theorem 1.4.2]{Urban}. Here we have to study how $\iota$ acts on each step from (\ref{step1}) to (\ref{formula1}).

\subsubsection{$\iota$-action on $(\mathfrak{m}_G,K_\infty)$-cohomology}
Consider the complex 
\begin{equation}
C^*(\mathfrak{m}_G,K_\infty; \V):=Hom_{K_\infty}(\wedge^*(\mathfrak{m}_G/\mathfrak{k_{\infty}}),V_G\otimes\V),
\end{equation}
which computes the $(\mathfrak{m}_G,K_\infty)$-cohomology $H^{*}(\mathfrak{m}_G,K_\infty; V_G\otimes\V)$. Let $\iota_*: \mathfrak{m}_G\rightarrow \mathfrak{m}_G$ be the push-forward map induced from $\iota: G(\mathbb{R})\rightarrow G(\mathbb{R})$. Define $\iota: V_G\rightarrow V_G$ by sending $\varphi$ to $\varphi^\iota$, such that for any $[g]\in S_G$, $\varphi^\iota([g])=\varphi([g]^\iota)$. Now $\iota$ acts on $Hom_{K_\infty}(\wedge^q(\mathfrak{m}_G/\mathfrak{k_{\infty}}),V_G\otimes\V)$ by sending $\phi$ to $\phi^\iota$, such that 
\begin{equation}\label{formula2}
\phi^\iota: X\mapsto (\phi((\iota^{-1})_*X))^\iota
\end{equation}
for any $X\in\wedge^q(\mathfrak{m}_G/\mathfrak{k_{\infty}})$. It is easy to check that the action is well defined up to homotopy. 

\vspace{1pc}
Let $\alpha: C^*(\mathfrak{m}_G,K_\infty; \V)\rightarrow\Omega^*(S_G,\V)$ be the morphism of complexes which induces  $H^q(\mathfrak{m}_G,K_\infty; V_G\otimes\V)\cong H^q(S_G,\V)$. Concretely, for $\phi\in C^q(\mathfrak{m}_G,K_\infty; \V)$, $\alpha$ assigns it a $q$-th differential form $\tau_\phi$, such that for any $g\in G(\mathbb{A})^1$,
\begin{equation}
\tau_\phi(\bar{v}_1\wedge\cdots\wedge\bar{v}_q)([g])=\phi(\bar{v}_1([e]),\cdots, \bar{v}_q([e]))([g]),
\end{equation}
where $\bar{v}$ indicates a left invariant vector field on $S_G$ and $\bar{v}([e])$ its value at $[e]\in S_G$. Since $\phi(\bar{v}_1([e]),\cdots, \bar{v}_q([e]))\in V_G\otimes\V$, $\phi(\bar{v}_1([e]),\cdots, \bar{v}_q([e]))([g])$ means to evaluate its first component at $[g]$. Compare with \S3.1.7, it is easy to check that the $\iota$-action on $H^q(\mathfrak{m}_G,K_\infty; V_G\otimes\V)$ is compatible with the $\iota$-action on $H^q(S_G, \V)$ defined in \S3.1.

\subsubsection{image of $A_{\lambda,\set{P}}$ under $\iota$}
Apparently that $A_\lambda$ is stable under $\iota$, we now study the behavior of decomposition (\ref{step1}) under $\iota$. Let $A^\iota_{\lambda,\set{P}}$ be the image of $A_{\lambda,\set{P}}$ under $\iota$. Given an associate class $\set{P}$, let $\set{P^\iota}$ be the class whose elements are $P^{\iota}$ for all $P\in\set{P}$. Apparently the map $\set{P}\mapsto\set{P^\iota}$ permutes the assoicate classes, let $\set{P}^\iota:=\set{P^\iota}$.

\begin{lem}
\begin{equation}
A^\iota_{\lambda,\set{P}}=A_{\lambda,\set{P^{\iota^{-1}}}} 
\end{equation}
\end{lem}
\begin{proof}
Given $\phi\in A_{\lambda, \set{P}}$, we have to show that, for any parabolic $Q\notin\set{P^{\iota^{-1}}}$ and  $g\in A_Q(\mathbb{A})\mathbb K$, $R_g(\phi^\iota)_{N_Q}\perp L^2_{cusp}(L_Q(\mathbb{Q})\backslash L_Q(\mathbb{A})^1)$. Let $dn_Q$ be the normalized Haar measure on $N_Q$, then $dn_{Q^\iota}$ is same to the Haar measure on $N_Q^\iota=N_{Q^\iota}$ induced by the map $\iota: N_Q\rightarrow N_Q^\iota=N_{Q^\iota}$. Now a direct computation shows that 
\begin{equation}
(\phi^\iota)_{N_Q}=(\phi_{N_{Q^\iota}})^\iota.
\end{equation}
Then $g^\iota\in A_{Q^\iota}(\mathbb{A})\mathbb K$, and
\begin{equation}
R_g(\phi^\iota)_{N_Q}=R_g((\phi_{N_{Q^\iota}})^\iota)=(R_{g^\iota}\phi_{N_{Q^\iota}})^\iota.
\end{equation}
Noting that the restriction of $\iota$ on $L^2_{cusp}(L_Q(\mathbb{Q})\backslash L_Q(\mathbb{A})^1)$ identify its image with $L^2_{cusp}(L_{Q^{\iota^{-1}}}(\mathbb{Q})\backslash L_{Q^{\iota^{-1}}}(\mathbb{A})^1)$, we have
\begin{equation}
L^2_{cusp}(L_{Q^{\iota}}(\mathbb{Q})\backslash L_{Q^{\iota}}(\mathbb{A})^1)^\iota=L^2_{cusp}(L_Q(\mathbb{Q})\backslash L_Q(\mathbb{A})^1).
\end{equation}
Now the conclusion follows.
\end{proof}
This lemma implies that, in the formula (\ref{formula1}), only those summand parameterized by $\set{P}=\set{P}^\iota$ will contribute to the twisted trace.

\subsubsection{on Eisenstein series}
Consider an associate class $\set{P}$ and $t=(P,\Lambda,\chi)\in\mathcal{M}_{\lambda,\set{P}}^{T,p}$ for some $p$. For $\varphi\in W(t)$ and $\mu\in(\check{\mathfrak{a}}_P^G)_\mathbb{C}$, let $E_t(\varphi,\mu)(g):=E(\varphi,\mu)(g)$ be the Eisenstein series defined in (\ref{ES}), then $A_\lambda$ is spanned by the principal values of derivatives of all such $E(\varphi,\mu)$. 

\begin{lem}
\begin{equation}
E_t(\varphi,\mu)^\iota=E_{t^\iota}(\varphi^\iota,\mu^{\iota^{-1}})
\end{equation}
where we define 
\begin{itemize}
\item[$(a)$] $\Lambda^{\iota^{-1}}: A_{P^{\iota^{-1}}}(\mathbb{A})/A(\mathbb{R})^0A_{P^{\iota^{-1}}}(\mathbb{Q})\rightarrow\mathbb{C}^\times$ as $\Lambda^{\iota^{-1}}(a):=\Lambda(a^{\iota})$;
\item[$(b)$] $\chi^{\iota^{-1}}:\mathfrak{m}_G\rightarrow\mathbb{C}^\times$ as $\chi^{\iota^{-1}}(x):=\chi(x^{\iota})$;
\item[$(c)$] $t^\iota:=(P^{\iota^{-1}}, \Lambda^{\iota^{-1}}, \chi^{\iota^{-1}})$
\item[$(d)$] $\varphi^\iota(g):=\varphi(g^\iota)$. Then $\varphi^\iota\in W(t^\iota)$.
\item[$(e)$] $\mu^{\iota^{-1}}\in(\check{\mathfrak{a}}_{P^{\iota^{-1}}}^G)_\mathbb{C}$, as a character, it is defined by $\mu^{\iota^{-1}}(a):=\mu(a^\iota)$. In particular, it induces a homomorphism $\iota: S((\check{\mathfrak{a}}_P^G)_\mathbb{C})\rightarrow S((\check{\mathfrak{a}}_{P^{\iota^{-1}}}^G)_\mathbb{C})$.
\end{itemize}
So we have a homomorphism between vector spaces
\begin{equation}
\iota: M(t)\rightarrow M(t^\iota)
\end{equation}
such that for $\varphi\otimes\delta\in M(t)=W(t)\otimes S(t)$
\begin{equation}
\mathbf{MW}_{\lambda_t}\delta E_t(\varphi,\mu)^\iota=\mathbf{MW}_{\lambda_{t^\iota}}\delta^\iota E_{t^\iota}(\varphi^\iota,\mu^{\iota^{-1}}).
\end{equation}
Moreover, $\iota$ induces an homomorphism between the $(\mathfrak{m}_G, K_\infty)$-cohomology group, 
\begin{equation}
\iota: H^*(\mathfrak{m}_G, K_\infty; M(t)\otimes\V)\rightarrow H^*(\mathfrak{m}_G, K_\infty; M(t^\iota)\otimes\V)
\end{equation}
\end{lem}
\begin{proof}
$(4.3.9)$ and $(4.3.11)$ are from definition directly. To show $(4.3.12)$, one only has to notice that $\iota$ is compatible with (\ref{formula2}).
\end{proof} 

\begin{remark}
The definition of $\mathbf{MW}_{\lambda_t}$ actually depends on the choice of a regular element $\xi_t\in\check{\mathfrak{a}}_P^G$, so here $(4.3.11)$ depends on the choice $\xi_{t^\iota}=\xi_t^{\iota^{-1}}$. However, just as the situation in \cite{Franke}, it does not matter our goal.
\end{remark}

Recall that the quotient $A_{\lambda, \set{P}}^{T,p}/A_{\lambda, \set{P}}^{T,p+1}$ is spanned by the elements of the form $\mathbf{MW}_{\lambda_t}\delta E(f,\mu)$. As observed by Franke and Schwermer in \cite{FS}, the only relations between these vectors are the relations provided by the functional equation of the Eisenstein series for singular $\lambda_t$. So if $t\cong t^\iota$ in $\cup_p\mathcal{M}_{\lambda,\set{P}}^{T,p}$, then the image of $H^*(\mathfrak{m}_G, K_\infty; M(t)\otimes\V)$ and $H^*(\mathfrak{m}_G, K_\infty; M(t^\iota)\otimes\V)$ in $H^*(\mathfrak{m}_G, K_\infty; A_{\lambda,\set{P}}\otimes\V)$ coincide. This implies that, in the first step of (\ref{formula1}), only those terms with $t\cong t^\iota$ will contribute to the twisted trace.

\subsubsection{}
For a standard parabolic subgroup $P$ and the associate class $\set{P}$ containing $P$, consider a triple $(P,\mu,\chi)\in\mathcal{M}_{\lambda,\set{P}}^{T,T(\mu)}$, and let $n_{P}(\mu)$ be the cardinality of the isomorphism class containing $(P,\mu,\chi)$. $n_{P}(\mu)$ is the number of Weyl chambres to which $\mu$ belongs, in particular, if $\mu$ is regular, then $n_P(\mu)=1$. Define:
\begin{equation}
\Upsilon:=\set{\mu\in\check{\mathfrak{a}}_0\mid pr_{\check{\mathfrak{a}}_0\rightarrow\check{\mathfrak{a}}_P^G}(\mu+\rho) \text{ is regular in }\check{\mathfrak{a}}_P^G, \forall P},
\end{equation}
$\Upsilon$ is dense in $\check{\mathfrak{a}}_0$. From now on, assume that 
\begin{itemize}
\item[$(R)$] $\lambda\in\Upsilon$.
\end{itemize}
With assumption $(R)$, $\mathcal{W}^t=\emptyset$ for any $t\in\mathcal{M}_{\lambda,\set{P}}^{T,p}$, unless $\lambda_t$ is regular. In particular, in the alternating trace (\ref{al}), only those $t$ with $\lambda_t$ regular will contribute to the trace. In this case, $n_p(\lambda_t)=1$, and $t\cong t^\iota\Leftrightarrow t=t^\iota$.

\vspace{1pc}
Combing the discussion in last several sections, we compute (\ref{al}) $=$
\begin{equation}\label{4314}
\begin{aligned}
\sum_{t=(P,\Lambda,\chi)\cong t^\iota}&\frac{1}{n_P(\lambda_t)}tr(\iota\times f \mid ( Ind_{P_t(\mathbb{A}_f)}^{G(\mathbb{A}_f)}\bigoplus_{w\in\mathcal{W}^t}\\
&H^{*+l(w)-n_t}(\mathfrak{m}_t,K^t_\infty;\mathbb{V}^{L_t,\vee}_{w(\lambda+\rho_{P_t})+\rho_{P_t}}(\mathbb{C})\otimes V(t))\otimes\mathbb{C}_{\rho_{R_t}+\lambda_t})(\xi_\lambda))\\
=\sum_{P=P^{\iota}}\sum_{\Lambda=\Lambda^\iota}&\sum_{\chi=\chi^\iota}(-1)^{n_P}tr(\iota\times f_L\mid\\
&(\bigoplus_{w\in\mathcal{W}^t}H^{*+l(w)}(\mathfrak{m}_P,K_\infty^t;\mathbb{V}^{L,\vee}_{w(\lambda+\rho_{P})+\rho_{P}}(\mathbb{C})\otimes V(t)))(\xi_{w(\lambda+\rho_{P})+\rho_{P}})),
\end{aligned}
\end{equation}
where, on the right side of this equation, $t=(P,\Lambda,\chi)$, $\mu$ is the differential of the Archimedean part of $\Lambda$ and $P=LN=MAN$ the Langlands decomposition. For a Levi subgroup $L$, $f_L$ is defined by
\begin{equation}\label{lower}
f_L(l)=e^{\langle\rho_P,H_P(l)\rangle}\int_{\mathbb K_f}\int_{N_P(\mathbb{A}_f)}f(klnk^{-1})dndk,
\end{equation} 
where the Haar measures are normalized with respect to the Iwasawa decomposition as in \cite[\S7.7]{Franke}. The twisting factor appears in (\ref{lower}) since we have twisted the character $\xi_\lambda$ in the $(\mathfrak{m}_G, K_\infty)$-cohomology. The twisting term $\mathbb{C}_{\rho_P+\mu}$ disappears, since as a $\mathfrak{a}_P$-module it does not contrube to the trace. Now combing all $\chi$ and all the finite parts $\Lambda_f$ of $\Lambda$ in the summation, by the definition $(a)$ of $V(t)$, (\ref{4314}) equals
\begin{equation}\label{formula4}
\begin{aligned}
&\sum_{P=P^{\iota}}\sum_{\mu=\mu^\iota}(-1)^{n_P}tr(\iota\times f_L\mid(\bigoplus_{w\in\mathcal{W}^t}H^{*+l(w)}(\mathfrak{m}_P,K_\infty^L;\\
&\mathbb{V}^{L,\vee}_{w(\lambda+\rho_{P})+\rho_{P}}(\mathbb{C})\otimes L^2_{disc}(A_P(\mathbb{R})^0L_P(\mathbb{Q})\backslash L_P(\mathbb{A}))))(\xi_{w(\lambda+\rho_{P})+\rho_{P}}))
\end{aligned}
\end{equation}

\subsubsection{}
Now we study the action of $\iota$ on the direct sum over Weyl elements in last formula. For every $\varphi\in \mathbb{V}^{L}_{w(\lambda+\rho_{P})+\rho_{P}}(\mathbb{C})\subset C(L_P(\mathbb{Q})\backslash L_P(\mathbb{A}))$, $\varphi^\iota(x)=\varphi(x^{\iota^{-1}})$. It is easy to check that \begin{equation}
(\mathbb{V}^{L}_{w(\lambda+\rho_{P})+\rho_{P}}(\mathbb{C}))^\iota = \mathbb{V}^{L}_{(w(\lambda+\rho_{P})+\rho_{P})^\iota}(\mathbb{C}).
\end{equation}

\begin{lem} 
Group $G$ and operator $\iota$ as before, $\iota$ acts on the Weyl group $\mathcal{W}=N_G(T)/T$ via $[x]\mapsto [x^\iota]$ for any $x \in N_G(T)$. Then
\begin{itemize}
\item[$(1)$] Let $S_\alpha$ be a simple reflection in $\mathcal{W}$ corresponding to a simple root $\alpha$, then $(S_\alpha)^\iota=S_{\alpha^\iota}$. In particular, $\iota$ preserves the length of a Weyl element.
\item[$(2)$] For any $w\in\mathcal{W}$, $(w(\alpha))^\iota=w^\iota(\alpha^\iota)$. 
\item[$(3)$] Let $P\in\mathcal{P}$, then $\rho_P^\iota=\rho_P$.
\end{itemize}
In particular
\begin{equation}\label{formula3}
(\mathbb{V}^{L}_{w(\lambda+\rho_{P})+\rho_{P}}(\mathbb{C}))^\iota=\mathbb{V}^{L}_{w^\iota(\lambda+\rho_{P})+\rho_{P}}(\mathbb{C}).
\end{equation}
Moreover, if $\lambda$ is regular, $(\mathbb{V}^{L}_{w(\lambda+\rho_{P})+\rho_{P}}(\mathbb{C}))^\iota=\mathbb{V}^{L}_{w(\lambda+\rho_{P})+\rho_{P}}(\mathbb{C})$ if and only if $w^\iota=w$.
\end{lem}
\begin{proof}
The proof is straightforward. Concretely speaking, $(1)$ follows from the fact that $S_\alpha$ is the only nontrivial element in $N_{G_\alpha}(T)/T$ (see, e.g. \cite[IV]{Milne}) and $(S_\alpha)^\iota$ is a non-trivial element in $N_{G_{\alpha^\iota}}(T)/T$. $(2)$ follows from a direct computation: let $[x]$ be a representative of $w$, for any $t\in T$, $(w(\alpha))^\iota(t)=w(\alpha)(t^{\iota^{-1}})=\alpha(x^{-1}t^{\iota^{-1}x})$ and $w^\iota(\alpha^\iota)(t)=\alpha^{\iota}((x^\iota)^{-1}tx^\iota)=\alpha(x^{-1}t^{\iota^{-1}x})$. $(3)$ follows from the definition that $\rho_P(t):=det( Ad(t)|_{\mathfrak{n}_P})^{1/2}$(see, e.g. \cite[III]{BW}) and the commutative diagram $ad(t)\circ\iota=\iota\circ ad(t^\iota)$. Finally, one deduces (\ref{formula3}) from $(1)-(3)$ directly.
\end{proof}

\vspace{1pc}
Now let $\mathcal{W}^{t,\iota}$ be the subset of $\mathcal{W}^{t}$ consisting of elements which are invariant under $\iota$. Apply the lemma, (\ref{formula4}) equals 
\begin{equation}\label{formula5}
\begin{aligned}
&\sum_{P=P^{\iota}}\sum_{\mu=\mu^\iota}\sum_{{w\in\mathcal{W}^{t,\iota}}}(-1)^{n_P-l(w)}tr(\iota\times f_L\mid(H^{*}(\mathfrak{m}_P,K_\infty^L;\\
&\mathbb{V}^{L,\vee}_{w(\lambda+\rho_{P})+\rho_{P}}(\mathbb{C})\otimes L^2_{disc}(A_P(\mathbb{R})^0L_P(\mathbb{Q})\backslash L_P(\mathbb{A}))))(\xi_{w(\lambda+\rho_{P})+\rho_{P}}))
\end{aligned}
\end{equation}
Recall that for any $w\in\mathcal{W}^t$, $w(\lambda+\rho)=\mu$. So when $\mu$ is running over all classes $t\in\mathcal{M}_{\lambda,\set{P}}^{T}$, $w$ is running over $\mathcal{W}^{L}_{Eis}$. Let $\mathcal{W}^{L,\iota}_{Eis}$ be the subset of $\mathcal{W}^{L}_{Eis}$ consisting of elements which are invariant under $\iota$, the previous formula equals：
\begin{equation}
\begin{aligned}
&\sum_{P=P^{\iota}}\sum_{w\in\mathcal{W}_{Eis}^{L,\iota}}(-1)^{n_P-l(w)}tr(\iota\times f_L\mid(H^{*}(\mathfrak{m}_P,K_\infty^L;\\
&\mathbb{V}^{L,\vee}_{w(\lambda+\rho_{P})+\rho_{P}}(\mathbb{C})\otimes L^2_{disc}(A_P(\mathbb{R})^0L_P(\mathbb{Q})\backslash L_P(\mathbb{A}))))(\xi_{w(\lambda+\rho_{P})+\rho_{P}}))\\
&=\sum_{P=P^{\iota}}\sum_{w\in\mathcal{W}_{Eis}^{L,\iota}}(-1)^{n_P-l(w)}tr(\iota\times f_L\mid(H^{*}(\mathfrak{m}_P,K_\infty^L;\\
&\mathbb{V}^{L,\vee}_{w(\lambda+\rho_{P})+\rho_{P}}(\mathbb{C})\otimes L^2_{cusp}(A_P(\mathbb{R})^0L_P(\mathbb{Q})\backslash L_P(\mathbb{A}))))(\xi_{w(\lambda+\rho_{P})+\rho_{P}}))
\end{aligned}
 \end{equation}
where the equality holds since for $\lambda\in\Upsilon$, $w(\lambda+\rho_{P})+\rho_{P}$ is regular. 

\vspace{1pc}
Finally, combing all the computation above, it is easy to deduces:
\begin{thm}[Twisted Franke's trace formula] Assume $\lambda$ is regular in $\Upsilon$, then for any $f\in C_c^\infty(G(\mathbb{A}_f))$,
\begin{eqnarray*}
& &tr^{st}(\iota\times f | H^* (S_G, \mathbb{V}^\vee_\lambda(\mathbb{C})))\\
&=&\sum_{P=P^\iota}\sum_{\substack{w\in\mathcal{W}^{L,\iota}_{Ein}}}(-1)^{l(w)+n_P}tr^{st}(\iota\times f_L | H^*_{cusp}(S_L, \mathbb{V}^{L, \vee}_{w(\lambda+\rho_{P})+\rho_{P}}))
\end{eqnarray*}
where the notation $st$ indicates that we are using the standard Hecke action.
\end{thm}

\subsection{Cuspidal decomposition formula}
Let $\lambda=\lambda^{alg}\epsilon$ be an arithmetic regular dominant weight in $\mathfrak{X}^\iota$, define
\begin{equation}
I^{cl}_G(\iota\times f, \lambda): = tr^*(\iota\times f | H^*(S_G, \mathbb{V}^\vee_{\lambda^{alg}}(L,\epsilon))),
\end{equation}
 where $*$ indicates that we are using the $*$-action defined in Section 3.1.

\begin{lem}
Assume $f=f^p\otimes u_t \in\mathcal{H}_p'$, then
\begin{equation}
I^{cl}_G(\iota\times f, \lambda)=\lambda(\xi(t))tr^{st}(\iota\times f | H^*(S_G, \mathbb{V}^{\vee}_{\lambda^{alg}}(\mathbb{C},\varepsilon))).\end{equation}
\end{lem}
Just like \cite[4.5.1 (29)]{Urban}, this is a direct consequence of (\ref{twistfactor}) and \cite[Lemma 4.3.8]{Urban}.

\vspace{1pc}
For any $f\in\mathcal{H}_p$, define the classical cuspidal alternating twisted trace by:
\begin{equation}
I^{cl}_{G,0}(\iota\times f, \lambda):=meas(K^p)\lambda(\xi(t))tr^{st}(\iota\times f| H^*_{cusp}(S_G(K_f),\mathbb{V}^\vee_\lambda(\mathbb{C},\epsilon)))
\end{equation}
if $f\in\mathcal{H}_p(K^p)$. This is well defined since $\iota$ is well defined on the cuspidal cohomology.

\vspace{1pc}
For any $w\in \mathcal{W}^L$, $f\in\mathcal{H}_p(G)$, define $f^{reg}_{L, w}\in\mathcal{H}_p(L)$ as below. For $f=f^p\otimes u_t \in \mathcal{H}_p(G)$, define:
\begin{equation}
f^{reg}_{L, w}: = e^{\langle\rho_P,H_P(l)\rangle}\epsilon_{\xi,w}(t)(f^p)'_L\otimes u_{wtw^{-1}},
\end{equation}
where the factor $\epsilon_{\xi,w}(t):=\xi(t)^{w^{-1}(\rho_P)+\rho_P}|t^{w^{-1}(\rho_P)+\rho_P}|_p$, which is trivial according to our choice of $\xi$ as (\ref{split}), $(f^p)'_L$ is the usual non-normalized constant term
\begin{equation}
(f^p)'_L(l^p)=\int_{\mathbb K_f^p}\int_{N(\mathbb{A}_f^p)}f_p(k^pl^pn^p(k^p)^{-1})dk^pdn^p.
\end{equation} 
For general $f$, the definition is given by linear extension.

\vspace{1pc}
\begin{thm}\label{cuspidalformula}
Let $\lambda$ arithmetic and regular such that $\lambda^{alg}\in\Upsilon$, then for any $f$ as above, $ I^{cl}_G(\iota\times f, \lambda)$ equals
\begin{equation*}
\sum_{P=P^\iota}\sum_{\substack{v\in\mathcal{W}^{L,\iota}_{Ein}}}\sum_{\substack{w\in\mathcal{W}^{L,\iota}}}(-1)^{l(v)+n_P}\xi(t)^{\lambda-w^{-1}v*\lambda}I^{cl}_{L,0}(\iota\times f^{reg}_{L,w},v(\lambda+\rho_{P})+\rho_{P})
\end{equation*}
\end{thm}
\begin{proof}
The proof is essentially same to \cite[Lemma 4.6.2]{Urban}. Since both sides of the equation are linear on $f$, it is innocuous to assume that $f=f^p\otimes u_t$ and $f^p=1_{K^p}$. If $\lambda=\lambda^{alg}\varepsilon$, the finite order character $\varepsilon$ simply appears in every step of the proof by multipling a twisting factor, so we only have to deal with the case  that $\lambda=\lambda^{alg}$ algebraic. By the twisted Franke's trace formula and lemma 4.4.1, $ I^{cl}_{G}(\iota\times f, \lambda)$ equals:
\begin{equation}\label{447}
\lambda(\xi(t))\sum_{P=P^\iota}\sum_{\substack{v\in\mathcal{W}^{L,\iota}_{Ein}}}(-1)^{l(v)+n_P}tr^{st}(\iota\times f_L|H^*_{cusp}(S_L, \mathbb{V}^{L,\vee}_{v\cdot\lambda+2\rho_P})).
\end{equation}
For group $H=L, N$ or $P$, write $K^p_H:=K^p\cap H(\mathbb{A}_f^p)$. Then for $f=1_{K^p}\otimes u_t$, we have 
\begin{equation}
f_L(l)=meas(K^p)meas(K_N^p)1_{K_L^p}\otimes(u_t)_L.
\end{equation}
So (\ref{447}) equals
\begin{equation}
\begin{aligned}
&\lambda(\xi(t))\sum_{P=P^\iota}\sum_{\substack{v\in\mathcal{W}^{L,\iota}_{Ein}}}(-1)^{l(v)+n_P}meas(K^p)meas(K_N^p)\\
& \cdot tr^{st}(\iota\times1_{K_L^p}\otimes(u_t)_L |H^*_{cusp}(S_L, \mathbb{V}^{L,\vee}_{v\cdot\lambda+2\rho_P}))\\
&=\lambda(\xi(t))\sum_{P=P^\iota}\sum_{\substack{v\in\mathcal{W}^{L,\iota}_{Ein}}}(-1)^{l(v)+n_P}meas(K^p)meas(K_P^p)\\
&\cdot tr^{st}(\iota\times(u_t)_L |H^*_{cusp}(S_L, \mathbb{V}^{L,\vee}_{v\cdot\lambda+2\rho_P})^{K^p_L}).
\end{aligned}
\end{equation}
Write $\mu:=v\cdot\lambda+2\rho_P$ and $\sigma_\mu:=H^*_{cusp}(S_L, \mathbb{V}^{L,\vee}_\mu(\mathbb{C}))^{K_L^p}$. Since $\iota$ is well defined on $\sigma_\mu$, viewing $Ind_{P(\mathbb{Q}_p)}^{G(\mathbb{Q}_p)}\sigma_\mu$ as the restriction of $Ind_{P(\mathbb{Q}_p)\rtimes\langle\iota\rangle}^{G(\mathbb{Q}_p)\rtimes\langle\iota\rangle}\sigma_\mu$ to $G(\mathbb{Q}_p)$, we have 
\begin{equation}
tr^{st}(\iota\times (u_t)'_L | \sigma_\mu)=tr^{st}(\iota\times u_t|Ind_{P(\mathbb{Q}_p)}^{G(\mathbb{Q}_p)}\sigma_\mu).
\end{equation}
According to the decomposition
\begin{equation}
G(\mathbb{Q}_p)=\bigsqcup_{w\in\mathcal{W}^L}P({\mathbb{Q}_p})wI,
\end{equation}
there is
\begin{equation}
(Ind_{P(\mathbb{Q}_p)}^{G(\mathbb{Q}_p)}\sigma_\mu)^I\cong(\sigma_\mu^{I_L})^{\mathcal{W}^L},
\end{equation}
where the isomorphism is given by $\phi\mapsto(\phi(w))_{w\in\mathcal{W}^L}$. In particular, $\iota$ acts on the right side by sending $(\phi(w))$ to $(\phi(w^\iota))$. Let $\mathcal{W}^{L,\iota}$ be the subset of $\mathcal{W}^L$ consisting of elements which are invariant under $\iota$. Write $N_w:=N\cap w^{-1}Nw$ and $I_L:=I\cap L(\mathbb{Q}_p)=wIw^{-1}\cap L(\mathbb{Q}_p)$, (4.4.8) equals
\begin{equation}
\begin{aligned}
&\lambda(\xi(t))\sum_{P=P^\iota}\sum_{\substack{v\in\mathcal{W}^{L,\iota}_{Ein}}}(-1)^{l(v)+n_P}meas(K^p)meas(K_P^p)\\
&\cdot|\rho_P(t)|_ptr^{st}(\iota\times ItI\mid(Ind_{P(\mathbb{Q}_p)}^{G(\mathbb{Q}_p)}\sigma_\mu)^I)\\
&=\lambda(\xi(t))\sum_{P=P^\iota}\sum_{\substack{v\in\mathcal{W}^{L,\iota}_{Ein}}}(-1)^{l(v)+n_P}meas(K^p)meas(K_P^p)|\rho_P(t)|_p\\
&\sum_{w\in\mathcal{W}^{L,\iota}}[\mathcal{N}_w(\mathbb{Z}_p):t\mathcal{N}_w(\mathbb{Z}_p)t^{-1}]tr^{st}(\iota\times I_Lwtw^{-1}I_L\mid \sigma_\mu^{I_L})
\end{aligned}
\end{equation}
Noting that
\begin{equation}
\begin{aligned}
& tr^{st}(\iota\times I_Lwtw^{-1}I_L\mid \sigma_\mu^{I_L})\\
&= tr^{st}(\iota\times I_Lwtw^{-1}I_L\mid H^*_{cusp}(S_L(K^p_LI_L), \mathbb{V}^{L,\vee}_{\mu}))\\
&= \frac{tr^{st}(\iota\times (1_{K^p})_L\otimes u_{wtw^{-1}}\mid H^*_{cusp}(S_L(K^p_LI_L), \mathbb{V}^{L,\vee}_{\mu}))}{meas(K^p)meas(K_P^p)},
\end{aligned}
\end{equation}
and 
\begin{equation}
[\mathcal{N}_w(\mathbb{Z}_p):t\mathcal{N}_w(\mathbb{Z}_p)t^{-1}]=|(w^{-1}(\rho_P)+\rho_P)(t)|_p^{-1},
\end{equation}
$(4.4.12)$ equals
\begin{equation}
\begin{aligned}
&\lambda(\xi(t))\sum_{P=P^\iota}\sum_{\substack{v\in\mathcal{W}^{L,\iota}_{Ein}}}(-1)^{l(v)+n_P}|\rho_P(t)|_p\sum_{w\in\mathcal{W}^{L,\iota}}\\
&[\mathcal{N}_w(\mathbb{Z}_p):t\mathcal{N}_w(\mathbb{Z}_p)t^{-1}]tr^{st}(\iota\times (1_{K^p})_L\otimes u_{wtw^{-1}}\mid \sigma_\mu^{I_L})\\
&=\sum_{P=P^\iota}\sum_{\substack{v\in\mathcal{W}^{L,\iota}_{Ein}}}(-1)^{l(v)+n_P}|\rho_P(t)|_p\sum_{w\in\mathcal{W}^{L,\iota}}\lambda(\xi(t))\\
& w^{-1}(\mu)(\xi(t))^{-1}|(w^{-1}(\rho_P)+\rho_P)(t)|_p^{-1}I^{cl}_{L,0}(\iota\times (1_{K^p})_L\otimes u_{wtw^{-1}},\mu)\\
&=\sum_{P=P^\iota}\sum_{\substack{v\in\mathcal{W}^{L,\iota}_{Ein}}}\sum_{w\in\mathcal{W}^{L,\iota}}(-1)^{l(v)+n_P}\xi(t)^{\lambda-w^{-1}(\mu)+w^{-1}(\rho_P)+\rho_P}\\
&\cdot|\rho_P(t)|_pI^{cl}_{L,0}(\iota\times \epsilon_{\xi,w}(t)(1_{K^p})_L\otimes u_{wtw^{-1}},\mu)\\
&=\sum_{P=P^\iota}\sum_{\substack{v\in\mathcal{W}^{L,\iota}_{Ein}}}\sum_{\substack{w\in\mathcal{W}^{L,\iota}}}(-1)^{l(v)+n_P}\xi(t)^{\lambda-w^{-1}v*\lambda}I^{cl}_{L,0}(\iota\times f^{reg}_{L,w},\mu).
\end{aligned}
\end{equation}
Here in the last two equations, we used (\ref{split}) again and substituted $\mu=v\cdot\lambda+2\rho_P$. This completes the proof.
 \end{proof}

\vspace{3pc}
\section{TWISTED FINITE SLOPE CHARACTER DISTRIBUTIONS}

\subsection{twisted finite slope character distributions}
In this section, we define the notion of twisted finite slope character distribution, which is a twisted version of Urban's finite slope character distributions in \cite[\S4.1.10]{Urban}.
 
\begin{defn}
Let $\iota$ be an automorphism of $G$ with finite order $l$, $L$ a finite extension of $\mathbb{Q}_p$ in $\overline{\mathbb{Q}}_p$. An $L$-valued $\iota$-twisted finite slope character distribution ($\iota$-twisted FSCD) is a $\mathbb{Q}_p$-linear map $J: \mathcal{H}_p'\rightarrow L$, such that for any $\iota$-invariant finite slope overconvergent representation $\pi$ of $\mathcal{H}_p$, there is a set of $l$ integers $\bar{m}_J(\pi):=\{m_{J,i}(\pi)| i=1, \dots, l\}$, satisfying:
\begin{itemize}
 \item[$(1)$] for any $t\in T^{++}$, $h\in\mathbb{Q}$ and $K^p$, there are finitely many $\pi$ of slope $\leqslant h$ and such that $\bar{m}_J(\pi)\neq 0$, $\pi^{K^p}\neq 0$.
\item[$(2)$] for any $f\in\mathcal{H}_p'$, $$J(f)=\sum_{\pi}\sum_{i}m_{J,i}(\pi) J_{\tilde\pi_i}(f)$$
\end{itemize}
\end{defn}
For any irreducible $\iota$-invariant finite slope representation $\pi$, we define the multiplicity of $\pi$ in $J$ by 
\begin{equation}
m_J(\pi):=\sum_im_{J,i}(\pi).
\end{equation}

\vspace{1pc}
We say $J$ is effective if it is non-trivial and all its coefficients $m_{J,i}(\pi)$ are non-negative. Given a twisted FSCD $J$, for any $f\in\mathcal{H}'_p$, define the Fredholm determinant of $f$ associated to $J$ by \begin{equation}P_J(X, f):=\prod_\pi\prod_{i}\det(1-X\tilde\pi_i(\iota\times f))^{m_{J,i}(\pi)} \end{equation}
As \cite[Lemma 4.1.12]{Urban}, $P_J(X, f)$ is an entire power series for all $f=f^p\otimes u_t\in\mathcal{H}_p'$, if and only if  $J$ is effective.

\vspace{1pc}
If $J$ is effective, for any $\iota$-invariant $K^p$, define $V_J(K^p)$ to be the completion of 
\begin{equation}
\bigoplus_{\pi}\bigoplus_i(V_{\tilde\pi_i}^{K^p})^{m_{J,i}(\pi)}
\end{equation}
under the super norm $\|\sum_iv_i\|:=sup_i\|v_i\|$. It is a $p$-adic Banach space over $\mathbb{C}_p$ with an action of $^\iota\mathcal{H}_p(K^p)$ such that an element $f$ in $\mathcal{H}_p'(K^p)$ acting completely continuously and 
\begin{equation}
J(f)=meas(K^p)tr(\iota\times f|V_J(K^p)).
\end{equation}
This observation leads us to give the next definition:
\begin{defn}
Fix $K^p$, an $L$-valued $\iota$-twisted finite slope character distribution of level $K^p$ is a $\mathbb{Q}_p$-linear map $J': \mathcal{H}_p'(K^p)\rightarrow L$, such that for any $\iota$-invariant finite slope overconvergent representation $\sigma$ of $\mathcal{H}_p(K^p)$, there is a set of $l$ integers $\bar{m}_{J'}(\sigma):=\{m_{J',i}(\sigma)| i=1, \dots, l\}$, satisfying:
\begin{itemize}
 \item[$(1)$] for any $t\in T^{++}$ and $h\in\mathbb{Q}$, there are finitely many $\sigma$ of slope $\leqslant h$ and such that $\bar{m}_{J'}(\sigma)\neq 0$.
\item[$(2)$] for any $f\in\mathcal{H}_p'(K^p)$, $$J'(f)=\sum_{\sigma}\sum_{i}m_{J',i}(\sigma) J_{\tilde\sigma_i}(f)$$
\end{itemize}
\end{defn}

\begin{lem}
If $J$ is twisted finite slope character distribution, then $J'_{K^p}:=meas(K^p)^{-1}J$ is a twisted finite slope character distribution of level $K^p$. 
\end{lem}
\begin{proof}
Let $\pi$ be a $\iota$-invariant finite slope overconvergent representation of $\mathcal{H}_p$ and $\tilde{\pi}$ an extension of $\pi$ to $^\iota\mathcal{H}_p$. Since $K^p$ is $\iota$-invariant, for $f\in\mathcal{H}_p(K^p)$, 
\begin{equation}
meas(K^p)^{-1}J_{\tilde{\pi}}(f)=tr(\iota\times f\mid \tilde\pi^{K^p})
\end{equation}
(if $\tilde{\pi}^{K^p}=0$, both sides are $0$). Let $\tilde\sigma$ be an irreducible constitute of  $^\iota\mathcal{H}_p(K^p)$ acting on $\tilde\pi^{K^p}$. If the restriction of $\tilde\sigma$ on $\mathcal{H}_p(K^p)$ is reducible, by Lemma \ref{LemmaBLS}, $tr(\iota\times f\mid \tilde\sigma)=0$. So we have
\begin{equation}
\begin{aligned}
tr(\iota\times f\mid \tilde\pi^{K^p}) &=\sum_{\tilde\sigma}m(\tilde{\sigma},\tilde{\pi}^{K^p})tr(\iota\times f\mid\tilde{\sigma})\\
& =\sum_{\sigma}\sum_{j=1}^{l}m(\tilde{\sigma}_j,\tilde{\pi}^{K^p})J_{\tilde{\sigma}_j}(f).
\end{aligned}
\end{equation}
In the first equality, the sum of $\tilde{\sigma}$ is running over all irreducible constitute of $^\iota\mathcal{H}_p(K^p)$ in $\tilde\pi^{K^p}$ and $m(\tilde{\sigma},\tilde{\pi}^{K^p})$ is its multiplicity, which equals $1$ by Proposition \ref{Ktype}; in the second equality, the sum of ${\sigma}$ is running over all irreducible constitute of $\mathcal{H}_p(K^p)$ in $\pi^{K^p}$ such that $\sigma$ is $\iota$-invariant and $m(\tilde{\sigma}_j,\tilde{\pi}^{K^p})$ the multiplicities of $\tilde{\sigma}_j$ in $\tilde{\pi}^{K^p}$, which are all $0$ except for one $j$. Now 
\begin{equation}
\begin{aligned}
meas(K^p)^{-1}J(f)&=\sum_{\sigma^\iota\cong\sigma}\sum_{j}(\sum_im_{J,i}(\pi)m(\tilde{\sigma}_j,\tilde{\pi}_i^{K^p}))J_{\tilde{\sigma}_j}(f)\\
&=\sum_{\sigma^\iota\cong\sigma}\sum_{j}m_{J,j}(\pi)J_{\tilde{\sigma}_j}(f).
\end{aligned}
\end{equation}
This verifies the condition $(2)$ in the definition. Condition $(1)$ is a direct consequence of Definition 5.1.1 (1).
\end{proof}

\begin{cor}\label{cor1}
If $\sigma=\pi^{K^p}$ is a finite slope overconvergent representation of $\mathcal{H}_p(K^p)$, then 
\begin{equation}
m_{J'_{K^p},i}(\sigma)=m_{J,i}(\pi).
\end{equation}
\end{cor}

\vspace{1pc}
Similarly, we define the Fredholm determinant for any $f\in\mathcal{H}'_p$ associated to $J'$ by
\begin{equation}
P_{J'}(X, f):=\prod_\sigma\prod_{i}\det(1-X\tilde\sigma_i(\iota\times f))^{m_{J',i}(\sigma)}
 \end{equation}
We say that $J'$ is effective if it is non-trivial and  all its coefficients $m_{J',i}(\sigma)$ are non-negative. In this case, we define $V_{J'}$ as completion of 
\begin{equation}
\bigoplus_{\sigma}\bigoplus_i(V_{\tilde\sigma_i})^{m_{J',i}(\sigma)}
\end{equation}
under the super norm $\|\sum_iv_i\|$. Then
\begin{equation}
J'(f)=tr(\iota\times f\mid V_{J'}).
\end{equation}
If $J'=J'_{K^p}$ for some effective $J$, then it is obvious that $J'$ is effective, $V_{J'}= V_J(K^p)$ and for any $f\in\mathcal{H}_p'(K^p)$
\begin{equation}
P_{J_{K^p}'}(X, f)\mid P_{J}(X, f).
\end{equation}

\subsection{Some $\iota$-twisted distributions}
For $\lambda\in \mathfrak{X}^\iota(L)$ and $f\in \mathcal{H}_p$, define 
\begin{equation}I^\dagger_G(\iota\times f, \lambda): = tr(\iota \times f\mid H^*_{fs}(S_G, \mathcal{D}_\lambda(L))).\end{equation}
If $f\in\mathcal{H}_p(K^p)$, then
\begin{equation}
\begin{aligned}
I^\dagger_G(\iota\times f, \lambda)&= meas(K^p)\times tr(\iota \times f\mid H^*_{fs}(S_G(K^pI), \mathcal{D}_\lambda(L)))\\
&=meas(K^p)\times tr(\iota \times f\mid R\Gamma^*(S_G(K^pI), \mathcal{D}_\lambda(L))).
\end{aligned}
\end{equation}
We also write 
\begin{equation}
I^{'\dagger}_G(\iota\times f, \lambda, K^p): = tr(\iota \times f\mid H^*_{fs}(S_G(K^pI), \mathcal{D}_\lambda(L))).
\end{equation}
Let $\mathcal{P}_G^\iota$ (\emph{resp.} $\mathcal{L}_G^\iota$) be the set of standard parabolic (\emph{resp.} Levi) subgroups which are invariant under $\iota$. For $L\in\mathcal{L}_G^\iota$ and $w\in\mathcal{W}^{L,\iota}_{Eis}$, we define distributions $I^\dagger_{G,0}(\iota\times f, \lambda) $ and $I^\dagger_{G,L,w}(\iota\times f, \lambda)$ by induction on the unipotent rank of $G$:

\vspace{1pc}
If $rk(G)=0$, define 
\begin{equation}
I^\dagger_{G,0}(\iota\times f, \lambda)=I^\dagger_{G,G}(\iota\times f, \lambda): = I^\dagger_{G}(\iota\times f, \lambda);
\end{equation}
Given a positive integer $r$ and assume the distributions have been defined for cases that $rk(G)$ is less than $r$, then for proper $L\in\mathcal{L}^\iota_G$ and $f=f^p\otimes u_t$, define
\begin{equation}
I^{cl}_{G,L,w}(\iota\times f,\lambda): = I^{cl}_{L,0}(\iota\times f^{reg}_{L,w},w\cdot\lambda+2\rho_P),
\end{equation}
for regular dominant weight $\lambda\in\Upsilon$. For general $p$-adic weight, define
\begin{equation}
I^\dagger_{G,L,w}(\iota\times f,\lambda): = I^\dagger_{L,0}(\iota\times f^{reg}_{L,w},w\cdot\lambda+2\rho_P),
\end{equation}
\begin{equation}
I^\dagger_{G,L}(\iota\times f,\lambda): = \sum_{w\in\mathcal{W}^{L,\iota}_{Eis}}(-1)^{l(w)+dim\mathfrak{n}_L}I^\dagger_{G,L, w}(\iota\times f, \lambda),
\end{equation}
and then define:
$$I^\dagger_{G,0}(\iota\times f, \lambda): = I^\dagger_G(\iota\times f,\lambda)-\sum_{proper\ L\in\mathcal{L}_G^\iota}I^\dagger_{G,L}(\iota\times f, \lambda)$$

\begin{prop}\label{prop3}
For $?=L,\set{L,w}$ or $0$, $I^\dagger_{G,?}(\iota\times f, \lambda)$ is a $\iota$-twisted FSCD. 
\end{prop}
\begin{proof}
For $L\in\mathcal{L}_G^\iota$, let $\sigma^L$ be an irreducible finite slope representation of $\mathcal{H}_{p}^L$, define
\begin{equation}
I^G_{L,w}: = ind_{L(\mathbb{A}_f^p)}^{G(\mathbb{A}_f^p)}(\sigma_f^p)\otimes \theta_{\sigma, w},
\end{equation}
where $ind_{L(\mathbb{A}_f^p)}^{G(\mathbb{A}_f^p)}$ is the normalized induction by multipling the factor $e^{\langle\rho_P,H_P(g)\rangle}$, $\theta_{\sigma, w}$ is the character of $\mathcal{U}_p$ defined by 
\begin{equation}
u_t\mapsto |\rho_P(t)|_p\theta_\sigma(u_{wtw^{-1}}).
\end{equation}
If $\sigma^L$ is $\iota$-invariant, let $\tilde\sigma^L_{i}$ be one of its irreducible extensions to $^\iota\mathcal{H}_{p}^L$. It is easy to see
\begin{equation}\label{induce}
J_{\tilde\sigma^L_{i}}(f^{reg}_{L,w})=tr(\iota\times f | I^G_{L,w}).
\end{equation}
So by the induction process in the definition above, one only has to show the proposition for $I^\dagger_G(\iota\times f, \lambda)$. This follows from proposition \ref{induction1}.
\end{proof}

\subsubsection{Classical distributions}
The distributions $I^\dagger_?$ defined above will be $p$-adic interpolations of the classical distributions $I^{cl}_?$, which are defined from traces on classical cohomology groups. For an arithmetic regular dominant weight $\lambda\in\Upsilon$ and $f\in\mathcal{H}_p(K^p)$, we compute $meas(K^p)^{-1}I^{cl}_{G,0}(\iota\times f,\lambda)=$
\begin{equation}
\begin{aligned}
 &\ \ \ \ \lambda(\xi(t))tr^{st}(\iota\times f | H^*_{cusp}(S_G(K^pI), \mathbb{V}_\lambda^\vee(\mathbb{C})))\\
 &=\lambda(\xi(t))tr^{st}(\iota\times f | H^*(\mathfrak{m}_G,K_\infty; L^2_{cusp}(G(\mathbb{Q})A_G(\mathbb{R})^0\backslash G(\mathbb{A}))^K\otimes\mathbb{V}_\lambda^\vee(\mathbb{C}))(\xi_\lambda))
 \\ &= \sum_{\pi}\lambda(\xi(t))m_{cusp}(\pi)tr^{st}(\iota\times f|H^*(\mathfrak{m}_G,K_\infty; (\pi^{fin})^{K_f}\otimes \mathbb{V}_\lambda^\vee(\mathbb{C})))
 \\&=\sum_{\pi}m_{cusp}(\pi)tr^{st}(\iota|H^*(\mathfrak{m}_G,K_\infty; \pi_\infty^{fin}\otimes \mathbb{V}_\lambda^\vee(\mathbb{C})))tr(\iota\times f|\pi_{f}^{K_f})\\
& =meas(K^p)^{-1}\sum_{\pi}m_{cusp}(\pi)tr^{st}(\iota|H^*(\mathfrak{m}_G,K_\infty; \pi_\infty^{fin}\otimes \mathbb{V}_\lambda^\vee(\mathbb{C})))tr(\iota\times f|\pi_{f}^{I})
 \end{aligned}
\end{equation}
where, like Proposition \ref{induction1}, the summation is running over all cuspidal representations $\pi\subset L^2_{cusp}(G(\mathbb{Q})\backslash G(\mathbb{A}), \xi_\lambda)$ such that $\pi^\iota=\pi$. $\pi^{fin}$ is the Harish-Chandra module of $\pi$, i.e., the subspace of $\pi$ consisting of smooth vectors that generate a finite dimensional vector space under $K_\infty$. $m_{cusp}(\pi)$ is the multiplicity of $\pi$ in $L^2_{cusp}$, that is,
\begin{equation}
L^2_{cusp}(G(\mathbb{Q})\backslash G(\mathbb{A}), \xi_\lambda)=\bigoplus_\pi\pi^{m_{cusp}(\pi)}.
\end{equation} 

\vspace{1pc}
By Proposition \ref{Ktype}, if $\pi_f^I\neq0$ then it is $\iota$-invariant and irreducible as a $C_c^{\infty}(I\backslash G(\mathbb{A}_f)/I, \overline{\mathbb{Q}}_p)$-module. A constitute of $\pi_f^I$ restricting on $\mathcal{H}_p$ gives a $p$-stabilization of $\pi_f$, and there are only finitely many such $p$-stabilizations (see \cite[\S4.1.9]{Urban}). So 
\begin{equation}\label{formula6}
tr(\iota\times f|\pi_{f}^{I})=\sum_{\rho}m(\rho,\pi_f^I)tr(\iota\times f\mid \rho),
\end{equation}
where $\rho$ is running over all irreducible $\mathcal{H}_p$-submodules of $\pi_f^I$ such that $\rho^\iota=\rho$, and  $m(\rho,\pi_f^I)$ is the multiplicity of $\rho$ in $\pi_f^I$. 

\vspace{1pc}
Now we compute the Lefschetz number
\begin{equation}
\begin{aligned}
L(\iota,\pi,\lambda):&=tr^{st}(\iota|H^*(\mathfrak{m}_G,K_\infty; \pi_\infty^{fin}\otimes \mathbb{V}_\lambda^\vee(\mathbb{C})))\\
&=\sum_{i}(-1)^itr^{st}(\iota|H^i(\mathfrak{m}_G,K_\infty; \pi_\infty^{fin}\otimes \mathbb{V}_\lambda^\vee(\mathbb{C}))).
\end{aligned}
\end{equation}

The next theorem is easily deduced from Proposition 5.4, 5.5 and Theorem 5.6 in the original paper of Vogan and Zuckerman \cite[\S5]{VZ}:

\begin{thm}
Let $\pi$ be a cuspidal representation (so it is essentially unitary) such that $H^*(\mathfrak{m}_G,K_\infty; \pi_\infty^{fin}\otimes\mathbb{V}_\lambda^\vee(\mathbb{C}))\neq 0$, then there exists a $\theta$-stable parabolic subalgebra $\mathfrak{q}=\mathfrak{l}+\mathfrak{u}$ such that
\begin{equation}
dim\mathbb{V}_\lambda^\vee(\mathbb{C})/\mathfrak{u}\mathbb{V}_\lambda^\vee(\mathbb{C}) = 1
\end{equation}
and $\pi_\infty$ is of the form $A_\mathfrak{q}(w_0\lambda)$.
Moreover, one has
\begin{equation}
H^i(\mathfrak{m}_G,K_\infty; \pi_\infty^{fin}\otimes\mathbb{V}_\lambda^\vee(\mathbb{C}))\cong Hom_{\mathfrak{l\cap k}}(\wedge^{i-R_\mathfrak{q}}(\mathfrak{l\cap k}),\mathbb{C})
\end{equation}
\end{thm}
To understand the theorem here we have to recall some notation from \cite{VZ} as well. $\theta$ is a usual Cartan involution of $G$, which gives a Cartan decomposition 
\begin{equation}
\mathfrak{g=p+k}.
\end{equation} 
For a $\theta$-stable parabolic algebra $\mathfrak{q}$ defined as in \cite[\S2]{VZ} and an admissible weight $\lambda$ defined as in \cite[(5.1)]{VZ}, $A_\mathfrak{q}(\lambda)$ is an irreducible $\mathfrak{g}$-module whose restriction on $\mathfrak{k}$ contains a represntation $\mu(\mathfrak{q},\lambda)$, which is the representation of  $K_\infty$ of highest weight $\lambda+2\rho(\mathfrak{p\cap u})$, as in \cite[Theorem 5.3]{VZ}. Here $R_\mathfrak{q}=dim (\mathfrak{p\cap u})$. 

\begin{cor}
Assumptions as last theorem, if $\lambda$ is regular, then $\mathfrak{u}$ is maximal unipotent and $\mathfrak{q}$ is Borel.
\end{cor}
\begin{proof}
It's a simple observation from the previous theorem. The fact that $\mathbb{V}_\lambda^\vee(\mathbb{C})/\mathfrak{u}\mathbb{V}_\lambda^\vee(\mathbb{C})$ is one dimensional implies that $\mathbb{V}_\lambda^\vee(\mathbb{C})$ can be realized in $Ind_{\mathfrak{q}}^{\mathfrak{g}}(\chi)$ with some character $\chi$ of $\mathfrak{l}$ whose restriction to $\mathfrak{t}$ is $-w_0\lambda$. However, since $\lambda$ is regular, $-w_0\lambda$ is regular too. This means that $-w_0\lambda$ cannot be extended to a character of any Levi subgroup which contains $T$ proporly (see e.g. \cite[\S II.1.18, II.2.4]{Jantzen}). So $\mathfrak{u}$ is maximal unipotent and $\mathfrak{q}$ is Borel.
\end{proof}

The corollary implies that there are only finitely many cuspidal $\pi$ such that $H^*(\mathfrak{m}_G,K_\infty; \pi_\infty^{fin}\otimes\mathbb{V}_\lambda^\vee(\mathbb{C}))\neq 0$, and their infinity parts $\pi_\infty$ are $A_\mathfrak{b}(w_0\lambda)$. Write $R=R_\mathfrak{b}$, then for each such $\pi$
\begin{equation}
\begin{aligned}
L(\iota,\pi,\lambda)
&=\sum_{i}(-1)^itr(\iota|Hom_{\mathfrak{l\cap k}}(\wedge^{i-R_\mathfrak{q}}(\mathfrak{t\cap p}),\mathbb{C}))
\\&=(-1)^R\sum_{i}(-1)^itr(\iota|Hom_{\mathfrak{l\cap k}}(\wedge^{i}(\mathfrak{t\cap p}),\mathbb{C}))
\end{aligned}
\end{equation}
So we define $q_{G,\iota}:=L(\iota,\pi,\lambda)$ for any $\lambda$ arithmetic regular dominant in $\Upsilon$, it is an integer depends on $G$ and $\iota$ only. In particular
\begin{equation}
\begin{aligned}
I^{cl}_{G,0}(\iota\times f,\lambda)&=\sum_{\pi=\pi^\iota}q_{G,\iota}m_{cusp}(\pi)tr(\iota\times f|\pi_{f}^{I})\\
&=\sum_{\rho=\rho^\iota}q_{G,\iota}\sum_{\rho\subset\pi^I_f}m_{cusp}(\pi)m(\rho,\pi_f^I)tr(\iota\times f\mid \rho)\\
&=\sum_{\rho\cong\rho^\iota}q_{G,\iota}\sum_{i=1}^l(\sum_{\rho\subset\pi^I_f}m_{cusp}(\pi)m(\tilde{\rho}_i, \pi_f^{I}))J_{\tilde{\rho}_i}(f).
\end{aligned}
\end{equation}

\vspace{1pc}
The discussion above together with (\ref{induce}) implies:
\begin{prop}
Let $\lambda$ be an arithmetic regular dominant weight in $\Upsilon$. For $?=\emptyset,\set{L,w}$ or $0$, $I^{cl}_{G,?}(\iota\times f, \lambda)$ is a $\iota$-twisted finite slope character distributions. 
\end{prop}
\begin{proof}
We only have to show the representations appearing in the twisted traces are of finite slope. Let $\sigma$ be an classical cuspidal representation of $\mathcal{H}_p(K^p)$, its $p$-component $\sigma_p$ must be of dimension one. So for any $f=f^p\otimes u_t\in\mathcal{H}_p(K^p)$
\begin{equation}
J_{\tilde{\sigma}}(f)=J_{\tilde{\sigma}^p}(f^p)\xi\theta_\sigma(u_t)
\end{equation}
for some $l$-th root of unit $\xi$. Now as observed in \S3.3.2, if $\sigma$ is of infinite slope, $\theta_\sigma(u_t)=0$.
\end{proof}

\begin{prop}\label{prop2}
The character distribution $I^{cl}_{G,0}(\iota\times f,\lambda)\neq0$ if and only if $q_{G,\iota}\neq0$. In this case, define $e_{G,\iota}:=q_{G,\iota}^{-1}$, then $e_{G,\iota}I^{cl}_{G,0}(\iota\times f,\lambda)$ is effective. 
\end{prop}
\begin{proof}
We only have to prove the first statement. The ``only if" part is obvious. Noting that $\mathcal{U}_p$ is commutative, $\rho_p$ as in (\ref{formula6}) must be a character $\theta_\rho$. So we only have to show that there exits a one  dimensional subspace of $\pi^I_p$ which is stable under $\iota$. This is true, since $\pi^I_p$ is finite dimensional and it is diagonalizable under $\iota$. 
\end{proof}

\begin{remark}\label{ob3}
In last proposition, we have to divide $q_{G,\iota}$ to make sure that the distribution has integral coefficients. If $\iota$ is of order $2$, the Lefschetz numbers are necessarily integral. In this case, we define $e_{G,\iota}=sign(q_{G,\iota})$. All the results and discussion will be same but the distribution carries more information.
\end{remark}

\begin{remark}\label{lateadd}
Throughout this paper we assume that $\iota$ is of Cartan-type to make sure that the cuspidal cohomology is not always trivial (\cite[Theorem 10.6]{BLS}). However, our results hold for any $\mathbb{Q}$-rational, finite order automorphism such that the Lefschetz number is non-trivial.
\end{remark}

Now for any $K^p$, write
\begin{equation}
V^{cl,\lambda}_{G,0}(K^p):=V_{e_{G,\iota}I^{cl}_{G,0}(\lambda)}(K^p) \text{ and }  V^{cl,\lambda,'}_{G,0}(K^p):=V_{(e_{G,\iota}I^{cl,'}_{G,0}(\lambda))_{K^p}}
\end{equation}
\begin{cor}
\begin{equation}
V^{cl,\lambda}_{G,0}(K^p)=V^{cl,\lambda,'}_{G,0}(K^p)
\end{equation}
\end{cor}

\vspace{1pc}
\subsubsection{Congruence between overconvergent and classical distributions}
\begin{lem}\label{lemma1}
Let $\lambda$ be a regular arithmetic weight and $f=f^p\otimes u_t \in \mathcal{H}_p'(K^p)$ $\mathbb{Z}_p$-valued, then
\begin{equation}
I_G^\dagger(\iota\times f, \lambda)\equiv I_G^{cl}(\iota\times f, \lambda) \hspace{1pc}mod \hspace{1pc} N^\iota(\lambda, t)Meas(K^p),
\end{equation}
where $N^\iota(\lambda, t)$ is defined in $(\ref{Nlambda})$.
\end{lem}
\begin{proof}
Both sides of the congruence are $mear(K^p)\times \mathbb{Z}_p$-valued by definition. Let $h=v_p(N^\iota(\lambda, t))$. Noting that for any $a\in\mathbb{Z}$ and $t\in T^+$,
\begin{equation}
|a\mu_\theta(t)|_p=|\theta(u_{t^a})|_p,
\end{equation}
$(\ref{ob1})$ and the observation $(\ref{ob2})$ imply that
$$I_G^\dagger(\iota\times f, \lambda)\equiv tr(\iota\times f| H^*_{fs}(S_G(K^pI),\mathcal{D}_\lambda(L)))^{\leqslant h}_\iota\hspace{1pc}mod \hspace{1pc} N^\iota(\lambda, t)Meas(K^p)$$
$$I_G^{cl}(\iota\times f, \lambda)\equiv tr(\iota\times f| H^*(S_G(K^pI_m),\mathbb{V}^\vee_\lambda(L)))^{\leqslant h} _\iota\hspace{1pc}mod \hspace{1pc} N^\iota(\lambda, t)Meas(K^p).$$
Then the lemma is obtained from Proposition \ref{classicity}.
\end{proof}

\begin{prop}\label{prop3}
Let $f=f^p\otimes u_t \in \mathcal{H}_p(K^p)$ be $\mathbb{Z}_p$-valued and $\lambda$ regular arithmetic. Then
\begin{equation}
I_{G,0}^\dagger(\iota\times f, \lambda)\equiv I_{G,0}^{cl}(\iota\times f, \lambda) \hspace{1pc}mod \hspace{1pc} N^\iota(\lambda, t)Meas(K^p).
\end{equation}
\end{prop}
\begin{proof}
We prove the proposition by induction on $rk(G)$. The case $rk(G)=0$ is just the lemma \ref{lemma1} above. Now assume the proposition is proved for any proper Levi subgroup $L\in\mathcal{L}_G^\iota$. Noting that $N^\iota(\lambda, t)$ divides $\xi(t)^{\lambda-w^{-1}v*\lambda}$ if $v\neq w$, the cuspidal decomposition formula, Theorem \ref{cuspidalformula} implies that, modulo $N^\iota(\lambda, t)Meas(K^p)$, $I^{cl}_{G,0}(\iota\times f,\lambda)$ is congruent to
\begin{equation}
I^{cl}_G(\iota\times f,\lambda)-\sum_{P\in\mathcal{P}^\iota, P\neq G}\sum_{w\in\mathcal{W}^{L,\iota}_{Eis}}(-1)^{l(w)+n_P}I^{cl}_{L,0}(\iota\times f^{reg}_{L,w},w\cdot\lambda+2\rho_P).
\end{equation}
Now using the induction hypotheses and Lemma \ref{lemma1} again, it is congruent to
\begin{equation}
I^{\dagger}_G(\iota\times f,\lambda)-\sum_{P\in\mathcal{P}^\iota, P\neq G}I^{\dagger}_{G,L}(\iota\times f,w\cdot\lambda+2\rho_P),
\end{equation}
which is, by definition, $I_{G,0}^\dagger(\iota\times f, \lambda)$.
\end{proof}

\begin{cor}\label{limit}
let $\set{\lambda_n}$ be a highly regular sequence of $\iota$-invariant dominant weights which converges $p$-adically to a weight $\lambda\in\mathfrak{X}^\iota(L)$. Then for any $f=f^p\otimes u_t\in \mathcal{H}_p'$, there is 
\begin{equation}
\lim_{n\rightarrow\infty}I^{cl}_{G,?}(\iota\times f, \lambda_n)=I^\dagger_{G,?}(\iota\times f, \lambda)
\end{equation}
for $?=\emptyset,0$.
\end{cor}

\subsection{Analyticity with respect to weight}
Now we study $I^\dagger_{G,?}(\iota\times f, \lambda)$ as the weight $\lambda$ varying over the weight space $\mathfrak{X}^\iota$. Let $\mathfrak{U}$ be an open affinoid of $\mathfrak{X}^\iota$ and $\mathcal{O}^0(\mathfrak{U})$ the ring of analytic functions on $\mathfrak{U}$ bounded by $1$. For any finite extension $L$ of $\mathbb{Q}_p$ in $\overline{\mathbb{Q}}_p$, define 
\begin{equation}
\Lambda_{\mathfrak{X}^\iota}:=\varprojlim_{\mathfrak{U}\subset \mathfrak{X}^\iota}\mathcal{O}^0(\mathfrak{U})\subset\mathcal{O}(\mathfrak{X}^\iota).
\end{equation} 
\begin{equation}
\Lambda_{\mathfrak{X}^\iota,L}:=\Lambda_{\mathfrak{X}^\iota}\otimes L
\end{equation}

\begin{prop}\label{analytic}
Fix $f\in\mathcal{H}_p'(K^p)$, then as functions of  $\lambda\in\mathfrak{X}^\iota$, $I^\dagger_G(\iota\times f, \lambda)$, $I^\dagger_{G,M,w}(\iota\times f, \lambda)$ and $I^\dagger_{G,0}(\iota\times f, \lambda)$ are all in $\Lambda_{\mathfrak{X}^\iota,\mathbb{Q}_p}$. In particular, they are analytic over $\mathfrak{X}^\iota$.
\end{prop}
\begin{proof}
The proof is same to \cite[Theorem 4.7.3]{Urban}, so we only sketch it here. By the induction process in defining $I^\dagger_{G,?}(\iota\times f, \lambda)$, it suffices to prove the proposition for $I^\dagger_G(\iota\times f, \lambda)$. Locally over an open affinoid $\mathfrak{U}\subset\mathfrak{X}^\iota$, for $n\geq n(\mathfrak{U})$, Lemma \ref{lemma2} and Proposition \ref{prop}(a) imply that 
\begin{equation}
R\Gamma^*(K^pI,\mathcal{D}_{\mathfrak{U},n})\otimes_\lambda L\cong R\Gamma^*(K^pI,\mathcal{D}_{\lambda,n}(L)).
\end{equation} 
Therefore $F_{\mathfrak{U}}:=meas(K^p)tr(\iota\times f|R\Gamma^*(K^pI,\mathcal{D}_{\mathfrak{U},n(\mathfrak{U})}))$ (viewed as a function on $\mathfrak{U}$ via specialization) is in $\mathcal{O}(\mathfrak{U})$, such that, for any $\lambda\in\mathfrak{U}$, $F_{\mathfrak{U}}(\lambda)=I^\dagger_G(\iota\times f, \lambda)$. Moreover,  since $\iota$ and $f\in\mathcal{H}_p'$ preserve the $\mathcal{O}^0(\mathfrak{U})$-lattice $R\Gamma^*(K^pI,\mathcal{D}^0_{\mathfrak{U},n})$, that $F_{\mathfrak{U}}$ is in $\mathcal{O}^0(\mathfrak{U})$, where $\mathcal{D}^0_{\mathfrak{U},n}$ is the $\mathcal{O}^0(\mathfrak{U})$ dual of $\mathcal{A}_{\mathfrak{U}, n}^0$. So we have 
\begin{equation}
I^\dagger_{G}(\iota\times f,-)=\varprojlim_{\mathfrak{U}\subset \mathfrak{X}^\iota}F_{\mathfrak{U}}\in\varprojlim_{\mathfrak{U}\subset \mathfrak{X}^\iota}\mathcal{O}^0(\mathfrak{U})=\Lambda_{\mathfrak{X}^\iota}.
\end{equation}
\end{proof}

\subsection{Effectivity}
\begin{prop}\label{effective}
If $q_{G,\iota}\neq0$, then $e_{G,\iota}I^\dagger_{G,0}(\iota\times f,\lambda)$ is an effective $\iota$-twisted finite slope character distribution.
\end{prop}
\begin{proof}
Since algebraic regular dominant weights are dense in the weight space, by Proposition \ref{analytic}, it suffices to prove the proposition for algebraic regular dominant weights $\lambda$ in $\Upsilon$. Since $q_{G,\iota}\neq0$, 
by Proposition \ref{prop2}, $e_{G,\iota}I^{cl}_{G,0}(\iota\times f,\lambda)$ is effective. Let $P^{cl}_{G.0}(\iota\times f,\lambda, X)$ and $P^\dagger_{G.0}(\iota\times f,\lambda, X)$ be the Fredholm determinants associated to $e_{G,\iota}I^{cl}_{G,0}(\iota\times f,\lambda)$ and $e_{G,\iota}I^\dagger_{G,0}(\iota\times f,\lambda)$ respectively, define
\begin{equation}
P^{\dagger-cl}_{G,0}(\iota\times f, \lambda, X)=\frac{P^\dagger_{G.0}(\iota\times f,\lambda, X)}{P^{cl}_{G.0}(\iota\times f,\lambda, X)}.
\end{equation}

%% by \cite[Theorem+ coro 3.2.5]{RS} sign((-1)^R\sum_{i}(-1)^itr(\iota|Hom_{\mathfrak{l\cap k}}(\wedge^{i}(\mathfrak{l\cap p}),\mathbb{C}))) independent of $\lambda$. since our $\iota$ is of Cartan-type
Now we need the lemma below, which is a direct consequence of Proposition \ref{prop3}.
\begin{lem}\label{lemma3}
If $\lambda$ is regular, then $P^{\dagger-cl}_{G,0}(\iota\times f, \lambda, X)$ is a meromorphic function on $\mathbb{A}^1_{rig}(\mathbb{C}_p)$, its zeros and poles are all lying in $$\{x\in\mathbb{C}_p | |x|_p\geq N^\iota(\lambda, t) \}$$
\end{lem}
\begin{proof}
\emph{of the lemma: }If $J_1$ and $J_2$ are two twisted finite slope character distributions, then so is $J_1-J_2$, and $P_{J_1-J_2}(X,f)=P_{J_1}(X,f)/P_{J_2}(X,f)$. Write $I^{\dagger-cl}_{G,0}(\iota\times f, \lambda):=I^{\dagger}_{G,0}(\iota\times f, \lambda)-I^{cl}_{G,0}(\iota\times f, \lambda)$. Proposition \ref{prop3} implies that, for $\mathbb{Z}_p$-valued $f=f^p\otimes u_t\in\mathcal{H}_p(K^p)$, $I^{\dagger-cl}_{G,0}(\iota\times f, \lambda)'_{K^p}\equiv 0$ mod $N^{\iota}(\lambda,t)$. So $P^{\dagger-cl}_{G,0}(\iota\times f, \lambda, X)\equiv1$ mod $N^{\iota}(\lambda,t)$. This proves the lemma.
\end{proof}

Now we can run the same argument as in the proof of \cite[Theorem 4.7.3]{Urban} to show our proposition. Choose a closed affinoid subdomain $\mathfrak{U}\subset \mathfrak{X}^\iota$ which contains one hence dense algebraic weights in $\Upsilon$.  Shrink $\mathfrak{U}$ if necessary so that we can write $P^\dagger_{G.0}(\iota\times f,\lambda, X)$ as a quotient of relatively prime Fredholm series over $\mathfrak{U}$, that is,
\begin{equation}
P^\dagger_{G.0}(\iota\times f,\lambda, X)=\frac{T(\iota\times f,\lambda, X)}{B(\iota\times f , \lambda, X)},
\end{equation}
with both $T(\iota\times f , \lambda, X)$ and $B(\iota\times f , \lambda, X)$ are in $\mathcal{O}(\mathfrak{U}\times \mathbb{A}^1_{rig})$. Assume $B(\iota\times f , \lambda, X)\neq 1$, let $\mathfrak{W}$ be the Fredholm subvariety of $\mathfrak{U}\times\mathbb{A}^1_{rig}$ cut out by $T(\iota\times f , \lambda, X)$ and $B(\iota\times f , \lambda, X)$, that is, $\mathfrak{W}=Z(B)-Z(T)$. Since the projection $pr: Z(B)\rightarrow\mathfrak{U}$ is flat, that its image $pr(\mathfrak{W})$ also contains dense algebraic weights. Now let $w=(\lambda, x)\in \mathfrak{W}(\overline{\mathbb{Q}}_p)$ with $\lambda$ algebraic, we can choose $w'=(\lambda', x')$ $p$-adically close to $w$ such that $\lambda'$ is algebraic regular dominant in $\Upsilon$ and $|x'|_p < N^\iota(\lambda', t)$. So by Lemma \ref{lemma3}, $x'$ must be a pole of  $P^{cl}_{G.0}(\iota\times f,\lambda', X)$. However, since $e_{G,\iota}I^{cl}_{G,0}(\iota\times f, \lambda')$ is effective, $P^{cl}_{G.0}(\iota\times f,\lambda', X)$ is entire. So our assumption leads to a contradiction. This implies that $P^\dagger_{G.0}(\iota\times f,\lambda, X)$ is entire, therefore, $e_{G,\iota}I^\dagger_{G,0}(\iota\times f,\lambda)$ is also effective.
\end{proof}

\begin{cor}
For any $\iota$-invariant standard Levi subgroup of $G$, there exists a number $e_{L,\iota}$ such that, for any $w\in\mathcal{W}_{Eis}^{L,\iota}$, $e_{L,\iota}I^{cl}_{G,L,w}(\iota\times f,\lambda)$ and $e_{L,\iota}I^\dagger_{G,L,w}(\iota\times f, \lambda)$ are effective, unless $I^{cl}_{G,L,w}(\iota\times f,\lambda)$ is trivial for some algrbraic regular weight $\lambda\in\Upsilon$
\end{cor}
\begin{proof}
It follows from the definition of $I^\dagger_{G,L,w}(\iota\times f, \lambda)$ and an exactly same argument for $L$ as above.
\end{proof}

\subsection{Multiplicities}
For $?={cl}$ or $\dagger$, write the Fredholm determinants associated to $e_{G,\iota}I^{?}_{G,0}(\iota\times f, \lambda)$ by $P^?_{G,0}(\iota\times f, \lambda, X)$, to $e_{L,\iota}I^?_{G,L,w}(\iota\times f, \lambda)$ by $P^?_{G,L,w}(\iota\times f, \lambda, X)$. Let $\pi$ be a finite slope overconvergent representation, write the multiplicities:
\begin{equation}
\bar{m}^{\iota,?}_{G,0}(\pi, \lambda):=\bar{m}_{e_{G,\iota}I^{?}_{G,0}(\lambda)}(\pi)\text{ , } {m}^{\iota,?}_{G,0}(\pi, \lambda):={m}_{e_{G,\iota}I^?_{G,0}(\lambda)}(\pi),
\end{equation}
\begin{equation}
\bar{m}^{\iota,?}_{G,L,w}(\pi, \lambda):=\bar{m}_{e_{L,\iota}I^?_{G,L,w}(\lambda)}(\pi)
\text{ , }
 {m}^{\iota,?}_{G,L,w}(\pi, \lambda):={m}_{e_{L,\iota}I^?_{G,L,w}(\lambda)}(\pi),
\end{equation}
For given $K^p$, if $\theta$ is an overconvergent Hecke eigensystem of level $K^p$, write its multiplicity in $V^{?, \lambda}_{G,0}(K^p)$  by $m_{G,0}^{?,\iota}(\theta, \lambda)$ . By Proposition \ref{prop3}, if $\theta$ or $\pi$ is not $\iota$-invariant, its multiplicities are $0$.

\vspace{1pc}
\begin{lem}\label{multiplicity}
Let $\lambda$ be an arithmetic regular weight and $\pi$ an $\iota$-invariant finite slope overconvergent representation, which is non-critical with respect to $\lambda^{alg}$, then
\begin{equation}\label{sum}
m_{G,0}^{\iota, cl}(\pi, \lambda)=m_{G,0}^{\iota, \dagger}(\pi, \lambda)\end{equation}
\end{lem}
\begin{proof}
Assume $\pi$ is of level $K^p$. Since $\pi$ is non-critical with respect to $\lambda$, there is $t\in T^{++}$ such that $v_p(N^{\iota}(\lambda,t))>v_p(\theta_\pi(u_t))$. if necessary, we can replace $t$ by $t^N$ for some positive integer $N$ to make $v_p(N^{\iota}(\lambda,t))-v_p(\theta_\pi(u_t))$ arbitrarily large. Now consider the finite set of $\iota$-invariant finite slope overconvergent representations:
\begin{equation*}
\Sigma^{t}_{K^p}=\set{\rho \mid \rho^{K^p}\neq0, v_p(\rho(u_t))\leq v_p(N^\iota(\lambda,t)),  \bar{m}^{\iota,\dagger}_{G,0}(\rho, \lambda)\text{ or } \bar{m}^{\iota,cl}_{G,0}(\rho, \lambda)\neq0}
\end{equation*}
Since $\Sigma^{t}_{K^p}$ is finite, by Jacobson's Lemma, there is $f\in\mathcal{H}_p(K^p)$ such that $\pi(f)=id_{\pi^{K^p}}$ and $\rho(f)=0$ for any $\rho\in\Sigma^{t}_{K^p}$ that is not isomorphic to $\pi$. 

Consider $f_0=(1_{K^p}\otimes u_t)f$, we have for $?=\dagger$ or $cl$ 
\begin{equation}
P^?_{G,0}(\iota\times f_0, \lambda, X)=\prod_{i=1}^l\det{(1-\tilde\pi_i(\iota\times1_{K^p}\otimes u_t)X)}^{{m}^{\iota,?}_{G,0,i}(\pi, \lambda)}S^?_{G,0}(X),
\end{equation}
where $S^?_{G,0}(X)$ is the product of the determinants of all representations whose slopes are strictly greater than $v_p(N^\iota(\lambda,t))$. Noticing that $\Sigma^{t^M}_{K^p}=\Sigma^{t}_{K^p}$ for any positive integer $M$, we can therefore choose $f$ independently for any $M$. So if necessary, we can replace $t$ by $t^M$ such that $f_0=(1_{K^p}\otimes u_t)f$ is $\mathbb{Z}_p$-valued. Then by Proposition \ref{prop3} and Lemma \ref{lemma3}, we have 
\begin{equation}
\begin{aligned}
&\prod_{i=1}^l\det{(1-\tilde\pi^{K^p}_i(\iota\times1_{K^p}\otimes u_t)X)}^{{m}^{\iota,\dagger}_{G,0,i}(\pi, \lambda)}\\
\equiv&\prod_{i=1}^l\det{(1-\tilde\pi^{K^p}_i(\iota\times1_{K^p}\otimes u_t)X)}^{{m}^{\iota,cl}_{G,0,i}(\pi, \lambda)} \ \ \ mod\ N^\iota(\lambda,t)
\end{aligned}
\end{equation}
and they share the same zeros (order counted) mod $N^\iota(\lambda,t)$. If necessary, replace $t$ by $t^N$ as we observed at the beginning of the proof, we can assume that $v_p(N^\iota(\lambda,t))$ is strictly greater than the $p$-adic valuation of all the coefficients of $\prod_{i=0}^l\det{(1-\tilde\pi^{K^p}_i(\iota\times1_{K^p}\otimes u_t)X)}^{{m}^{\iota,?}_{G,0,i}(\pi, \lambda)}$, so we have 
\begin{equation}\label{formula6}
\begin{aligned}
&\prod_{i=1}^l\det{(1-\tilde\pi^{K^p}_i(\iota\times1_{K^p}\otimes u_t)X)}^{{m}^{\iota,\dagger}_{G,0,i}(\pi, \lambda)}\\
=&\prod_{i=1}^l\det{(1-\tilde\pi^{K^p}_i(\iota\times1_{K^p}\otimes u_t)X)}^{{m}^{\iota,cl}_{G,0,i}(\pi, \lambda)}.
\end{aligned}
\end{equation}
In particular, they have the same degree, which are $\dim(\pi)m_{G,0}^{\iota, \dagger}(\pi, \lambda)$ and $\dim(\pi)m_{G,0}^{\iota, cl}(\pi, \lambda)$ respectively. So $m_{G,0}^{\iota, cl}(\pi, \lambda)=m_{G,0}^{\iota, \dagger}(\pi, \lambda)$.

%% Old proof: Noting that $\tilde\pi_i(\iota\times1_{K^p}\otimes u_t)=\tilde\pi_i(\iota)\theta_{\pi}(u_t)$, the leading coefficient of $\prod_{i=0}^l\det{(1-\tilde\pi_i(\iota\times1_{K^p}\otimes u_t)X)}^{{m}^{\iota,?}_{G,0,i}(\pi, \lambda)}$ is a product of $\theta_{\pi}(u_t)$ with some $l$-th roots of unity. So the leading coefficients are not $0$ mod $N^\iota(\lambda, t)$. This implies that the both sides mod $N^\iota(\lambda,\lambda)$ are of the same degree, which are $\dim(\pi)m_{G,0}^{\iota, \dagger}(\pi, \lambda)$ and $\dim(\pi)m_{G,0}^{\iota, cl}(\pi, \lambda)$ respectively. In particular $m_{G,0}^{\iota, cl}(\pi, \lambda)=m_{G,0}^{\iota, \dagger}(\pi, \lambda)$.
\end{proof}

Noting that $\tilde\pi^{K^p}_i(\iota\times1_{K^p}\otimes u_t)=\tilde\pi^{K^p}_i(\iota)\theta_{\pi}(u_t)$ and $\pi^{K^p}$ is finite dimensional, we can compute (\ref{formula6}) more explicitly. Fix $\xi$ a primitive $l$-th roots of unity, let $k_j^{(i)}$ be the multiplicity of $\xi^j$ as an eigenvalue of $\iota$ in $\tilde\pi_i^{K^p}$. Then for any $i=1,\cdots, l$,
\begin{equation}
\sum_{j=i}^lk_j^{(i)}=\dim\pi^{K^p}
\end{equation}
and 
\begin{equation}
\begin{aligned}
&\prod_{i=1}^l\det{(1-\tilde\pi^{K^p}_i(\iota\times1_{K^p}\otimes u_t)X)}^{{m}^{\iota,?}_{G,0,i}(\pi, \lambda)}\\
=&\prod_{i=1}^l(\prod_{j=1}^l(1-\xi^j\theta_\pi(u_t)X)^{k_j^{(i)}})^{{m}^{\iota,?}_{G,0,i}(\pi, \lambda)}\\
=&\prod_{j=1}^l(1-\xi^j\theta_\pi(u_t)X)^{\sum_i{k_j^{(i)}}{{m}^{\iota,?}_{G,0,i}(\pi, \lambda)}}.
\end{aligned}
\end{equation}
This gives us varies identities between the multiplicities. In particular, compare the coefficients of degree $1$, we have
\begin{equation}\label{difference}
\sum_{i=1}^l\xi^i{m}^{\iota,\dagger}_{G,0,i}(\pi, \lambda)=\sum_{i=1}^l\xi^i{m}^{\iota,cl}_{G,0,i}(\pi, \lambda).
\end{equation}
If $l=2$, $(\ref{sum})$ and $(\ref{difference})$ imply that 
\begin{equation}
\bar{m}_{G,0}^{\iota, cl}(\pi, \lambda)=\bar{m}_{G,0}^{\iota, \dagger}(\pi, \lambda).
\end{equation}

\vspace{1pc}
Now assume that $\iota$ is of order $2$, everything is defined as in Remark \ref{ob3}. Let $\lambda=\lambda^{alg}\epsilon$ be an arithmetic weight. If $\pi$ is a $\iota$-invariant finite slope cuspidal representation, its $\iota$-twisted Euler-Poincare characteristic $m_{EP}^{\iota}(\pi,\lambda)$ is defined by:
\begin{gather}\label{epdef}
\begin{aligned}
\sum_q(-1)^q tr(\iota\mid Hom_{\mathcal{H}_p}(\pi, H^q(\tilde S_G,\mathbb{V}_{\lambda^{alg}}^\vee(\mathbb{C},\epsilon)))).
\end{aligned}
\end{gather}
If $\pi$ is of level $K^p$, then $m_{EP}^\iota(\pi,\lambda)$ equal:
\begin{gather}
\begin{aligned}
 \sum_q(-1)^q tr(\iota\mid Hom_{\mathcal{H}_p(K^p)}(\pi^{K^p}, H^q(S_G(K^pI),\mathbb{V}_{\lambda^{alg}}^\vee(\mathbb{C},\epsilon)))).
\end{aligned}
\end{gather}
Then a computation as in \S5.2 shows that
\begin{equation}\label{ep}
m_{EP}^\iota(\pi,\lambda)=m^{\iota,cl}_{G,0}(\pi,\lambda).
\end{equation}

\vspace{1pc}
We close this section by considering the multiplicities of Hecke eigensystems:
\begin{cor}\label{arithmetic}
Let $\lambda$ be an arithmetic regular weight and $\theta$ a finite slope $\iota$-invariant overconvergent Hecke eigensysem, which is non-critical with respect to $\lambda^{alg}$, then
\begin{equation}
m_{G,0}^{\iota, cl}(\theta, \lambda)=m_{G,0}^{\iota, \dagger}(\theta, \lambda)\end{equation}
\end{cor}
\begin{proof}
This is a direct consequence of Lemma \ref{multiplicity} and the formula that 
\begin{equation}\label{mult}
m_{G,0}^{\iota, ?}(\theta, \lambda)=\sum_{\sigma|_{R_{\mathcal{S},p}}=\theta}m_{G,0}^{\iota, ?}(\sigma, \lambda)\dim{\sigma}.
\end{equation}
\end{proof}

\vspace{3pc}
\section{TWISTED EIGENVARIETIES}
In this section, assuming that $e_G:=e_{G,\iota}\neq 0$, we construct eigenvarieties which parametrize $\iota$-invariant finite slope overconvergent Hecke eigensystems of $G$. We call such an eigenvariety a $\iota$-twisted eigenvariety of $G$.

\subsection{Twisted spectral varieties}
Consider the effective distribution $e_GI^\dagger_{G.0}(\iota\times f,\lambda)$. For fixed $K^p$ and $f\in\mathcal{H}_p(K^p)$, write $V^{\dagger,\lambda}_{G,0}(K^p)$ and $P^\dagger_{G.0}(\iota\times f,\lambda, X)$ as last section.

\begin{prop}[twisted spectral varieties]
For any $f=f^p\otimes u_{t}\in \mathcal{H}'_p(K^p)$ with $t\in T^{++}$, there is a rigid analytic space $\mathfrak{S}^\iota(f)\subset \mathfrak{X}^\iota\times\mathbb{A}^1_{rig}$, such that $(\lambda, \alpha)\in\mathfrak{S}^\iota(f)(\overline{\mathbb{Q}}_p)$ if and only if $\alpha^{-1}$ is an eigenvalue of $\iota\times f$ on $V^{\dagger,\lambda}_{G,0}(K^p)$.
\end{prop}
\begin{proof}
This is same to \cite[Proposition 5.1.6]{Urban}. $\mathfrak{S}^\iota(f)$ is simply defind as the Fredholm hypersurface cut out by $P^\dagger_{G.0}(\iota\times f,\lambda, X)$ in $\mathfrak{X}^\iota\times\mathbb{A}^1_{rig}$.
\end{proof}

\subsection{Full eigenvariety}
For later use, we summarize some results of \cite{Xiang}. Given $K^p$, let $\widetilde{R}_{\mathcal{S},p}$ be the $p$-adic completion of $R_{\mathcal{S},p}[u^{-1}_t, t\in T^+]$. Define $\mathfrak{R}_{\mathcal{S},p}$ as the $p$-adic analytic space, such that for any $L/\mathbb{Q}_p$ in $\overline{\mathbb{Q}}_p$, 
\begin{equation}
\mathfrak{R}_{\mathcal{S},p}(L)=Hom_{ct\ alg}(\widetilde{R}_{\mathcal{S},p}, L)
\end{equation}
By construction, $\theta\in\mathfrak{R}_{\mathcal{S},p}(L)$ is of finite slope. $\mathfrak{R}_{\mathcal{S},p}(L)$ has the canonical $p$-adic topology induced by the metric $|\theta-\theta'|=:\sup_{f\in R_{\mathcal{S},p}}|\theta(f)-\theta'(f)|_p$.

\vspace{1pc}
Set $\mathfrak{Y}=\mathfrak{X}\times \mathfrak{R}_{\mathcal{S},p}$, its $L$-points $\mathfrak{Y}(L)$ are pairs $(\lambda, \theta)$. The full eigenvariety is a rigid analytic space $\mathfrak{E}:=\mathfrak{E}_{K^p}$ over $\mathbb{Q}_p$, which is a subspace of $\mathfrak{Y}$ as a $p$-adic analytic space. $\mathfrak{E}$ is equipped with a projection onto $\mathfrak{X}$, such that $(\lambda, \theta)\in\mathfrak{E}_{K^p}(L)$ if and only if $H_{fs}^*(S_G(K^pI), \mathcal{D}_\lambda(L))[\theta]\neq 0$, and, for any $f\in R_{\mathcal{S},p}$, $R_f(\theta):=\theta(f)^{-1}$ is an eigenvalue of $f$ acting on $H^*(S_G(K^pI),\mathcal{D}_\lambda(L))$. Indeed, for any $f\in {R}_{\mathcal{S},p}$, $(\lambda, \theta)\mapsto(\lambda, R_f(\theta))$ gives a projection from $\mathfrak{E}_{K^p}$ onto a subvariety $\mathfrak{S}_f$ of $\mathfrak{X}\times \mathbb{A}^1_{rig}$, and $\mathfrak{S}_f$ is a piece of the spectral variety which parameterizes Hecke eigenvalues of $f$. For detail, refer \cite[\S6, 7, 8]{Xiang}.

\subsection{a big twisted eigenvariety}
We constructed our first twisted eigenvariety $\mathfrak{G}^\iota$ in this section using the method and notation in \cite{Xiang}. Most results are parallel to \cite{Xiang}, so we omit the proofs. $\mathfrak{G}^\iota$ gives us a big space so that we can construct other twisted eigenvarities inside it.

\subsubsection{Spectral varieties}
\noindent
Let $\mathfrak{U}$ be an open affinoid subdomain of $\mathfrak{X}^\iota$, consider the action of $^\iota R_{\mathcal{S}, p}$ on $R\Gamma(K^pI,\mathcal{D}_\mathfrak{U}):=\oplus R\Gamma^q(K^pI,\mathcal{D}_\mathfrak{U})$ as in \S3.1. By Proposition \ref{prop} and the discussion in \S3.1.3, for any $f\in R_{\mathcal{S},p}$ admissible, there is a power series $P_{\mathfrak{U}}(f,\lambda,X)\in\mathcal{O}(\mathfrak{U})\{\{X\}\}$,  such that for any $\lambda\in\mathfrak{U}$, the specialization of $P_{\mathfrak{U}}(f,\lambda,X)$ at $\lambda$ is the Fredholm determinant of $f$ acting on $R\Gamma(K^pI,\mathcal{D}_\lambda)$.

\begin{lem}[Urban]\label{urbanlem}
Let $j: N\hookrightarrow M$ be a continuous injection of $L$-Banach spaces. Let $u_N$ and $u_M$ be respectively compact endomorphisms of $N$ and $M$ such that $j\circ u_N=u_M\circ j$. Then $M/j(N)$ has slope decomposition with respect to $u_{M/N}=u_M$ $($ mod $j(N)$ $)$, and
$$\det(1-Xu_M)=\det(1-Xu_N)\det(1-Xu_{M/N})$$
\end{lem}
This lemma is \cite[Proposition 2.3.9]{Urban}. Apply lemma \ref{urbanlem} to the situation $M=N=R\Gamma(K,\mathcal{D}_\mathfrak{U})$ and $j=\iota*$, we have
\begin{equation}
P_{\mathfrak{U}}(f,\lambda,X)=P_{\mathfrak{U}}(f^\iota,\lambda,X)
\end{equation}
Let $P_{\mathfrak{U}}(f,\lambda,X)=Q_{\mathfrak{U}}(X)S_{\mathfrak{U}}(X)$ be a polynomial decomposition as in Lemma \ref{slopedecomp}, and  
\begin{equation}
R\Gamma(K,\mathcal{D}_\mathfrak{U})=N_f(Q_\mathfrak{U})\oplus F_f(Q_\mathfrak{U})
\end{equation}
the corresponding $\mathcal{O}(\mathfrak{U})$-module decomposition. Apply Lemma \ref{urbanlem} to the situation $N=N_f(Q_\mathfrak{U})$, $M=R\Gamma(K,\mathcal{D}_\mathfrak{U})$, $j=\iota*$, $u_N=f$ and $u_M=f^\iota$, we have
\begin{equation}
N_{f^\iota}(Q_\mathfrak{U})=\iota*(N_f(Q_\mathfrak{U}))
\end{equation}
Define
\begin{equation}
N_{f,\iota}(Q_\mathfrak{U})=\bigcap_{i=1}^l N_{f^{\iota^i}}(Q_\mathfrak{U})
\end{equation}
Now $^\iota\mathcal{H}_p(K^p)$ acts on $N_{f,\iota}(Q)$.

% % begin{remark}$N_{f,\iota}$ is not always trivial. For example, if $\pi$ is an irreducible representation of $\mathcal{H}_p$ lying in the set $A$ as in the proof of Proposition \ref{induction1}. Then there is some $Q(X)$ of degree large enough such that $V_\sigma\subset N_{f^{\iota^i}}(Q)$ for every $i$. end{remark}

\vspace{1pc}
Recall that, in \cite[Proposition 6.4]{Xiang}, we defined a weight space $\mathfrak{W}\subset\mathfrak{X}$, such that for $\lambda\in\mathfrak{X}$, $\lambda\in\mathfrak{W}$ if and only if $H_{fs}(S_G(K),\mathcal{D}_\lambda)\neq 0$. Similarly, define
\begin{equation}
\mathfrak{W}_{Q, \mathfrak{U}}^\iota:= supp_{\mathcal{O}(\mathfrak{U})}\widetilde{H(N_{f,\iota}^*(Q_\mathfrak{U}))}\subset\mathfrak{U}
\end{equation}
and $\mathfrak{W}^\iota$ the subspace of $\mathfrak{X}^\iota$ obtained by gluing $\mathfrak{W}_{Q,\mathfrak{U}}^\iota$ for all $\mathfrak{U}$ and $Q_\mathfrak{U}$. Then $\lambda\in\mathfrak{W}_Q^\iota(f)(\overline{\mathbb{Q}}_p)$ if and only if $H(N_{f,\iota}^*(Q))\neq 0$. As a direct consequence of the fact that
\begin{equation}
H_{fs}(S_G(K),\mathcal{D}_\lambda)=\varinjlim_{h}H(S_G(K),\mathcal{D}_\lambda)^{\leq h}_f=\varinjlim_{h}\bigcap_iH(S_G(K),\mathcal{D}_\lambda)^{\leq h}_{f^{\iota^i}},
\end{equation}
we have: 
\begin{prop}
\begin{equation}
\mathfrak{W}^\iota=\mathfrak{X}^\iota\cap\mathfrak{W}
\end{equation}
\end{prop}

\vspace{1pc}
For an admissible $f\in R_{\mathcal{S},p}$, define set $\set{f}^\iota:=\{f^{\iota^i}\times\iota^j | 1\leq i,j\leq l\}$. For any $g\in\set{f}^\iota$, define
\begin{equation}
\mathfrak{S}^\iota_{Q,\mathfrak{U},g}:=supp_{\mathcal{O}(\mathfrak{U})[g]}\widetilde{H^*_{f,\iota}(Q)}\subset \mathfrak{U}\times\mathbb{A}^1_{rig}
\end{equation}

\begin{prop}
$\mathfrak{S}^\iota_{Q,\mathfrak{U}, g}$ is locally finite over $\mathfrak{W}^\iota_{Q,\mathfrak{U}}$. A point $s=(\lambda, \alpha)$ of $\mathfrak{U}\times\mathbb{A}^1_{rig}$ is in $\mathfrak{S}^\iota_{Q,\mathfrak{U},g}(\overline{\mathbb{Q}}_p)$ if and only if $\lambda\in\mathfrak{W}^\iota_{Q,\mathfrak{U}}(\overline{\mathbb{Q}}_p)$ and $\alpha^{-1}$ is an eigenvalue of $g$ acting on $H^*_{f,\iota}(Q)$.
\end{prop}
The proof is same to \cite[Proposition 6.6]{Xiang}. Moreover, discuss as in \cite[\S6]{Xiang}, given $\set{f}^\iota$ and $g\in\set{f}^\iota$, we can glue the local spectral varieties $\mathfrak{S}^\iota_{Q,\mathfrak{U},g}$ for polynomials $Q_\mathfrak{U}$ and open affinoid domains $\mathfrak{U}\subset\mathfrak{X}^\iota$:
\begin{thm}
There is a spectral variety $\mathfrak{S}^\iota_{g}=\mathfrak{S}^\iota_{\mathfrak{W}^\iota, g}$ as a rigid subspace of $\mathfrak{X}^\iota\times \mathbb{A}^{rig}_1$, such that, $s=(\lambda, \alpha)\in\mathfrak{S}^\iota_{g}(\overline{\mathbb{Q}}_p)$ if and only if $\lambda\in\mathfrak{W}^\iota(\overline{\mathbb{Q}}_p)$ and $\alpha^{-1}$ is an eigenvalue of $g$ acting on $H^*_{fs}(S_G(K),\mathcal{D}_\lambda)$.
\end{thm}

\begin{cor}\label{cor2}
If $g=\iota\times f$, then 
\begin{equation}
\mathfrak{S}^\iota(f)\subset \mathfrak{S}^\iota_{g}.
\end{equation}
\end{cor}

\subsubsection{A big twisted eigenvariety}
We build an eigenvariety over the spectral varieties constructed in last subsection as in \cite[\S8]{Xiang}. Let $^\iota\tilde{R}_{\mathcal{S},p}$ be the $p$-adic completion of $^\iota R_{\mathcal{S},p}[u_t^{-1}, t\in T^{+}]$. Since $^\iota R_{\mathcal{S},p}=R_{\mathcal{S},p}\rtimes \langle\iota\rangle$ and $\iota$ is of finite order $l$, that $^\iota\tilde{R}_{\mathcal{S},p}$ is an $l$ pieces union of $\tilde{R}_{\mathcal{S},p}$. Define the $p$-adic space $\mathfrak{B}=\mathfrak{B}_{\mathcal{S},p}$ be such that for any $L/\mathbb{Q}_p$,
\begin{equation}
\mathfrak{B}(L)=Hom_{alg\ ct}(^\iota\tilde{R}_{\mathcal{S},p}, L).
\end{equation}
There is a natural morphism
\begin{equation}\label{i1}
\mathfrak{i}: \mathfrak{B}_{\mathcal{S},p}(L)\longrightarrow\mathfrak{R}_{\mathcal{S},p}(L)
\end{equation}
given by restricting a character $\tilde{\theta}$ of $^\iota\tilde{R}_{\mathcal{S},p}$ to $\tilde{R}_{\mathcal{S},p}$, i.e. $\theta:=\mathfrak{i}(\tilde\theta)=\tilde{\theta}|_{\tilde{R}_{\mathcal{S},p}}$. This morphism is finite and continuous.

\vspace{1pc}
Define $\mathfrak{Z}^\iota=\mathfrak{Z}^\iota_{\mathcal{S},p}:=\mathfrak{X}^\iota\times \mathfrak{B}$. For any admissible $f $ and $g\in\set{f}^\iota\subset {^\iota R_{\mathcal{S},p}}$, we define the morphism of ringed space 
\begin{equation}
R_g: \mathfrak{Z}^\iota\rightarrow\mathfrak{X}^\iota\times \mathbb{A}^{rig}_1\end{equation}
 by $(\lambda, \tilde{\theta})\mapsto (\lambda, \tilde{\theta}(g)^{-1})$ on $L$-points, and
\begin{equation}R_g^{*}: \mathcal{O}(\mathfrak{X}^\iota)\{\{X\}\}\longrightarrow \mathcal{O}(\mathfrak{X}^\iota)\hat{\otimes}{^\iota\tilde{R}_{\mathcal{S},p}}\end{equation}
by $\sum a_nX^n\mapsto \sum a_n(g)^{-n}$ on the function rings. 

\vspace{1pc}
Define the rigid space,
\begin{equation}\tilde{\mathfrak{D}}^\iota:=\prod_{[f]^\iota}\prod_{g\in[f]^\iota}R_g^{-1}\mathfrak{S}_g^\iota\end{equation}
as in \cite[\S8]{Xiang}, where its $G$-topology is defined via $R_g$'s. Concretely, an open subset of $\tilde{\mathfrak{D}}^\iota$ is admissible if it is a union of open subsets of the form $R_{g,1}\times\cdots\times R_{g_r}^{-1}(\mathfrak{V})$ for $\mathfrak{V}$ an open admissible affinoid of $\mathfrak{S}_{g_1}^\iota\times\cdots\times\mathfrak{S}_{g_r}^\iota$; and an admissible covering is the inverse images by $R_g$'s of the admissible coverings of the corresponding spectral varieties. Then we have a parallel result to \cite[Proposition 8.1]{Xiang}:
\begin{prop}\label{prop4}
Assume $\tilde{y}=(\lambda, \tilde\theta)$ is in $\mathfrak{Z}^\iota(\overline{\mathbb{Q}}_p)$, then $\tilde{y}\in\tilde{\mathfrak{D}}^\iota(\overline{\mathbb{Q}}_p)$ if and only if $H^*(S_G(K),\mathcal{D}_\lambda)[\tilde{\theta}]\neq 0$ as a $^\iota R_{\mathcal{S},p}$-module. Moreover, given $\tilde{y}\in\tilde{\mathfrak{D}}^\iota(\overline{\mathbb{Q}}_p)$, there exists an admissible $f$, such that 
\begin{equation}
\bigcap_{g\in[f]^\iota}R_g^{-1}(R_g(\tilde{y}))\bigcap\tilde{\mathfrak{D}}^\iota(\overline{\mathbb{Q}}_p)=\{\tilde{y}\}
\end{equation}
\end{prop}

\vspace{1pc}
For $\mathfrak{U}\subset\mathfrak{X}^\iota$ and $P_\mathfrak{U}(f, X)=Q(X)S(X)$ as in \S6.3.1, let $h_\mathfrak{U}$ and $h_\mathfrak{U}^\iota$ be the image of $R_\mathfrak{U}:=\mathcal{O}(\mathfrak{U})\otimes R_{\mathcal{S},p}$ and $R_\mathfrak{U}^\iota:=\mathcal{O}(\mathfrak{U})\otimes{^\iota R_{\mathcal{S},p}}$ in $End_{pf}^b(R\Gamma(K^pI, \mathcal{D}_\mathfrak{U}))$ respectively, and let $h_{\mathfrak{U},Q}$ and $h_{\mathfrak{U},Q}^\iota$ be the image of $R_{\mathfrak{U}}$ and $R_{\mathfrak{U}}^\iota$ in $End_{pf}^b(N_{f, \iota}(Q))$ respectively. Define
\begin{equation}
\tilde{\mathfrak{G}}^{\iota'}_\mathfrak{U}:=sp(h^\iota_\mathfrak{U})
\end{equation}
 and 
 \begin{equation}
 \tilde{\mathfrak{G}}^{\iota}_{\mathfrak{U}, Q}:= supp_{h_{\mathfrak{U},Q}^\iota}\widetilde{H(N^*_{f,\iota}(Q))}
 \end{equation}

\begin{prop}\label{prop5}
$$\tilde{\mathfrak{G}}^{\iota}_{\mathfrak{U}, Q}(\overline{\mathbb{Q}}_p)=\prod_{g\in [f]^\iota}R_g^{-1}\mathfrak{S}^\iota_{Q,\mathfrak{U},g}(\overline{\mathbb{Q}}_p)$$
\end{prop}
The proof of proposition is same to \cite[Proposition 8.2]{Xiang}. Moreover, an argument as \cite[Proposition 8.2, 8.3, 8,4]{Xiang} shows that we can patch $\tilde{\mathfrak{G}}^{\iota}_{\mathfrak{U}, Q}$ with respect to $\mathfrak{U}$ and $Q$ to obtain a rigid space $\tilde{\mathfrak{G}}^{\iota}_{f}$.  Define 
\begin{equation}\tilde{\mathfrak{G}}^{\iota}:=\prod_{f}\tilde{\mathfrak{G}}^{\iota}_{f}.\end{equation}
It is a reduced rigid analytic space, and by Proposition \ref{prop4}, \ref{prop5} \begin{equation}\tilde{\mathfrak{G}}^{\iota}(\overline{\mathbb{Q}}_p)=\tilde{\mathfrak{D}}^\iota(\overline{\mathbb{Q}}_p).\end{equation}

\vspace{1pc}
Now given $\mathfrak{U}$ and $Q$ as above, define $\mathfrak{i}: \tilde{\mathfrak{G}}^{\iota}\rightarrow \mathfrak{E}$ by locally defining:
\begin{equation}\label{i2}
\mathfrak{i}: \tilde{\mathfrak{G}}^{\iota}_{\mathfrak{U}, Q}:= supp_{h_{\mathfrak{U},Q}^\iota}\widetilde{H(N^*_{f,\iota}(Q))}\rightarrow\mathfrak{E}_{\mathfrak{U}, Q}:=supp_{h_{\mathfrak{U},Q}}\widetilde{H(N^*_{f}(Q))}.
\end{equation}
This is defined by the inclusions $h_{\mathfrak{U},Q}\hookrightarrow h^\iota_{\mathfrak{U},Q}$ and $H(N^*_{f,\iota}(Q))\rightarrow H(N^*_{f}(Q))$. In particular, on $\overline{\mathbb{Q}}_p$-points, $\mathfrak{i}$ sends $(\lambda, \tilde\theta)$ to $(\lambda, \theta:=\tilde\theta|_{R_{\mathcal{S},p}})$. So it is defined on each fibre $\tilde{\mathfrak{G}}^{\iota}_{f}$ coincident with (\ref{i1}) and defined locally on points $R_g^{-1}\mathfrak{S}^\iota_{Q,\mathfrak{U},g}(\overline{\mathbb{Q}}_p)$.  Finally, define $\mathfrak{G}^\iota$ as the image of $\tilde{\mathfrak{G}}^\iota$ under $\mathfrak{i}$. Its points are described by the theorem below

\begin{thm}\label{big}
Assume $y=(\lambda,\theta)$ is in $\mathfrak{E}(\overline{\mathbb{Q}}_p)$, then $y\in\mathfrak{G}^\iota(\overline{\mathbb{Q}}_p)$ implies that $\theta$ is a $\iota$-invariant finite slope overconvergent Hecke eigensystem of weight $\lambda$. For any $f\in R_{\mathcal{S},p}$, $R_f$ maps $\mathfrak{G}^\iota$ to $\mathfrak{S}^\iota_{\mathfrak{W}^\iota, f}$. It is locally finite and surjective. In particular, $\dim\mathfrak{G}^\iota\leq\dim\mathfrak{X}^\iota$.
\end{thm}
This is directly follows from the definition, noting that $f\in\set{f}^\iota$. 

\begin{remark}\label{big2}
Generally, $\mathfrak{G}^\iota$ is NOT the eigenvariety parameterizing all $\iota$-invariant Hecke eigensystems. Actually, it parameterizes those $\theta$ such that, as a one-dimensional subspace of $H(S_G(K), \mathcal{D}_\lambda)$, $\iota*$ maps $V_\theta$ to itself, as the set $A$ in the proof of Proposition \ref{induction1}.
\end{remark}

\subsection{Twisted eigenvariety}
Now fix $K^p$, for any $f\in R_{\mathcal{S},p}$, consider the morphism 
$R_{\iota\times f}: \mathfrak{Z}^\iota \rightarrow \mathfrak{X}^\iota\times \mathbb{A}^1_{rig}$ defined as in \S6.3.7, 
and define
\begin{equation}
\tilde{\mathfrak{E}}^\iota=\prod_f (R_{\iota\times f})^{-1}(\mathfrak{S}^\iota(f)),
\end{equation}
where the product is running over admissible Hecke operators $f$.

\begin{thm}[twisted eigenvarities]\label{mainthm}
Given $K^p$ as above, there is a subvariety $\mathfrak{E}_{K^p}^\iota$ of $\mathfrak{E}_{K^p}$, satisfying:
 \begin{itemize}
\item[$(a)$] For any $(\lambda,\theta)\in\mathfrak{E}_{K^p}(\overline{\mathbb{Q}}_p)$, $(\lambda,\theta)$ is in $\mathfrak{E}^\iota_{K^p}(\overline{\mathbb{Q}}_p)$ if and only if $\theta$ is $\iota$-invariant and $m^{\dagger,\iota}_{G,0}(\theta,\lambda)\neq0$.
\item[$(b)$] Every irreducible component of $\mathfrak{E}^\iota_{K^p}$ projects surjectively onto a Zariski dense subset of $\mathfrak{X}^\iota$.
\item[$(c)$] $\mathfrak{E}_{K^p}^\iota$ is equidimensional with the same dimension to $\mathfrak{X}^\iota$,  and every irreducible component is arithmetic.
\end{itemize}
\end{thm}

\begin{proof}
Define $\mathfrak{E}^\iota:=\mathfrak{E}^\iota_{K^p}$ to be the image of $\tilde{\mathfrak{E}}^\iota$ under $\mathfrak{i}$. Its underlying topological space $\mathfrak{E}^\iota(\overline{\mathbb{Q}}_p)$ is given by the image $ \mathfrak{i}(\tilde{\mathfrak{E}^\iota}(\overline{\mathbb{Q}}_p))$. Firstly we show $\mathfrak{E}^\iota(\overline{\mathbb{Q}}_p)\subset\mathfrak{E}_{K^p}(\overline{\mathbb{Q}}_p)$ and the part $(a)$ by modifying the proof of \cite[Proposition 5.2.3]{Urban}. As in \S5.1, write 
\begin{equation}\label{decomp}
V=V^{\dagger, \lambda}_{G,0}(K^p)=\bigoplus_{\sigma}\bigoplus_{i=1}^{l} V_{\tilde\sigma_i}^{m_{G,0,i}^{\iota,\dagger}(\sigma,\lambda)},
\end{equation}
where $\sigma=\pi^{K^p}$ is running over all $\iota$-invariant finite slope representations of $\mathcal{H}_p(K^p)$ appearing in the distribution $e_{G}I^{\dagger}_{G,0}(\iota\times f,\lambda)'_{K^p}$. Given $(\lambda, \tilde{\theta})\in\tilde{\mathfrak{E}^\iota}(L)$, fix $t\in T^{++}$, set $h=v_p(\tilde{\theta}(u_t))$ and 
$$W=V^{\leq h}.$$ 
$\mathcal{H}_p(K^p)$ acts on $W$ since $R_{\mathcal{S},p}$ is in its center. Since every $\sigma$ appearing in $V$ is $\iota$-invariant, that $^\iota\mathcal{H}_p(K^p)$ acts on $W$. Let $h_W$ be the image of $R_{S,p}\rightarrow End_L(W)$. It is finitely generated by the image of finitely many elements $\{f_1,\cdots f_r\}$ in $R_{\mathcal S,p}$. Let $\Omega$ be the set consisting of $\tilde{\theta}(u_t), \tilde{\theta}(f_i)$, and all eigenvalues of $u_t$, $\iota\times u_t$, $f_i$, $\iota\times f_i$ on $W$. Now let $R$ be a number such that for any $\alpha, \alpha' \in \Omega$, $v_p(\alpha-\alpha')\leq v_p(R)$, define operators $h_1=f_1$, $h_{i+1}=f_{i+1}(1+Rh_i)$ and $f=u_t(1+Rh_r)$.

\vspace{1pc}
Since $(\lambda, \tilde{\theta})\in\tilde{\mathfrak{E}^\iota}(L)$, there is $0\neq w_f\in V$ and $\tilde\sigma_i$ appeaing in $V$, such that
\begin{equation}
\tilde{\theta}(\iota\times f)w_f=\sigma({\iota\times f})w_f.
\end{equation}
in particular, $w_f$ is an eigenvector of $\sigma(\iota\times u_t)$ and $\sigma(\iota)$. Denote their eigenvalue by $a_f$ and $b_f$ respectively. Since $\iota$ is of finite order, $b_f$ is a unit. Write $\theta=\sigma|_{R_{S,p}}$, then $(\lambda,\theta)\in\mathfrak{E}_{K^p}(\overline{\mathbb{Q}}_p)$, we want to show $\mathfrak{i}(\tilde{\theta})=\theta$. 

\vspace{1pc}
Indeed, $v_p(b_f)+v_p(\theta(u_t))+v_p(\theta(1+Rh_r))=v_p(\tilde{\theta}(\iota))+v_p(\tilde{\theta}(u_t))+v_p(\tilde{\theta}(1+Rh_r))$. Since $b_f$ and $\tilde{\theta}(\iota)$ are units, by the setting of $R$, $v_p(\theta(u_t))=v_p(\tilde{\theta}(u_t))$. This implies that $\sigma$ actually appears in $W$. On the other hand, $(a_f-\tilde{\theta}(\iota\times u_t))\theta(1+Rh_r)=\tilde{\theta}(\iota\times u_t)(\tilde{\theta}(1+Rh_r)-\theta(1+Rh_r))$. This implies $v_p(a_f-\tilde{\theta}(\iota\times u_t))>v_p(R)$. So $a_f=\tilde{\theta}(\iota\times u_t)$ and $\theta(h_r)=\tilde{\theta}(h_r)$. Repeating the process, we have $\tilde{\theta}(f_i)=\theta(f_i)$ for all $f_i$. Therefor $\tilde{\theta}|_{R_{\mathcal{S},p}}=\theta$ and $\mathfrak{i}(\lambda, \tilde{\theta})=(\lambda,\theta)\in\mathfrak{E}(L)$. In particular $\theta$ is $\iota$-invariant. By our construction of $\theta$ and formula (\ref{mult}), we have $m^{\dagger,\iota}_{G,0}(\theta,\lambda)\neq0$.

\vspace{1pc}
Now we prove the other direction of $(a)$. If $V_i:=V_{\tilde\sigma_i}$ appearing in (\ref{decomp}), let $V_i=\oplus_{\zeta} V_i[\zeta]$ be the eigen decomposition of $V_i$ under $\iota$, then $^\iota R_{S,p}$ acts on each $V_i[\zeta]$. Let $(\lambda, \theta)\in\mathfrak{E}(\overline{\mathbb{Q}}_p)$ be $\iota$-invariant. If $m^{\iota,\dagger}_{G,0}(\theta, \lambda)\neq 0$, there is some $V_i$ such that $V_i[\theta]\neq0$ as a $R_{S,p}$-module. In particular, $\theta$ appears in some $V_i[\zeta]$. Define $\tilde{\theta}$ be the extension of $\theta$ to $^\iota R_{S,p}$ by setting $\tilde{\theta}(\iota)=\zeta$. It is then clear that $(\lambda, \tilde{\theta})\in\tilde{\mathfrak{E}}^\iota(\overline{\mathbb{Q}}_p)$ and $\mathfrak{i}(\lambda, \tilde{\theta})=(\lambda, \theta)$.

\vspace{1pc}
With $(a)$, Theorem \ref{big} and Remark \ref{big2} imply that 
\begin{equation}\label{sitting}
\mathfrak{E}^\iota(\overline{\mathbb{Q}}_p)\subset\mathfrak{G}^\iota(\overline{\mathbb{Q}}_p)
\end{equation}

By construction, for any $f\in R_{\mathcal{S},p}$, there is a commutative diagram:
\begin{equation}\begin{array}[c]{ccccc}
\tilde{\mathfrak{E}}^\iota_{K^p}(\overline{\mathbb{Q}}_p)&\stackrel{\mathfrak{i}}{\rightarrow}&\mathfrak{E}^\iota_{K^p}(\overline{\mathbb{Q}}_p)&\hookrightarrow&\mathfrak{G}^\iota_{K^p}(\overline{\mathbb{Q}}_p)\\
\downarrow\scriptstyle{R_{\iota\times f}}& &\downarrow\scriptstyle{R_f}&&\downarrow\scriptstyle{R_f}\\
\mathfrak{S}^\iota(f)(\overline{\mathbb{Q}}_p)&&(\mathfrak{X}^\iota\times\mathbb{A}_1^{rig})\cap\mathfrak{S}_f(\overline{\mathbb{Q}}_p)&\hookrightarrow&\mathfrak{S}^\iota_f(\overline{\mathbb{Q}}_p)\\
\searrow&&\swarrow&&\downarrow\\
&\mathfrak{X}^\iota&\hookrightarrow&&\mathfrak{W}^\iota=\mathfrak{X}^\iota\\
&\searrow&&\swarrow&\\
&&\mathfrak{X}&&
\end{array}\end{equation}

Consider the first column of the diagram. $R_{\iota\times f}$ is locally finite, and $\mathfrak{S}_f$ is constructed by the Fredholm power series. So by the same argument of \cite[Theorem 5.3.7]{Urban}, $R_{\iota\times f}$ is finite surjective. Now Proposition \ref{prop4} and Corollary \ref{cor2} enable us to run an argument as in \cite[Corollary 5.3.8]{Urban}, so the composition of the first two arrows in the first column is surjective onto a Zariski dense subset of $\mathfrak{X}^\iota$. Since $\mathfrak{i}$ keeps the first coordinate $\lambda$, that the projection from $\mathfrak{E}^\iota(\overline{\mathbb{Q}}_p)$ to $\mathfrak{X}^\iota$ in the second column is also Zariski surjective. So $\dim (\mathfrak{E}^\iota)\geq\dim (\mathfrak{X}^\iota)$. However, (\ref{sitting}) and Theorem \ref{big} imply that $\dim (\mathfrak{E}^\iota)\leq\dim(\mathfrak{G}^\iota)=\dim(\mathfrak{W}^\iota)\leq\dim(\mathfrak{X}^\iota)$. So we have $\dim (\mathfrak{E}^\iota)=\dim (\mathfrak{X}^\iota)$, $\mathfrak{X}^\iota=\mathfrak{W}^\iota$ ( This is the reason of the third row of the diagram) and in particular the $R_f$ in the third column is also surjective. Together with Proposition 6.4.8, an argument as in \cite[Corollary 5.3.8]{Urban} again shows that $R_f$ in the second column is also locally finite and surjective. 

\vspace{1pc}
Now since $\dim(\mathfrak{G}^\iota)=\dim(\mathfrak{X}^\iota)$, denote by $\mathfrak{G}^{\iota,c}$ the union of codimension $0$ arithmetic components (i.e. irreducible components contains a Zariski dense subset of arithmetic points) of $\mathfrak{G}^\iota$. Then the same argument as in \cite[Proposition 8.9]{Xiang} shows that 
\begin{equation}\label{rigid}
\mathfrak{G}^{\iota,c}(\overline{\mathbb{Q}}_p)=\mathfrak{E}^\iota(\overline{\mathbb{Q}}_p)\end{equation} 
This proves $(b)$ and $(c)$ of the theorem. 
\end{proof}

\begin{remark}
By the last observation (\ref{rigid}) in the proof above, throughout this section we can work on the $p$-adic spaces $\mathfrak{E}^\iota(\overline{\mathbb{Q}}_p)$ and finally define the rigid analytic structure on $\mathfrak{E}^\iota$ by the one induced from $\mathfrak{G}^{\iota,c}$. 
\end{remark}

\vspace{3pc}
\section{THE CASE OF $Gl_n$}
\indent 
In this section, we study the case of $G=Gl_n$ over $\mathbb{Q}$. Fix the pair $(B, T)$, where $B$ is the subgroup of upper triangular matrices and $T$ the diagonal subgroup. Then 
\begin{equation}
T^{+}=\set{diag(t_1,\cdots,t_n)\mid v_p(t_1)\geq\cdots\geq v_p(t_n)}
\end{equation}
\begin{equation}
T^{++}=\set{diag(t_1,\cdots,t_n)\mid v_p(t_1)>\cdots> v_p(t_n)}
\end{equation}
Define $g^\iota:=j{(^tg^{-1}})j^{-1}$ for any $g\in G$, where $j=(\delta_{i,n+1-j})_{1\leq i, j\leq n}$ if $n$ is odd; and $$j=\begin{pmatrix} 0 & (\delta_{i,k+1-j})_{1\leq i, j\leq k}\\ -(\delta_{i,k+1-j})_{1\leq i, j\leq k} & 0 \end{pmatrix}$$ if $n=2k$ is even. It is easy to check that $\iota$ is an automorphism of $G$ of order $2$, and $\iota$ stabilizes $(B, T, I_m)$ and $T^+$, $T^{++}$. 

\vspace{1pc}
Let $\pi$ be an automorphic representation of $G$. $\pi$ is $\iota$-invariant if and only if $\pi$ is self-dual in the usual sense. 

\subsection{ $I^\dagger_{G,0}(\iota\times f,\lambda)$ is non-trivial}
Consider the $\iota$-twisted distribution $I^\dagger_{G,0}(\iota\times f,\lambda)$ defined for $G$ and  $\iota$ as in \S5.2, we prove it is non-trivial by computing $q_{G,\iota}$ as in Proposition \ref{prop2}. 
\begin{prop}
Let $\lambda\in \mathfrak{X}^\iota$ be an arithmetic regular dominant weight in $\Upsilon$. Let $\pi$ be a self-dual finite slope cuspidal representation of weight $\lambda$. Assume that $\pi$ is non-critical with respect to $\lambda$. Then $q_{G,\iota}\neq0$. 
\end{prop}
\begin{proof}
It is a computation given by Barbasch and Speh in \cite[VI.3]{BS} that  
\begin{equation}
\begin{aligned}
L(\iota,A_\mathfrak{b}(\lambda),\lambda)=&(-1)^{R_\mathfrak{b}}\sum_{i}(-1)^itr(\iota |Hom_{\mathfrak{l\cap k}}(\wedge^i(\mathfrak{l\cap p}), \mathbb{C}))\\
=&(-1)^{\lceil\frac{n}{2}\rceil}2^{\lceil\frac{n}{2}\rceil}\neq 0
\end{aligned}
\end{equation}
\end{proof}

\vspace{1pc}
Since $\iota$ is of order $2$, for $?=cl$ or $\dagger$, we define $e_{G,\iota}I^{?}_{G,0}(\iota\times f,\lambda)$ as in Remark \ref{ob3}. Considering the last paragraph in \S5.5, we have 
\begin{cor}\label{epcor}
Let $\lambda\in \mathfrak{X}^\iota$ be an arithmetic regular dominant weight in $\Upsilon$. Let $\pi$ be a self-dual finite slope cuspidal representation of weight $\lambda$. Assume that $\pi$ is non-critical with respect to $\lambda$. Then $m_{EP}^\iota(\sigma,\lambda)\neq 0$.
\end{cor}
\begin{proof}
By the definition (\ref{epdef}), combining (\ref{sum}) and (\ref{ep}), we compute as in \S5.2:
\begin{equation} 
m_{EP}^\iota(\pi,\lambda)=\sum_{\rho}m_{cusp}(\rho)L(\iota,\rho,\lambda)tr(\iota | Hom_{\mathcal{H}_p(K^p)}(\pi_f^{K^p}, \rho_f)).
\end{equation}
Since $G=Gl_n$, the cohomological packet at infinity for $\lambda$ has only one element, which is of the form $A_\mathfrak{b}(\lambda)$ as in Theorem 5.2.3 (see also \cite{Speh}). So
\begin{equation}
\begin{aligned}
&\ \ \ \  m_{EP}^\iota(\pi,\lambda)\\
&=L(\iota, A_\mathfrak{b}(\lambda),\lambda)\sum_{\rho_\infty=A_\mathfrak{b}(\lambda)}m_{cusp}(\rho)tr(\iota | Hom_{\mathcal{H}_p(K^p)}(\pi_f^{K^p}, \rho_f))\\
&=(-1)^{\lceil\frac{n}{2}\rceil}2^{\lceil\frac{n}{2}\rceil}\sum_{\rho_\infty=A_\mathfrak{b}(\lambda)}tr(\iota | Hom_{\mathcal{H}_p(K^p)}(\pi_f^{K^p}, \rho_f))\\
&=(-1)^{\lceil\frac{n}{2}\rceil}2^{\lceil\frac{n}{2}\rceil}\neq0,
\end{aligned}
\end{equation}
where the last two lines hold since $Gl_n$ admits the multiplicity one theorem.
\end{proof}

\begin{remark}
Corollary \ref{epcor} implies that the $\iota$-twisted eigenvariety $\mathfrak{E}^\iota_{K^p}$ we constructed in Theorem \ref{mainthm} for $Gl_n$ parameterizes all non-critical self-dual finite slope cuspidal Hecke eigensystems of level $K^p$.
\end{remark}

\subsection{Essentially self-dual representations}
A representation $\pi$ is essentially $\iota$-invariant if there exists an algebraic character $\chi$ of $\mathbb{G}_m$ such that $\chi\circ \det\otimes\pi^\iota\cong \pi$.In case $G=Fl_n$, $\pi$ is essentially $\iota$-invariant if and only if it is essentially self-dual in the usual sense.

\vspace{1pc}
A weight $\lambda\in \mathfrak{X}(L)$ is determined by $n$ $p$-adic characters $\chi_1,\cdots, \chi_n$ of $\mathbb{G}_m(\mathbb{Z}_p)$ such that
\begin{equation}
\lambda: diag(t_1,\cdots,t_n)\mapsto\chi_1(t_1)\cdots\chi_n(t_n)
\end{equation}
An essentially self-dual weight is characterized by $\chi_i\chi_{n+1-i}=\chi_j\chi_{n+1-j}$ for any $1\leq i,j\leq n$. Denote by $\mathfrak{X}^{e}$ the subspace of essentially self-dual weights in $\mathfrak{X}$, then $\dim(\mathfrak{X}^e)=[\frac{n}{2}]+1$. Given a character $\chi$, denote by $\mathfrak{X}^e_\chi$ the subspace of essentially self-dual weights with respect to $\chi$ in $\mathfrak{X}^e$. It  is cut out by the relation $\chi_i\chi_{n+1-i}=\chi_j\chi_{n+1-j}=\chi$, then $\dim(\mathfrak{X}^e_\chi)=[\frac{n}{2}]$.

\vspace{1pc}
We now construct $\mathfrak{E}^e$, an eigenvariety which parameterizes all essentially self-dual finite slope overconvergent Hecke eigensystems of $G$, by applying our method to the group $\tilde{G}=Gl_n\times Gl_1$, with involution $\mu: (g,x)\mapsto (g^\iota, \det(g)x)$. 

\vspace{1pc}
Consider the weight space $\tilde{\mathfrak{X}}=\mathfrak{X}\times \mathfrak{B}^1$. Here $\mathfrak{B}^1$ is the $p$-adic weight space of $Gl_1$, it is the rigid unit ball. Denote by $\tilde{\mathfrak{X}}^\mu$ the $\mu$-invariant subspace of $\tilde{\mathfrak{X}}$. It is easy to check that
\begin{equation}
 \tilde{\mathfrak{X}}^\mu=\{\tilde{\lambda}=(\lambda, \chi) | \lambda\in\mathfrak{X}^e_\chi , \chi\in\mathfrak{B^1}\}.
 \end{equation}
Its first component projects bijectively to $\mathfrak{X}^e$.

\vspace{1pc}
Consider an open compact subgroup  $\tilde{K}_f=K_f\times K^1_f$ of $\tilde{G}(\mathbb{A}_f)$, where $K_f$ is defined as previous sections and $K^1_f$ is a neat open compact subgroup of $Gl_1(\mathbb{A}_f)$. We simply fix $K^1_p=\hat{\mathbb{Z}}$ and normalize the Haar measure on $Gl_1$ such that $meas(\mathbb{Z}_l)=1$ for any finite place $l$. Then 
\begin{equation}
S_{\tilde{G}}(\tilde{K}_f)\cong S_G(K_f)\times (\mathbb{Q}^\times\backslash\mathbb{A}^\times_\mathbb{Q}/\hat{\mathbb{Z}}\mathbb{R}^\times)\simeq S_G(K_f)\times \set{pt}.
\end{equation}
For the group $Gl_1$, it is easy to see that $T_{Gl_1}^{+}=T_{Gl_1}^{++}=T_{Gl_1}(\mathbb{Q}_p)=\mathbb{Q}_p^\times$. So $\mathcal{U}_p(Gl_1)=\mathbb{Z}_p[T^+/T(\mathbb{Z}_p)]=\mathbb{Z}_p[\mathbb{Q}_p^\times/\mathbb{Z}_p^\times]$. This implies that the $p$-adic Hecke algebra for $\tilde{G}$ is given by:
\begin{equation}
\begin{aligned}
\mathcal{H}_p(\tilde{K}^p)&=\mathcal{H}_p(K^p)\times C_c^\infty(\hat{\mathbb{Z}}^p\backslash\mathbb{A}_f^{\times p}/\hat{\mathbb{Z}}^p)\otimes\mathbb{Z}_p[\mathbb{Q}_p^\times/\mathbb{Z}_p^\times]\\
&=\mathcal{H}_p(K^p)\times C_c^\infty(\mathbb{A}_f^{\times }/\hat{\mathbb{Z}},\mathbb{Z}_p)
\end{aligned}
\end{equation}
For any $f\in \mathcal{H}'_{p}(\tilde{K}^p)$ admissible, write $f=(f_G, f_1)$ with $f_G=f_G^p\otimes u_t$, $t\in T^{++}$ and $f_1^p\otimes u_{t'}$ $t'\in\mathbb{Q}_p^\times$.

\vspace{1pc}
For $\tilde{\lambda}=(\lambda,\chi)\in\tilde{\mathfrak{X}}^\mu$ a regular dominant weight, let $\mathbb{V}_{\tilde{\lambda}}$ be the finite dimensional irreducible algebraic representation of $\tilde{G}$ with highest weight $\tilde{\lambda}$ and $\mathcal{D}_{\tilde{\lambda}}$ the local distribution space defined in \S2.3. It is not hard to see
\begin{equation}\label{esstwist}
\mathbb{V}_{\tilde{\lambda}}\cong\mathbb{V}_\lambda\times\chi
\end{equation}
where the right side is understood as the space $\mathbb{V}_\lambda$ together with an action of $Gl_1$ given by multiplying the value of $\chi$ and the isomorphism is given by $\phi\mapsto\phi_1$ such that $\phi_1(g):=\phi(g,1)$ for any $g\in G$. Generally, let $\pi$ be an irreducible representation of $\tilde{G}$, since $Gl_1$ is in the center of $\tilde{G}$, $\pi|_{Gl_1}$ is given by a character $\chi$ and $\pi|_{G}$ is irreducible as well. It is easy to check that $\pi\cong\pi|_{G}\times\chi$, and $\pi$ is $\mu$-invariant if and only if $\pi|_{G}$ is essentially $\iota$-invariant with respect to $\chi$. If $\pi$ is a $\mu$-invariant representation of $\tilde{G}$, as in \S3.3, we can extend $\pi$ to a representation $\tilde{\pi}$ of $\tilde{G}\rtimes\langle\mu\rangle$, then we restrict $\tilde{\pi}$ to $G\rtimes\langle\mu\rangle$. This gives an $\mu$-action on $V_\pi$ such that for any $g\in G$,
\begin{equation}\label{twistrestrict}
\mu\times\chi(\det(g))\pi(g^\iota)=\pi(g)\times\mu
\end{equation}

\vspace{1pc}
Now we can define $\mu$-twisted finite slope character distributions $I^{cl}_{?}(\mu\times f, \tilde{\lambda})$ and $I^\dagger_?(\mu\times f, \tilde{\lambda})$ as $(4.4.1)$, $(4.4.3)$ and \S5.2, where $?=\tilde{G}, (\tilde{G},0), (\tilde{G}, \tilde{M}, 0)$. In particular, $I^{cl}_{\tilde{G}, 0}(\mu\times f, \tilde{\lambda})$ and $I^{\dagger}_{\tilde{G}, 0}(\mu\times f, \tilde{\lambda})$ are essentially effective, as in \S5.4. Moreover, we can compute them explicitly and relate them to the distributions $I^{cl}_{G,0}(\iota\times f_G,\lambda)$ and $I^{\dagger}_{G,0}(\iota\times f_G,\lambda)$ for $G$:
\begin{equation*}
\begin{aligned}
  I^{cl}_{\tilde{G},0}(\mu\times f, \tilde{\lambda})
=&meas(K^p)\tilde{\lambda}(\xi(t,t'))tr(\mu\times f | H^*_{cusp}(S_{\tilde{G}}(\tilde{K}_f), \mathbb{V}_{\tilde{\lambda}}^\vee(\mathbb{C})))\\
=&meas(K^p)\lambda(\xi(t))\chi(t')tr(\mu\times f |H^*_{cusp}(S_{\tilde{G}}(\tilde{K}_f), \mathbb{V}_{\tilde{\lambda}}^\vee(\mathbb{C}))\\
=meas(K^p)&\lambda(\xi(t))\chi(t')\chi^\vee(f_1)tr(\mu\times f_G |H^*_{cusp}(S_G(K_f), \mathbb{V}^\vee_{\lambda}))\\
=&\chi(t')\chi^\vee(f_1)I^{cl}_{{G},0}(\mu\times f_G, {\lambda}),
\end{aligned}
\end{equation*}
where the last second equation follows from \S3.1.8 and (\ref{esstwist}), the action of $\mu$ is given as in $(\ref{twistrestrict})$. In particular, if $\chi$ is trivial, 
\begin{equation}
I^{cl}_{\tilde{G},0}(\mu\times f, \tilde{\lambda})=I^{cl}_{G,0}(\iota\times f_G, \lambda).
\end{equation}
Now Corollary (\ref{limit}) implies that 
\begin{equation}
I^{\dagger}_{\tilde{G},0}(\mu\times f, \tilde{\lambda})=\chi(t')\chi^\vee(f_1)I^{\dagger}_{{G},0}(\mu\times f_G, {\lambda}),
\end{equation}
and if $\chi$ is trivial,
\begin{equation}
I^{\dagger}_{\tilde{G},0}(\mu\times f, \tilde{\lambda})=I^{\dagger}_{G,0}(\iota\times f_G, \lambda).
\end{equation}

\vspace{1pc}
\begin{remark}
Given character $\chi$ and $\lambda\in\mathfrak{X}^e$ which is essentially $\iota$-invariant with respect to $\chi$, inspired by the above computation, we can directly define for $f\in\mathcal{H}_p(K^p)$ that 
\begin{equation}
I^{cl, \chi}_{G,0}(\iota\times f, \lambda):=I^{cl}_{G,0}(\mu_{\chi}\times f, \lambda)
\end{equation}
and
\begin{equation}
I^{\dagger, \chi}_{G,0}(\iota\times f, \lambda):=I^{\dagger}_{G,0}(\mu_{\chi}\times f, \lambda)
\end{equation}
where the lower index in $\mu_\chi$ emphsis that the twisted action $\mu$ on the cohomology spaces are defined according to $\chi$. Just like Proposition \ref{induction1}, one can show that only the trace of representations which are essentially $\iota$-invariant with respect to $\chi$ can contribute to these distributions, and the discussion in \S4-6 works for them as well.
\end{remark}

Now let $\tilde{\lambda}=(\lambda, \chi)$ be an arithmetic weight, $\pi$ a $\mu$-invariant finite slope cuspidal representation of $\tilde{G}$, assume that $\pi|_{Gl_{1}}=\chi$. Let $\pi(\chi^{-1})$ be $\pi$ twisting the inverse of $\chi$ by $Gl_1$, and $\pi_\chi$ be the factorization of $\pi(\chi^{-1})$ to $G\cong\tilde{G}/{Gl_1}$, it is easy to see that $\pi_\chi$ is $\iota$-invariant. Compute by definition, we have 
\begin{equation}
\begin{aligned}
m_{EP}^\mu(\pi,\tilde{\lambda})=m_{EP}^\iota(\pi_\chi,\lambda)\neq0.
\end{aligned}
\end{equation}

\vspace{1pc}
Now by Theorem \ref{mainthm}, we have
\begin{thm}\label{mainthm2}
Assume $G=Gl_n$ there is an eigenvariety $\mathfrak{E}^e\subset\mathfrak{E}$ defined as in Theorem 6.5.1, such that
\begin{itemize}
\item[(a)] there are two projections $p_1: \mathfrak{E}^e\rightarrow\mathfrak{X}^e$ and $p_2: \mathfrak{E}^e\rightarrow\mathfrak{B}^1$, such that $y=(\lambda_y,\theta_y)\in\mathfrak{E}^e(\bar{\mathbb{Q}}_p)$ if and only if $\theta$ is a finite slope overconvergent Heche eigensystem of weight $\lambda_y=p_1(y)$ and is essentially self-dual with respect to $\chi_y=p_2(y)$ with $m^{\dagger,\iota}_{G,0}((\theta\times\chi_y)_{\chi_y},(\lambda\times\chi_y)_{\chi_y})\neq0$ .
\item[(b)] $\mathfrak{E}^e$ is equidimensional of dimension $[\frac{n}{2}]+1$
\item[(c)] For any $\chi\in\mathfrak{B}^1$, set $\mathfrak{E}^e_\chi=p_2^{-1}(\chi)$, then $\mathfrak{E}^e_\chi$ is the eigenvariety parameterizing essentially self-dual Hecke eigensystems of $G$  with respect to $\chi$. In particular, $\mathfrak{E}^e_0=\mathfrak{E}^\iota$.
\end{itemize}
\end{thm}

\vspace{1pc}
\begin{remark}\ {}
\begin{itemize}
\item[$(a)$] As Remark 7.1.3, $\mathfrak{E}^e(\overline{\mathbb{Q}}_p)$ contains all essentially self-dual finite slope cuspidal Hecke eigensysems of $G$.
\item[$(b)$] Applying our theory to $I^{\dagger, \chi}_{G,0}(\iota\times f, \lambda)$, we can obtain $\mathfrak{E}^e_\chi$ directly.
\end{itemize}
\end{remark}

\subsection{Ash-Pollack-Stevens Conjecture}
Let $\theta_0$ be a classical finite slope cuspidal Hecke eigensystem of $Gl_n$ of regular weight $\lambda_0$. $\theta_0$ is called $p$-adic arithmetic rigid, if it is not contained in any arithmetic irreducible component of $\mathfrak{E}_{K^p}$. Conjecture \ref{APS} claims that, if $\theta$ is not essentially self-dual then it is $p$-adic arithmetic rigid. Theorem \ref{mainthm2} gives its inverse:

\begin{cor}
Assume $\theta_0$ is essentially self-dual, then it is not $p$-adic arithmetic rigid, lying in an arithmetic component of $\mathfrak{E}^e$.
\end{cor}
\begin{proof}
By Theorem \ref{mainthm2}, $(\lambda_0,\theta_0)\in\mathfrak{E}^e(\overline{\mathbb{Q}}_p)$. Consider the subset $\Sigma$ of $\mathfrak{E}^e(\bar{\mathbb{Q}}_p)$ consisting of $(\lambda,\theta)$ such that $\lambda$ is arithmetic and $\theta$ is non-critical with respect to $\lambda$. $\Sigma$ is Zariski dense in $\mathfrak{E}^e(\bar{\mathbb{Q}}_p)$ since its projection to $\mathfrak{X}^e$ contains an arithmetic point $\lambda$. By Corollary \ref{arithmetic}, those points in $\Sigma$ are classical and corresponding to cuspidal Hecke eigensystems.
\end{proof}

\vspace{1pc}
The next theorem shows that the smooth hyperthesis on arithmetic points of an eigenvarity may give some hint on the Ash-Pollack-Stevens conjecture.

\begin{thm}
Assume that every arithmetic point in the eigenvariety is smooth. Assume $\theta_0$ is not $p$-adic arithmetic rigid, and its arithmetic component contains one arithmetic, essentially self-dual Hecke eigensystem, then this arithmetic component contains a Zariski dense subset of essentially self-dual Hecke eigensystems.
\end{thm}
\begin{proof}
By \cite{APG}, $(\lambda_0,\theta_0)$ is contained in an arithmetic component of dimension $[\frac{n}{2}]+1$ over $\mathfrak{X}^e$. By our assumption, this component intersects with $\mathfrak{E}^e$ at some smooth point and $\mathfrak{E}^e$ is also of dimension $[\frac{n}{2}]+1$. So the arithmetic component contains an irreducible component of $\mathfrak{E}^e$.
\end{proof}

\begin{remark}
Assume $G=Gl_3$. The theorem 7.3.2 assumes that there is an essentially self-dual point in the arithmetic component. This is not surprising if the eigenvarieties have good geometry. By \cite{Xiang}, we know the full eigenvariety $\mathfrak{E}$ has dimension $\leqslant3$. Let $\mathfrak{A}$ an arithmetic component, it also projects onto $\mathfrak{X}^e$ and is of dimension at least $1$. Since we know that $\mathfrak{E}^e$ has dimension $2$, it should meet $\mathfrak{A}$.
\end{remark}

%%%%%%%%%%%%%%%%%

%%--------------------Here the manuscript ends--------------------------------

\begin{thebibliography}{}

\bibitem{Arthur}
J. Arthur \textit{The $L^2$-Lefschetz numbers of Hecke operators}, Invent. Math. 97 (1981) 257-290

\bibitem{AC}
J. Arthur and L. Clozel \textit{Simple algebras, base change and the advanced theory of the trace formula}, AMS 120, Princeton University Press 1989

\bibitem{APG}
A. Ash, D. Pollack and G. Stevens \textit{Rigidity of p-adic cohomology classes of congruence subgroups of $Gl(n, \mathbb{Z})$}, Proc. London Math. Soc. (3) 96 (2008) 367-388

\bibitem{Ash-Stevens}
A.Ash and G. Stevens \textit{p-adic deformations of arithmetic
cohomolgy}, priprint 2008

\bibitem{BS}
D. Barbasch and B. Speh \textit{Cuspidal representations of reductive groups}, preprint 2008

\bibitem{Borel}
A. Borel, \textit{Admissible representations of a semi-simple group over a local field with vectors fixed under an Iwahori subgroup}, Invent. Math. 35 (1976) 233-259

\bibitem{BLS}
A. Borel, J-P. Labesse, J. Schwermer \textit{On the cuspidal cohomology of S-arithmetic subgroups of reductive groups over number fields}, Compositio Math. 102 (1996), No. 1, 1-40

\bibitem{BW}
A. Borel, N. Wallach \textit{Continuous cohomology, discrete subgroups, and representations of reductive groups}, MSM 67, 1999

\bibitem{BGR}
S. Bosch, U. Guntzer, and R. Remmert \textit{Non-Archimedean
analysis}, volume 261 of Grundlehren der Mathematischen
Wissenschaften [Fundamental Principles of Mathematical Sciences],
Springer-Verlag, Berlin, 1984.

\bibitem{Buzzard}
K.Buzzard \textit{On p-adic families of automorphic forms}, Modular
curves and abelian varieties, 23-44, Progr. Math., 224, Birkhuser,
Basel, 2004

\bibitem{Buzzard eigenvarieties}
K.Buzzard \textit{Eigenvarieties}, L-functions and Galois
representations, 59-120, London Math. Soc. Lecture Note Ser., 320,
Cambridge Univ. Press, Cambridge, 2007

\bibitem{CM}
F. Calegari and B. Mazur \textit{Nearly ordinary galois deformations over arbitrary number
fields}, Journal of the Inst. of Math. Jussieu (2009) 8(1), 99-177

\bibitem{Chenevier}
G. Chenevier \textit{Families p-adiques de formes automorphes pour
$GL_n$}, J. Reine Angew. Math. 570 (2004), 143-217

\bibitem{Clozel}
L. Clozel \textit{Reprsentations galoisiennes associes aux reprsentations automorphes de $Gl(n)$ }, Publication math. de l'I.H..S, 73(1991) 97-145

\bibitem{Coleman}
R.Coleman \textit{p-adic Banach spaces and families of modular
forms}, Invent. Math. 127(1997), no. 3, 417-479

\bibitem{Coleman-Mazur}
R.Coleman and B.Mazur \textit{The eigencurve}, Galois
representations in arithmetic algebraic geometry, 1-113, London
Math. Soc. Lecture note Ser., 254, Cambridge Univ. Press, Cambridge,
1998

\bibitem{Eisenbud}
D. Eisenbud \textit{Commutative algrbra}, GTM 150, Springer, 1994

\bibitem{Franke}
J. Franke \textit{Harmonic analysis in weighted $L_2$-spaces}, Ann. Sci. Ecole. Norm. Sup (4), 31 (1998), no. 2, 181-279

\bibitem{FS}
J. Franke and J. Schwermer \textit{A decomposition of spaces of automorphic forms and the Eisenstein cohomology of arithmetic groups}, Math. Ann. 331, 1998

\bibitem{Jantzen}
J. Jantzen \textit{Representations of algebraic groups}, SURV 107, AMS, 2003

\bibitem{LS}
J. Li and J. Schwermer \textit{On the Eisenstein cohomology of arithmetic groups}, Duke. Math. Vol123, No.1, 2004

\bibitem{Milne}
J. Milne \textit{Algebraic groups, Lie Groups, and their arithmetic subgroups}, available at \textit{www. jmilne.org/math/}, 2010

\bibitem{MW}
C. Moeglin and J.-L. Waldspurger \textit{Spectral decomposition and Eisenstein series}, Cambridge Tracts in Mathematics 113, Cambridge University Press, 1995

\bibitem{RS}
J. Rohlfs and B. Speh \textit{Automorphic representations and Lefschetz numbers}, Annales scientifique de l'E.N.S, 22 (1989) No. 3, 473-499

\bibitem{Serre}
J.-P. Serre \textit{Endomorphismes completement continus des espaces
de Banach p-adiques}, Publ. Math. IHES, vol. 12, 1962

\bibitem{Speh}
B. Speh \textit{Unitary representations of $Gl(n,\mathbb{R})$ with non-trivial $(\mathfrak{g}, K)$-cohomology}, Invent. Math. 71 (1983), no. 3, 443-465

\bibitem{Springer}
T.A. Springer \textit{Reductive groups}, in Proc. Sympos. Pure Math., Vol 33, American Mathematical Society, Providence, RI, 1979

\bibitem{Urban}
E.Urban \textit{Eigenvariety}, Annals of Math. Volume 174 (2011), 1685-1784

\bibitem{VZ}
D. Vogan and G. Zuckerman \textit{unitary representations with nonzero cohomology}, Compositio Math. 53 (1984), No. 1, 51-90

\bibitem{Xiang}
Z. Xiang \textit{A construction of the full eigenvariety of a reductive group}, Journal of Number Theory 132 (2012), 938-952

\bibitem{Xiang2}
Z. Xinag \textit{A twisted $p$-adic trace formula}, in preparation


\end{thebibliography}
\end{document}